\documentclass[leqno,ssecnum]{rims27}
\makeatletter
\renewcommand{\subsection}{\@startsection
{subsection}%
{2}%
{0mm}%
{-\baselineskip}
{0.5\baselineskip}%
{\normalfont\normalsize\bfseries}}%
\makeatother
\usepackage{amsmath}
\usepackage{amssymb,amscd}
\usepackage{latexsym}
\usepackage[mathscr]{eucal}
\theoremstyle{definition}

\theoremstyle{remark}
\newtheorem*{rem}{Remark}
\theoremstyle{definition}
\newtheorem*{definition}{Definition}

\newtheorem*{fact}{Fact}
\newtheorem*{example}{Example}
\theoremstyle{remark}
\newtheorem*{remark}{Remark}
\newtheorem*{note}{Note}
\theoremstyle{plain}
\newtheorem*{assertion}{Assertion}
\newtheorem*{lemma}{Lemma}

\newtheorem*{cor}{Corollary}
\newtheorem*{thm}{Theorem}
\newcommand{\Conf}{\operatorname{Conf}}
\newcommand{\Hom}{\operatorname{Hom}}
\newcommand{\Aut}{\operatorname{Aut}}
\newcommand{\aug}{\operatorname{aug}}
\newcommand{\id}{\operatorname{id}}
\newcommand{\Emb}{\operatorname{Emb}}
\newcommand{\rank}{\operatorname{rank}}

\def\Z{{\mathbb{Z}}}
\def\R{{\mathbb{R}}}
\def\C{{\mathbb{C}}}
\def\Q{{\mathbb{Q}}}
\def\SS{{\mathbb{S}}}
\def\T{{\mathbb{T}}}
\def\A{{\mathbb{A}}}

\def\LL{{\mathcal L}}

\def\al{{\alpha}}

\newdimen\argwidth
\def\db[#1\db]{%
 \setbox0=\hbox{$#1$}\argwidth=\wd0
 \setbox0=\hbox{$\left[\box0\right]$}
    \advance\argwidth by -\wd0
 \left[\kern.3\argwidth\box0 \kern.3\argwidth\right]}
\title{Limit Elements in the Configuration\\
Algebra for a Cancellative Monoid      
\footnote{  
\quad The present paper is a complete version of the 
announcement \cite{Sa2} based on the preprint RIMS-726.
We rewrote the introduction, left out the 
filtration by $(p,q)$,
divided section \S10 into \S10 and 11, 
and updated the references.
The \S11, 12 are newly written, where, applying the results in \S2-10 to Cayley graphs $(\Gamma,G)$ of a cancellative monoid, we introduce a fibration $\Omega(\Gamma,G)\to\Omega(P_{\Gamma,G})$ of our interest.
In Winter semester 05-06 at RIMS, 
the author held a series of seminars on the present paper.  
He thanks to its participants  Yohei Komori, Michihiko Fujii, Yasushi Yamashita, 
Masahiko Yoshinaga, Takefumi Kondo, and   
Makoto Fuchiwaki.
Particular thanks go to Yohei Komori, without whose encouredgment, this paper would not have appeared. 
The author is also grateful to Brian Forbes and Ken Shackleton for the careful reading
of the manuscript.
}}
\TitleHead{Limits in the Configuration Algebra}         
\dedicatory{Dedicated to Professor Heisuke Hironaka\\
on the occasion of his seventy-seventh birthday}
\author{Kyoji \textsc{Saito}
      \footnote{
RIMS, Kyoto University, Kyoto 606-8502, Japan.}}
\AuthorHead{Kyoji Saito}         
 \communication{
 }      
\begin{document}
%
\maketitle

\vspace{-0.5cm}
\noindent
{\footnotesize {\bf Abstract.} We introduce two spaces $\Omega(\Gamma,G)$ and $\Omega(P_{\Gamma,G})$ of pre-partition functions and of opposite series, respectively, which are associated with a Cayley graph $(\Gamma,G)$ of a cancellative monoid $\Gamma$ with a finite generating system $G$ and with its growth function $P_{\Gamma,G}(t)$. Under mild assumptions on $(\Gamma,G)$, we introduce a fibration $\pi_\Omega\!:\!\Omega(\Gamma,G)\!\to\! \Omega(P_{\Gamma,G})$ equivariant with a $\Z_{\ge0}$-action, which is transitive if it is of finite order. Then, the sum of  pre-partition 
functions in a fiber 
is a linear combination of residues of the proportion of two growth functions $P_{\Gamma,G}(t)$ and $P_{\Gamma,G}\mathcal{M}(t)$ attached to $(\Gamma,G)$ at the places of poles on the circle of the convergent radius.
}

\vspace{0.1cm}

\tableofcontents
\vspace{-.6cm}

\section{Introduction}\label{sec:1} 

 \vspace{-0.08cm}
Replacing the square lattice $\Z^2$ in the classical 
Ising model (\cite{Gi}\cite{I}\cite{O}\cite{Ba}) by a Cayley graph $(\Gamma,G)$ of a cancellative monoid $\Gamma$ with a finite generating system $G$, we introduce the space  $\Omega(\Gamma,G)$ of {\it pre-partition functions}. 
Here, the word pre-partition function is used only in the present introduction for a reason we explain now. 
Namely, for any finite region $T$ of the Cayley graph, we define the {\it free energy} $\frac{\mathcal{M}(T)}{\#(T)}$ by the logarithm of the sum of configurations in $T$ ((5.1.5) and (6.1.1)), and then consider  accumulating points set $\Omega(\Gamma,G)$ (in a suitable toplogical setting) of the sequence $\{\frac{\mathcal{M}(\Gamma_n)}{\#(\Gamma_n)}\}_{n\!\in\!\Z_{\ge0}}$ of free energies of balls $\Gamma_n$ of radius $n$ in $(\Gamma,G)$ (\S11.1 Definition). 
In the case of $\Gamma\!=\!\Z^2$, $\Omega(\Gamma,G)$ consists of a single element. By inputting the data of Boltzmann weights to it, we get the partition function: an elliptic function dependent on the parameters involved in the Boltzmann weights. 
This fact, inspired the author {\it to use the pre-partition functions 
to construct functions
on the moduli of} $\Gamma$ ([Sa1,3]).

In our new setting, the space $\Omega(\Gamma,G)$ is no longer a single element set but is a compact Hausdorff space. Under mild assumptions on $(\Gamma,G)$, we construct a fibration $\pi_\Omega:\Omega(\Gamma,G)\!\to\!\Omega(P_{\Gamma,G})$ (11.2.12), where $\Omega(P_{\Gamma,G})$ is another newly introduced compact space, consisting of opposite sequences of the growth function $P_{\Gamma,G}(t)$ (11.2.3). The fibration is equivariant with an action $\tau_\Omega$. If the action is of finite order, then it is transitive and the sum of the partition functions in a fiber of $\pi_\Omega$ is given by a linear combination of the proportions of residues of the two series $P_{\Gamma,G}\mathcal{M}(t):=\sum_{n=0}^\infty\mathcal{M}(\Gamma_n)t^n$ and $P_{\Gamma,G}(t):=\sum_{n=0}^\infty\#(\Gamma_n)t^n$ at their poles on the circle $|t|=r_{\Gamma,G}$ of convergent radius $r_{\Gamma,G}$. We publish these results in the present paper, eventhough our original attempt is not achieved. 

The paper is divided into two parts. 
In the first part \S2-10, we develop a general frame work on a topological Hopf algebra $\R\db[\Conf\db]$ called the {\it configuration algebra}, where necessary concepts such as the configuration sums, the free energies (called equally dividing points) etc. are introduced. The algebra is equipped with two (one adic and the other classical) topologies in order to discuss carefully limit process in it. 
Then, inside its subspace $\LL_{\R,\infty}$ of Lie-like elements at infinity, the set $\Omega_\infty:= \overline{\log(\text{EDP})}_{\infty}$ of all accumulation points of all free energies is introduced. 
In the second half, \S11-12, we consider a Cayley graph $(\Gamma,G)$ of a monoid. Then, the set of pre-partition functions $\Omega(\Gamma,G)$ is defined as the subset of $\Omega_\infty$ of all accumulating points of the sequence of free energies of the balls $\Gamma_n$ of radius $n\! \in\!\Z_{\ge0}$ in $(\Gamma,G)$. We also introduce another limit set $\Omega(P_{\Gamma,G})\!\subset\! \R\db[s\db]$, called {\it the space of opposite sequences}, depending only on the Poincare series $P_{\Gamma,G}(t)$  of $(\Gamma,G)$ (see (11.2.1-4) and (11.2.6)).
The space $\Omega(P_{\Gamma,G})$ is the key to relate the space $\Omega(\Gamma,G)$  with the singularities of the Poincare series  $P_{\Gamma,G}(t)$ on the circle $|t|=r_{\Gamma,G}$ of convergence (\S11 Theorems 1-5).
 Then, comparing the two limit spaces  $\Omega(\Gamma,G)$ and $\Omega(P_{\Gamma,G})$, we arrive at the goal: a residual presentation of pre-partition functions (\S11 Theorem 6.)



Let us explain the contents of the present paper in more detail. 

The isomorphism class of a 
colored oriented finite graph is called a {\it configuration}. 
The set of all configurations with fixed bounds of valency and colors, denoted by $\Conf$, has an additive monoid structure generated by $\Conf_0$, isomorphism classes of connected graphs (by taking the disjoint 
union as the product) and a partial ordering structure (\S2.3). 
In \S2.4, we introduce the basic invariant
$\left(\begin{smallmatrix}S \\ S_{1},\ldots.S_{m}\end{smallmatrix}\right)\in \Z_{\ge0}$ for $S_{1},\ldots,S_{m}$ and $S \in \Conf$,
called a {\it covering coefficient}.
We denote by  $\A\db[\Conf\db]$ the completion of the semigroup ring  
$\A\cdot \Conf$ with respect to the grading $\deg(S)\!\!:=\!\!\#(S)$, called the {\it configuration algebra}  (\S3), where $\A$ is the ring of coefficients. The algebra $\A\db[\Conf\db]$ carries a topological Hopf algebra structure by taking the covering coefficients as structure constants (\S4).

For a configuration $S\!\in\! \Conf$, let $\mathcal{A}(S)\!\in\! \A\cdot\Conf$ be the  sum of all its subgraphs, the configuration sum. We put 
$\mathcal{M}(S)\!:=\!\log(\mathcal{A}(S))$, then 
$\{\mathcal{M}(S)\}_{S\in\Conf_0}$ forms a basis of the Lie-like space of the non\!-\!complete bi-algebra $\A\!\cdot\!\Conf$ (\S5\! and\! 6). 
However, 
this is not a topological basis of the 
Lie-like space $\LL_{\A}$ 
of the algebra $\A\db[\Conf\db]$. 
\vspace{-0.04cm}
Therefore, we introduce  a topological basis, 
denoted by $\{\varphi(S)\}_{S\in \Conf_0}$, of  $\LL_{\A}$. 
\vspace{-0.04cm}
The coefficients of the transformation matrix between $\{\mathcal{M}(S)\}_{S\in \Conf_0}$ and 
$\{\varphi(S)\}_{S\in \Conf_0}$
are described by {\it kabi-coefficients},
introduced in \S7. The base-change induces
a linear map, called the {\it kabi-map}, from $\LL_{\A}$ 
to a formal  module spanned by $\{\mathcal{M}(S)\}_{S\in \Conf_0}$. 
The kernel of the kabi-map is 
denoted by $\LL_{\A,\infty}$ and is called the
{\it Lie-like space 
at infinity} (\S8).

The group-like elements $\frak{G}_{\mathbb{Z}, finite}$ of the configuration algebra $\Z\db[\Conf\db]$ is isomorphic to the fractional group of the monoid $\Conf$ by the correspondence $\mathcal{A}(S)\leftrightarrow S$ (\S9). Thus, it contains a positive cone spanned by $\Conf$.  
%
We are interested in the 
{\it equal division points} $\mathcal{A}(S)^{1/\#S}$  ($S\!\in\! \Conf$)
of the lattice points in the positive cone, 
and the set $\overline{\text{EDP}}$ of
{\it their accumulation points with respect to the classical topology} 
by specializing the coefficient $\A$ to $\R$. 
In \S10, by taking their logarithms\footnote
{
The logarithm $\log(\mathcal{A}(S)^{1/\#S})=\mathcal{M}(S)/\#S$, which we call the logarithmic equal dividing point, is called the (Helmholz) free energy in statistical mechanics (\cite{Gi}\cite{I}\cite{O}\cite{Ba}).
}
, we define their accumulation set 
$\Omega\!:=\!\overline{\log(\text{EDP})}$ in $\LL_{\R}$.
The set $\Omega$ decomposes into a join of the 
infinite simplex spanned by the 
vertices $\frac{\mathcal{M}(S)}{\#S}$ 
for $S\!\in\!\Conf_0$ 
and  
a compact subset 
$\Omega_\infty\!:=\!\overline{\log(\text{EDP})}_{\infty}$ 
of $\LL_{\R,\infty}$ (\S10).

From \S11, we fix a monoid $\Gamma$ with a finite generating system $G$. 
The sequence 
of the logarithmic equal division 
points $\frac{\mathcal{M}(\Gamma_n)}{\#\Gamma_n}$ 
for the sequence of balls $\Gamma_n$ of radius 
$n\!\in\!\Z_{\ge0}$ in the Cayley graph 
accumulates to a compact subset $\Omega(\Gamma,G)$ of $\Omega_\infty$, called the 
{\it space of limit elements} for $(\Gamma,G)$. 
This is the main object of interest of the present article.
If $\Gamma$ is a group of polynomial growth, then due to  
results of Gromov \cite{Gr1} and Pansu \cite{P}, 
for any generating system $G$, $(\Gamma,G)$ is simple
accumulating, i.e.\ $\#(\Omega(\Gamma,G))\!=\!1$.

In order to study the multi-accumulating cases, we introduce 
in \S11 (11.2.3) another compact accumulating set $\Omega(P_{\Gamma,G})$: the {\it space of opposite series of} 
the growth series 
$P_{\Gamma,G}(t)\!:=\!\sum_{n=0}^\infty \#(\Gamma_n)t^n$. 
Under mild {\bf Assumptions 1} in \S11.1 and {\bf 2} in \S11.2 on $(\Gamma,G)$, 
we show that a natural proper surjective map $\pi_\Omega: \Omega(\Gamma,G)\to \Omega(P_{\Gamma,G})$ equivariant with an action of $\Z_{\ge0}$ exists (see 11.2 {\bf Theorems 1,2,3} and {\bf 4.}), where i) $\pi_\Omega$ is a forgetful map which remembers only the portion $\underset{n\to \infty}{\lim}\!\frac{A(\Gamma_{n-k},\Gamma_n)}{\#\Gamma_n}$ ($k\in\Z_{\ge0}$) of the limit elements
\vspace{-0.05cm}
(here $A(\Gamma_{n-k},\Gamma_n)\!\!:=\!\!\#$ of subgraphs of $\Gamma_n$ isomorphic to $\Gamma_{n\!-\!k}$) and ii) the $\Z_{\ge0}$-action on $\Omega(\Gamma,G)$  is generated 
\vspace{-0.07cm}
by the map $\tilde\tau_\Omega$: {\it the limit element $\underset{m\to\infty}\lim\frac{\mathcal{M}(\Gamma_{n_m})}{\#\Gamma_{n_m}}$ $\mapsto$  
\vspace{-0.07cm}
the limit element 
\vspace{-0.07cm}
 $\underset{m\to\infty}{\lim}\frac{\mathcal{M}(\Gamma_{n_m-1})}{\#\Gamma_{n_m-1}}$.} 
This $\tilde\tau_\Omega$-action, up to an initial constant factor, has an interpretation by the fattening action on $\Conf$: $S\!\mapsto\! S\Gamma_1$ (here, $S\Gamma_1$ is the equivalence class of $\mathbb{S}\cdot\Gamma_1$ for a representative $\mathbb{S}$ of $S$) in the level of $\Omega(\Gamma,G)$, and interpretation by the degree shift action $t^n\!\mapsto\! t^{n+1}$ in the level of $\Omega(P_{\Gamma,G})$.

Subsections 11.3 and 11.4 are devoted to the study of the space of opposite sequences $\Omega(P)$ for a power series $P(t)$ (11.2.1) in general with a constraint on the growth of coefficients.
\vspace{-0.09cm}
The main concern is to clarify a certain duality between the set $\Omega(P)$ and the set of singularities of P(t) on the boundary of the disc of radius $r$ of convergence. It asks intricate analysis, and, in the present paper, we have clarified only when $\Omega(P)$ is a finite set. Actually, if  $\Omega(P)$ is finite, then the $\Z_{\ge0}$-action becomes a cyclic $\Z/h\Z$ and simple-transitive action. We explicitly determine  $\Omega(P)$ as a set of rational functions in the opposit variable $s$ of $t$ (i.e.\ $st\!=\!1$).
In particular, their common denominator $\Delta_P^{op}(s)$, which is a factor of $1-(rs)^h$, has the degree equals to the rank of the space $\R\Omega(P)$ spanned by $\Omega(P)$. If, further, $P(t)$ is meromorphic in a neighborhood of the convergent disc, then the top part $\Delta_P^{top}(t)$  of the denominator of $P(t)$ on the convergent circle of radius $r$ (see (11.4.1)) and the opposite denominator $\Delta^{op}_P(s)$ are related by the opposite transformation $st=1$ (11.4 {\bf Theorem 5}).

If $\Omega(\Gamma,G)$ is finite, then again 
the $\Z_{\ge0}$-action on $\Omega(\Gamma,G)$ 
becomes cyclic $\Z/\tilde h_{\Gamma,G}\Z$ and simple-transitive action for some $\tilde h_{\Gamma,G}\!\in\! \Z_{>0}$ such that $h_{\Gamma,G}| \tilde h_{\Gamma,G}$ for the period $h_{\Gamma,G}$ of the $\Z_{\ge0}$-action on $\Omega(P_{\Gamma,G})$. Therefore, the map $\pi_\Omega$ is equivalent to the Galois covering map $\Z/\tilde h_{\Gamma,G}\Z\to\!> \Z/h_{\Gamma,G}\Z$. 
Let us call the kernel of the homomorphism the {\it inertia group} 
and the sum of elements in $\Omega(\Gamma,G)$  of an orbit
of the inertia group a {\it trace element}.
As the goal of the present paper, 
we express the trace elements  
as linear combinations of  the proportions
$\frac{P_{\Gamma,G}\mathcal{M}(t)}{P_{\Gamma,G}(t)}\bigr|_{t=x}$
of the residues of meromorphic functions 
$P_{\Gamma,G}\mathcal{M}(t):=\sum_{n=0}^\infty\mathcal{M}(\Gamma_n)t^n$ 
and  $P_{\Gamma,G}(t)$ at the places $x$ in the root of $\Delta_P^{top}(t)=0$. 
(11.5 {\bf Theorem 6}). In the proof, we  essentially uses the duality theory in 11.4.

\smallskip
Finally in \S12, we give a few concluding remarks. 
Since we are only at the starting of the study of the limit space $\Omega(\Gamma,G)$,
the questions are scattered in various directions of general nature or of specific nature. 

As an immediate generalization of our goal Theorem 6 for the cases when 
$\Omega(\Gamma,G)$ is not finite, in Problem 1.1, we ask {\it measure theoretic approach for the duality between $\Omega(P)$ and $Sing(P)$}, and give a conjectural formula. 

Another important generalization of Theorem 6 is the {\it globalization} in the following sense: in many important examples, the growth function $P_{\Gamma,G}(t)$ analytically extends to a meromorphic function in a covering regions of $\C$ (and same for $P_{\Gamma,G}\mathcal{M}(t)$). Let $x$ be a pole of order $d$ of such a meromorphic function, then $\left(\frac{d^i}{dx^i}\frac{P_{\Gamma,G}\mathcal{M}(t)}{P_{\Gamma,G}(t)}\right)\big|_{t=x}$ for $0\le i\!<\!d$ (which we call a {\it residue} of depth $i$ at $x$) belongs to $\mathcal{L}_{\C,\infty}$ (even though it is no longer a limit element). Theorem 6 treats only the extremal case $|x|\!=\!r$ and $i\!=\!0$. Therefore, we ask {\it to study  all residues at all possible poles together with a possible action of a Galois group}, in particular, {\it to clarify the meaning of the (higher) residues at} $x\!=\!1$.

We conjecture that 
hyperbolic groups and some groups of geometric significance (surface groups, 
mapping class groups and Artin groups for suitable choices of generators) are finite accumulating, 
i.e.\ $\#(\Omega(\Gamma,G))\!<\!\infty$. 


\section{Colored graphs and covering coefficients}\label{sec:2}

An isomorphism  class of finite graphs  
with a fixed color-set and a bounded number 
of edges (valency) at each vertex is called a {\it configuration}.
The set of all configurations carries the structure of an
 abelian monoid with a partial ordering.
The goal of the present section is to introduce 
a numerical invariant, 
called the {\it covering coefficient}, 
and to show some of its basic properties.

\subsection{Colored Graphs.}\label{subsec:2.1}
We first give a definition of colored graph which is used in the present paper.
\begin{definition} 
1. 
A pair $(\Gamma,B)$ is called a {\it graph},
if $\Gamma$ is a set and $B$ is a subset of
$\Gamma\!\times\!\Gamma\!\setminus\!\Delta$ with
$\sigma (B)\!=\!B$, where $\sigma$ is the involution
$\sigma(\alpha,\beta)\!:=\!(\beta,\alpha)$ and $\Delta$
is the diagonal subset.
An element of $\Gamma$ is called a
{\it vertex} and a $\sigma$-orbit in $B$ is called an {\it edge}. 
A graph is called {\it finite} if $\sharp\Gamma\!\! <\!\infty$. 
We sometimes denote a graph by $\Gamma$ and 
the set of its vertices by $|\Gamma|$.

2.
Two graphs are {\it isomorphic} if there is a bijection of 
vertices inducing a bijection of edges.
Any subset $\mathbb{S}$ of $|\Gamma|$ carries  
a graph structure  
by taking $B\cap(\mathbb{S}\times \mathbb{S})$ as the set of edges for $\mathbb{S}$. 
The set $\SS$ equipped with this graph structure is called 
a {\it subgraph} (or a {\it full subgraph}) of $\Gamma$ and is denoted by the 
same $\SS$.
In the present paper, the word ``subgraph'' shall be used only in this sense,
and the notation $\mathbb{S}\subset \Gamma$ shall mean also that $\mathbb{S}$ 
is a subgraph of $\Gamma$ associated to the subset.
Hence, we have the bijection:
$\{\text{subgraphs of }\Gamma\}\  \simeq\  \{\text{subsets of }|\Gamma|\}.$

3. A pair $(G,\sigma_G)$ of a set $G$ and an involution $\sigma_{G}$ 
on $G$ (i.e.\ a map $\sigma_G\!:\!G\!\to\! G$ with $\sigma_G^2\!=\!id_G$) 
is called a {\it color set}. 
For a graph $(\Gamma,B)$, a map $c:B\!\rightarrow\! G$ is called 
a $(G,\sigma_G)$-{\it coloring }, or $G$-{\it coloring}, 
if $c$ is equivariant with respect to involutions: 
$c\circ\sigma\!=\!\sigma_{G}\circ c$. 
The pair consisting of a graph and a $G$-coloring  is called a {\it $G$-colored graph}.
Two $G$-colored graphs are called {\it $G$-isomorphic} 
if there is an isomorphism 
of the graphs compatible with the colorings.
Subgraphs of a $G$-colored graph are naturally $G$-colored.

If all points of $G$ are fixed by $\sigma_{G}$,
then the graph is called un-oriented.
If $G$ consists of one orbit of $\sigma_{G}$,
then the graph is called un-colored.

The isomorphism class of a $G$-colored graph $\mathbb{S}$ is
denoted by $[\mathbb{S}]$. 
Sometimes we will write $\mathbb{S}$ instead of $[\mathbb{S}]$
(for instance, we put $\sharp[\mathbb{S}]:=\sharp\mathbb{S}$, and call 
$[\mathbb{S}]$ {\it connected} if $\SS$ is topologically connected as a simplicial complex). 
\end{definition}
%
\noindent
{\bf Example.}
(Colored Cayley graph of a monoid with cancellation conditions). \ 
Let $\Gamma$ be a monoid satisfying the left and right cancellation conditions:
if $axb\!=\!ayb$ in $\Gamma$ for $a,b,x,y\in\Gamma$ then $x\!=\!y$ in $\Gamma$. 
In the other words, for any two elements $a,b\in \Gamma$, if there exists $g\in\Gamma$ such that $a=bg$ (resp.\  $a=gb$) then $g$ is uniquely determined from $a$ and $b$, which we shall denote by $b^{-1}a$ (resp.\ $ab^{-1}$).
Let $G$ be a finite generating system of $\Gamma$ 
with $e \notin G$. Then, we equip $\Gamma$ with a graph structure by taking 
$B:=\{(\al,\beta)\in\Gamma\times\Gamma :\al^{-1}\beta \text{ or }\beta^{-1}\al \in G\}$ as the 
set of edges.  Due to the left cancellation condition, it becomes a colored graph by taking $G\cup G^{-1}$as the color set and by putting $c(\al,\beta)\!=\!\al^{-1}\beta$ for $(\al,\beta)\in B$, where $G^{-1}$ is a formally defined set consisting of elements of symbols $\al^{-1}$ for $\al\in G$ and identifying $\al^{-1}$ with $\beta\!\in\! G$ if $\al\beta\!=\!e$ in $\Gamma$ (such $\beta$ may not always exists). Due to the right cancellation condition, for any vertex $x$ and any $\al\!\in\! G$, vertices connected with $x$ by the edges of color $\al$ (i.e.\ $y\!\in\!\Gamma$ s.t.\ $y\al\!=\!x$) is unique. 
Let us call the graph, denoted by $(\Gamma,G)$ or $\Gamma$ for simplicity, 
the {\it colored Cayley graph} of the monoid $\Gamma$ 
with respect to the generating system $G$. 
The left action of $g\!\in\! \Gamma$ on $\Gamma$ is a color preserving graph embedding map from $(\Gamma,G)$ to itself. 

If $G=G^{-1}$, then $\Gamma$ is a group and the above definition coincides with the usual definition of a Cayley graph of a group.

\subsection{Configuration.}\label{subsec:2.2}

For the remainder of the paper, 
{\it we fix a finite color set $(G,\sigma_G)$ (i.e.\ $\#G<\infty$) 
and a non-negative integer $q\!\in\!\Z_{\ge0}$, and  
consider only the $G$-colored graphs such that 
 the number of edges ending at any vertex (called valency) is at 
most $q$}. 
The isomorphism class $[\mathbb{S}]$ of such a graph
$\SS$ is  called a $(G,q)$-{\it configuration} (or, a {\it configuration}).
The set of all (connected) configurations is defined by:
\begin{eqnarray}
\label{eq:2.2.1}
\qquad\Conf& := & \{ 
\text{$G$-isomorphism classes of $G$-colored graphs such that the } \\
&&  \ \text{number of edges ending at any given vertex is at 
most $q$} \} \nonumber\\
\label{eq:2.2.2}
\qquad \Conf_{0}\! & : = &  \{S\in\Conf \mid S
 \ \text{is connected}\}.
\end{eqnarray}
The isomorphism class $[\emptyset]$ of an empty graph is contained 
in $\Conf$ but not in $\Conf_{0}$.
Sometimes it is convenient to exclude $[\emptyset]$ from $\Conf$.
So put:
\begin{equation}\label{eq:2.2.3}
\Conf_{+}\ :\ =\ \  \Conf\setminus\{[\emptyset]\}.
\end{equation}
\begin{remark} 
To be exact, the set of configurations \eqref{eq:2.2.1} 
should have been denoted by $\Conf^{G,q}$.
If there is a map $G\to G'$ between two color sets compatible 
with their involutions and an inequality $q\le q'$, 
then there is a natural map 
$\Conf^{G,q}\to \Conf^{G',q'}$.
Thus, for any inductive system $(G_n,q_n)_{n\in\Z_{>0}}$ (i.e.\
$G_n\!\to\! G_{n+1}$ and $q_n\!\le\! q_{n+1}$ for $n$), 
we get the 
inductive limit  $\underset{n\to\infty}{\lim} \Conf^{G_n,q_n}$.
In [S2], we used such limit set.
However, in this paper, we fix $G$ and $q$,
since  the key limit processes 
\eqref{eq:3.2.2} and \eqref{eq:10.1.1} 
can be carried out for fixed $G$ and $q$.
\end{remark}

\subsection{Semigroup structure and partial ordering structure on $\Conf$.}
\label{subsec:2.3}
We introduce the following two structures 1. and 2. on $\Conf$.

1.
{\it The set $\Conf$ naturally has an abelian semigroup structure 
by putting}

\centerline{
$[\mathbb{S}]\cdot[\mathbb{T}]\ :\ =\ [\mathbb{S}\sqcup \mathbb{T}]
\qquad \text{for} \ 
[\mathbb{S}],[\mathbb{T}]\in\Conf,$
}

\medskip
\noindent
where  $\mathbb{S}\sqcup \mathbb{T}$ is the disjoint union of 
graphs $\mathbb{S}$ and $\mathbb{T}$
representing the  isomorphism classes $[\mathbb{S}]$ and $[\mathbb{T}]$. 
The empty class $[\emptyset]$ plays the role of the unit and
is denoted by 1.
It is clear that 
{\it $\Conf$ is freely generated by $\Conf_{0}$.}
The power $S^{k}$ or $[\mathbb{S}]^k$  $(k\geq 0)$ denotes 
the class of a disjoint union
$\mathbb{S}\sqcup\cdots \sqcup\mathbb{S}$
of $k$-copies of $\mathbb{S}$.

\medskip
2. {\it The set $\Conf$ is partially ordered}, where we define, 
for $S$ and $T\in\Conf$, 

\centerline{
$ S\leq T \ 
\overset{def.}{\Leftrightarrow} \ $
there exist graphs $\mathbb{S}$ and $\mathbb{T}$ with 
$S=[\mathbb{S}], T=[\mathbb{T}]$  and $\mathbb{S}\subset \mathbb{T}$.
}

\medskip
{\it The unit $1=[\emptyset]$ is the unique minimal element in $\Conf$
by this ordering.}

\subsection{Covering coefficients.}\label{subsec:2.4}
For $S_{1},\ldots,S_{m}$ and $S\in\Conf$, we introduce 
a non-negative integer:
\begin{equation}\label{eq:2.4.1}
{\small
\begin{pmatrix} S \\ S_{1},\ldots,S_{m}\end{pmatrix} 
}
\ :\ =\ \ \sharp
{\small
\begin{pmatrix}\mathbb{S}\\ S_{1},\ldots,S_{m}  \end{pmatrix}
} 
\quad \in \mathbb{Z}_{\geq 0}
\vspace{-0.2cm}
\end{equation} 
and call it the {\it covering coefficient},
where 
$\left(\begin{smallmatrix} \mathbb{S}\\S_{1},\ldots,S_{m}\end{smallmatrix}
\right)$
is defined  
by the following:

\noindent
i)
Fix any  $G$-graph $\mathbb{S}$ with $[\mathbb{S} ]=S$.

\noindent
ii)
Define a set:
\vspace{-0.2cm}
\begin{equation}
\label{eq:2.4.2}  
\begin{array}{lll}
{\small
\begin{pmatrix} {\SS}\\ S_{1},\ldots,S_{m} \end{pmatrix}
}
&\ :\ =\ &\{(\mathbb{S}_{1},\ldots,\mathbb{S}_{m}) \mid
\ \mathbb{S}_{i}\subset\mathbb{S} \text{ such that } 
[\mathbb{S}_{i}]=S_{i} \\
&& \quad  (i=1,\ldots,m) \text{ and } \cup^{m}_{i=1}
|\mathbb{S}_{i}|=|\mathbb{S}|.\}
\end{array}
\vspace{-0.2cm}
\end{equation}
\vspace{-0.2cm}
\noindent
iii) Show:
an isomorphism
$\mathbb{S}\simeq\mathbb{S}'$ induces a bijection
$\left(\begin{smallmatrix} \mathbb{S}\\ S_{1},\ldots,S_{m}
\end{smallmatrix}
\right)
\simeq 
\left(\begin{smallmatrix}\mathbb{S}'\\ S_{1},\ldots,S_{m} 
\end{smallmatrix}
\right)$.

\begin{remark}
In the definition \eqref{eq:2.4.2},
one should notice that 

\noindent
i)
Each $\mathbb{S}_{i}$ in \eqref{eq:2.4.2} should be a full subgraph of
$\mathbb{S}$ (see (\ref{subsec:2.1}) Def.~2.).

\noindent
ii)
The union of the edges of $\mathbb{S}_{i}$  ($i=1,\ldots,k$) does not have to 
cover all edges of $\mathbb{S}$.

\noindent
iii)
The set of vertices 
$|\mathbb{S}_{i}|$ ($i=1,\ldots,k$) may overlap the set $|\mathbb{S}|$.
\end{remark}

\begin{example}
Let $X_{1}, X_{2}$ be elements of $\Conf_{0}$
with $\sharp X_{i}=i$ for $i=1,2$. Then
$
\left(\begin{smallmatrix}X_{2}\\ X_{1}\cdot X_{1} \end{smallmatrix}\right) 
 = 0 $
 and 
$
\left(\begin{smallmatrix}X_{2}\\ X_{1}, X_{1}
\end{smallmatrix}\right)
 = 2.
$
\end{example}

The covering coefficients are the most basic tool in the present paper.
We shall give their elementary properties in  \ref{subsec:2.5} 
and their two basic rules: {\it the composition rule}  in \ref{subsec:2.6}
and {\it the decomposition rule} in \ref{subsec:2.7}.

\subsection{Elementary properties for covering coefficients.}\label{subsec:2.5}
Some elementary properties of covering coefficients, as 
immediate consequences of the definition, 
are listed below.
They will be used in the study of the Hopf 
algebra structure on the configuration algebra in \S4. 

{\it
\noindent
{\rm i)}
$\left(\begin{smallmatrix}S\\S_{1},\ldots,S_{m}\end{smallmatrix}\right)=0$
unless $S_{i}\leq S$ for $i=1,\ldots, m$ and
$\sum\sharp S_{i}\geq\sharp S$.

\noindent
{\rm ii)}
$\left(\begin{smallmatrix}S\\ S_{1},\ldots, S_{m} \end{smallmatrix}\right)$ 
is invariant by permutations of $S_{i}\ 's$.

\noindent
{\rm iii)}
For $1\leq i\leq m$, one has an elimination rule:
\vspace{-0.1cm}
\begin{equation}\label{eq:2.5.1}
{\small
\begin{pmatrix}
S\\ S_{1},\ldots,S_{i-1},[\emptyset],S_{i+1},\ldots,S_{m} 
\end{pmatrix}\ =\
\begin{pmatrix}
S\\ S_{1},\ldots,S_{i-1},S_{i+1},\ldots,S_{m}
\end{pmatrix}
}.
\end{equation}

\noindent
{\rm iv)}
For the case $m=0$, the covering coefficients are given by
\vspace{-0.1cm}
\begin{equation}\label{eq:2.5.2}
{\small
\begin{pmatrix}S\\ [\emptyset]\end{pmatrix}
\ =\
\begin{cases}
1 \qquad & \text{\rm if} \ S=[\emptyset],\\
0 & \text{\rm otherwise},
\end{cases}
}
\end{equation}

\noindent
{\rm v)}
For the case $m=1$, the covering coefficients are given by
\vspace{-0.1cm}
\begin{equation}\label{eq:2.5.3}
{\small
\begin{pmatrix}S\\T \end{pmatrix}
\ =\ 
\begin{cases}
1 \qquad & \text{\rm if} \ S=T,\\
0 & \text{\rm otherwise},
\end{cases}
}
\end{equation}

\noindent
{\rm vi)}
For the case $S=[\emptyset]$, the covering coefficients are given by
\vspace{-0.1cm}
\begin{equation}\label{eq:2.5.4}
{\small
\begin{pmatrix}
[\emptyset]\\ S_{1},\ldots,S_{m} 
\end{pmatrix}\ =\ 
\begin{cases}
1 \qquad & \text{\rm if} \ \cup S_{i}=\emptyset,\\
0 & \text{\rm otherwise}.
\end{cases}
}
\vspace{-0.5cm}
\end{equation}
}

\subsection{Composition rule.}\label{subsec:2.6}
\begin{assertion}  
For $S_{1},\ldots,S_{m}, T_{1},\ldots,T_{n}$,
$S\!\in\! \Conf$ ($m,n\!\in\!\Z_{\ge0}$), one has
\begin{equation}\label{eq:2.6.1}
\sum_{U\in \Conf}
{\small
\begin{pmatrix}U\\ S_{1},\ldots, S_{m} \end{pmatrix}
\begin{pmatrix}S\\ U,T_{1},\ldots, T_{n} \end{pmatrix}\ =\ 
\begin{pmatrix}S\\ S_{1},\ldots, S_{m},T_{1},\ldots, T_{n} \end{pmatrix}.
}
\end{equation}
\end{assertion}
\begin{proof}
If $m=0$, then the formula reduces to \ref{subsec:2.5} iii) and iv). 
Assume $m\ge1$ and consider the map
\vspace{-0.2cm}
\begin{align*}
{\footnotesize
\begin{pmatrix}\mathbb{S}\\ S_{1},\ldots,S_{m},T_{1},\ldots,T_{n} 
\end{pmatrix}
}&\ \longrightarrow \
\bigsqcup_{U\in \Conf}
{\footnotesize
\begin{pmatrix}\mathbb{S}\\ U,T_{1},\ldots,T_{n} 
\end{pmatrix}
}
\\
(\mathbb{S}_{1},\ldots,\mathbb{S}_{m},\mathbb{T}_{1},\ldots,\mathbb{T}_{n})
&\ \longmapsto\
(\cup^{m}_{i=1}\mathbb{S}_{i},\mathbb{T}_{1},\ldots,\mathbb{T}_{n}).
\end{align*}
Here, $\displaystyle{\cup^{m}_{i=1}}\ \mathbb{S}_{i}$
means the subgraph of $\mathbb{S}$
whose vertices are the union of the vertices of the
$\mathbb{S}_{i}$  ($i=1,\ldots, m$)
(cf. (\ref{subsec:2.1}) Def.~2.) and the class 
$\left[\displaystyle{\cup^{m}_{i=1}}\ \mathbb{S}_{i}\right]$
is denoted by $U$.
The fiber over a point
$(\mathbb{U},\mathbb{T}_{1},\ldots,\mathbb{T}_{n})$ 
is bijective to the set
\vspace{-0.2cm}
{\footnotesize
$\begin{pmatrix}\mathbb{U}\\ S_{1},\ldots,S_{m} 
\end{pmatrix}$
}
so that one has the bijection
\vspace{-0.2cm}
\begin{equation*}
{\footnotesize
\begin{pmatrix}
\mathbb{S}\\ S_{1},\ldots,S_{m},T_{1}\ldots,T_{n} \end{pmatrix}
\ \simeq \bigsqcup_{U\in \Conf}
\begin{pmatrix}\mathbb{U}\\ S_{1},\ldots,S_{m} \end{pmatrix}
\begin{pmatrix}\mathbb{S}\\ U,T_{1},\ldots,T_{n} \end{pmatrix}.
}
\vspace{-0.5cm}
\end{equation*}
\vspace{-0.3cm}
\end{proof}
\begin{note} 
The LHS of \eqref{eq:2.6.1} is a finite sum, since
the only positive summands arise with $U\leq S$ due to \ref{subsec:2.5} i).
\vspace{-0.3cm}
\end{note}

\subsection{Decomposition rule.}\label{subsec:2.7}
\begin{assertion} Let $m\in\Z_{\ge0}$.
For $S_{1},\ldots,S_{m},U$ and $V\in \Conf$, 
one has
\begin{equation}\label{eq:2.7.1}
{\small
\begin{pmatrix}U\cdot V\\ S_{1},\ldots,S_{m} 
\end{pmatrix}
\ = \sum_{\substack{R_1,T_1\in \Conf\\
S_{1}=R_{1}\cdot T_{1}}} \cdots
\sum_{\substack{R_m,T_m\in\Conf\\ S_{m}=R_{m}\cdot T_{m}}}
\begin{pmatrix}U\\ R_{1},\ldots,R_{m} \end{pmatrix}
\begin{pmatrix}V\\ T_{1},\ldots,T_{m} \end{pmatrix}.
}
\end{equation}
%

Here $R_{i}$ and $T_{i} \in \Conf$ run over all possible
decompositions of $S_{i}$ in $\Conf$.
\end{assertion}
\begin{proof} If $m=0$, this is \eqref{eq:2.5.2}.
Consider the map
\vspace{-0.1cm}
\begin{align*}
{\small
\begin{pmatrix}\mathbb{U}\cdot \mathbb{V}\\ S_{1},\ldots ,S_{m} 
\end{pmatrix}
}
&\longrightarrow \bigcup_{S_{1}=R_{1}\cdot T_{1}}\ldots
\bigcup_{S_{m}=R_{m}\cdot T_{m}}
{\small
\begin{pmatrix}\mathbb{U}\\ R_{1},\ldots,R_{m} 
\end{pmatrix} \times
\begin{pmatrix}\mathbb{V}\\ T_{1},\ldots,T_{m} 
\end{pmatrix}}
,
\\
(\mathbb{S}_{1},\cdots,\mathbb{S}_{m})
&\ \longmapsto \qquad (\mathbb{S}_{1}\cap\mathbb{U},\ldots,\mathbb{S}_{m}\cap\mathbb{U})
\times
(\mathbb{S}_{1}\cap \mathbb{V},\ldots,\mathbb{S}_{m}\cap\mathbb{V}).
\end{align*}
One checks easily that the map is bijective.
\end{proof}
\begin{note}
The RHS of \eqref{eq:2.7.1} is a finite sum,
since the only positive summands arises when $R_{i}\le U$ and $T_{i}\le V$.
\end{note}


\section{Configuration algebra.}\label{sec:3}
We complete the semigroup ring $\A\cdot\Conf$, 
where  $\A$ is a commutative associative unitary algebra, 
by use of the adic topology with respect to the grading $\deg(S)\!:=\!\# S$,
and call the completion the configuration algebra.
It is a formal power series ring of  
infinitely many variables $S\!\in\!\Conf_0$. 
We discuss several basic properties of the algebra, 
including topological tensor products.

\subsection{The polynomial type configuration algebra $\Z\cdot \Conf$.}
\label{subsec:3.1}

The free abelian group generated by
$\Conf$:
\begin{equation}\label{eq:3.1.1}
\mathbb{Z}\cdot \Conf 
\end{equation}
naturally carries the structure of an algebra by the use 
of the semigroup structure on $\Conf$  (recall \ref{subsec:2.3}), 
where $[\emptyset]=1$ plays the role of the unit element.
It is isomorphic to the free polynomial algebra generated by $\Conf_{0}$, 
and hence is called {\it the polynomial type configuration algebra.}
The algebra is graded by taking $\deg(S):=\sharp(S)$
for $S\in \Conf$, since one has additivity:
\begin{equation}\label{eq:3.1.2}
\sharp(S\cdot T)\ \ =\ \ \sharp(S)+\sharp(T).
\end{equation}

\subsection{The completed configuration algebra $\mathbb{Z}\db[\Conf\db]$.}\label{subsec:3.2}
The polynomial type algebra \eqref{eq:3.1.1} is not sufficiently large for our purposes, 
since it does not contain certain limit elements which we want 
to investigate 
(cf 4.6 {\it Remark} 3 and 6.4 {\it Remark} 2).
Therefore, we localize the algebra by the completion 
with respect to the grading given in \ref{subsec:3.1}.

For $n\geq 0$, let us define an ideal in 
$\mathbb{Z}\cdot \Conf$
\begin{equation}\label{eq:3.2.1}
\mathcal{J}_{n}\ :\ =\ \ 
 \ \text{the ideal generated by} \ 
\{S\in \Conf\mid \sharp(S)\geq n\}.
\end{equation}
Taking $\mathcal{J}_n$ as a fundamental system
of neighborhoods of $0\in\Z\cdot \Conf$,
we define the {\it adic topology} on $\Z\cdot \Conf$ 
(see Remark below).
The completion 
%
\begin{equation}\label{eq:3.2.2}
\mathbb{Z}\db[\Conf\db]\ :\ =\ \ 
\underset{n}{\varprojlim\limits}\ \mathbb{Z}\cdot \Conf/\mathcal{J}_{n}
\end{equation}
will be called {\it the completed configuration algebra}, or, simply, 
{\it the configuration algebra}.
More generally, for any commutative algebra $\mathbb{A}$ with
unit, we put 
\begin{equation}
\label{eq:3.2.3}
\mathbb{A}\db[\Conf \db]\ :\ =\ \ 
\underset{n}{\varprojlim\limits}\ \ \mathbb{A}\cdot \Conf/\mathbb{A}\mathcal{J}_{n},
\end{equation}
and call it the {\it configuration algebra} over $\mathbb{A}$, or, simply, 
the configuration algebra.
The {\it augmentation ideal} of the algebra is defined as
\begin{align*}
\mathbb{A}\db[\Conf\db]_{+}& :\ =\ 
 \ \text{the closed ideal generated by} \Conf_{+}
\\
&\ \  = \ \ \text{the closure of} \ \mathcal{J}_{1} 
 \ \text{with respect to the adic topology}.
\end{align*}

Let us give an explicit expression of an element of the 
configuration algebra by an infinite series. The quotient 
$\mathbb{A}\cdot\Conf/\mathbb{A}\mathcal{J}_{n}$ is naturally bijective 
to the free module 
$\displaystyle{\prod_{\substack{S\in \Conf\\ \sharp S<n}}}\mathbb{A}\cdot S$
of finite rank. Taking the inverse limit of the bijection, 
\vspace{-0.1cm}
we obtain
\begin{equation*}
\mathbb{A}\db[\Conf\db]\ \ \simeq
\prod_{S\in \Conf}\mathbb{A}\cdot S.
\end{equation*}
In the other words, {\it any element $f$ of the configuration algebra is
expressed uniquely by an infinite series}
\begin{equation}\label{eq:3.2.4}
\begin{array}{ll}
f\ \ =\sum_{S\in \Conf}S\cdot f_{S}
\end{array}
\end{equation}
for some constants $f_{S}\in\mathbb{A}$  for all $S\in\Conf$.
The coefficient $f_{[\emptyset]}$ of the unit element is called the 
{\it constant term of $f$}. The augmentation ideal is nothing but the 
collection of those $f$ having vanishing constant term. 

\begin{rem} The topology on $\A\db[\Conf\db]$
(except for the case $q=0$) defined above is {\it not equal} to the 
topology defined by taking the powers of the augmentation ideal 
as the fundamental system of neighborhoods of 0.
More precisely, for $n>1$ and $q\not=0$, the image of the product map:
\begin{equation}\label{eq:3.2.5}
(\mathbb{A}\db[\Conf\db]_{+})^n\ \longrightarrow \ \overline{\mathbb{A}\mathcal{J}}_{n}
\end{equation}
(c.f.\ \eqref{eq:3.5.4} and \eqref{eq:3.5.5}) 
does not generate (topologically) the target ideal on the RHS 
(= the closure in $\mathbb{A}\db[\Conf \db]$ of 
the ideal $\mathbb{A}\mathcal{J}_{n}$   
$=\{f\in \A\db[\Conf\db]\mid \deg S\ge n \text{ for }f_S\not=0\}$), 
since 
there exists a connected configuration $S$ with $\deg S=n$,
but $S$, as an element in  $\mathcal{J}_{n}$, 
cannot be expressed as a function of elements of
$\mathcal{J}_{m}$ for $m<n$. 
In this sense, the name ``adic topology'' is {\it misused} here.

The notation $\mathbb{A}\db[\Conf\db]$ should not be mistaken for 
the algebra of formal power series in $\Conf$. In fact, it 
is the set of formal series in $\Conf_{0}$. 
\end{rem}

\subsection{Finite type element in the configuration algebra.}
\label{subsec:3.3}
The support for the series $f$ \eqref{eq:3.2.4}
is defined as
\begin{equation}\label{eq:3.3.1}
\operatorname{Supp}(f)\ :\ =\ \ 
\{S\in \Conf\mid f_{S}\neq 0\}.
\end{equation}

\begin{definition}
An element $f$ of a configuration algebra is said to
be of {\it finite type} if $\operatorname{Supp}(f)$ is 
contained in a finitely generated semigroup in $\Conf$.
Note that $f$ being of finite type does not mean
that $f$ is a finite sum, but means that it is expressed
only by a finite number of ``variables''. 
The subset of $\mathbb{A}\db[\Conf\db]$ consisting of all elements of 
finite type is
denoted by $\mathbb{A}\db[\Conf\db]_{finite}$. 
The polynomial type configuration 
algebra $\mathbb{A}\cdot\Conf$ is a subalgebra of 
$\mathbb{A}\db[\Conf\db]_{finite}$. 
\end{definition}

\subsection{Saturated subalgebras of the configuration algebra.}
\label{subsec:3.4}
The configuration 
algebra is sometimes a bit too large. For later applications, 
we introduce a class of its subalgebras, 
called the saturated subalgebras.

A subset $P\subset \Conf$ is called {\it saturated}
if for $S\in P$, any $T\in \Conf_0$ with  $T\le S$ belongs to $P$.
For a saturated set $P$, let us define a subalgebra 
\begin{equation}\label{eq:3.4.1}
\mathbb{A}\db[P\db]  :\ =\   \left\{ 
f\in \mathbb{A}\db[\Conf\db] \mid 
\operatorname{Supp}(f)\subset \text{the semigroup generated by $P$}
\right\}.\!\!\!\!\!\!\!\!\!
\end{equation}
We shall call a subalgebra of the configuration algebra of the form 
\eqref{eq:3.4.1} for some saturated $P$ a {\it saturated subalgebra. }
A saturated algebra $R$ is characterized 
by the properties: i) $R$ is a closed subalgebra under the adic topology of 
the configuration algebra, and 
ii) if $S\in \operatorname{Supp}(f)$ for $f\in R$
then any connected component of $S$ (as a monomial) belongs to $R$.
We call the set
\begin{equation}\label{eq:3.4.2}
\begin{array}{ll}
\operatorname{Supp}(R)\ :\ =\ \bigcup_{f\in R}\operatorname{Supp}(f)
\end{array}
\end{equation}
the support of $R$. Obviously, $\operatorname{Supp}(R)$ 
is the saturated subsemigroup of $\Conf$ generated by $P$. 
The algebra $R$ is determined from $\operatorname{Supp}(R)$.

It is clear that if $R$ is a saturated subalgebra of $\A\db[\Conf\db]$ 
then $R\cap(\A\cdot\Conf)$ is a dense subalgebra of $R$ and that $R$ is
naturally isomorphic to the completion of  $R\cap(\A\cdot\Conf)$ with 
respect to the induced adic topology.

\begin{example}
We give two typical examples of saturated sets.  

1. For any element $S\in \Conf$,
we define its {\it saturation} by
\begin{equation}\label{eq:3.4.3}
\langle S\rangle\ :\ =\ \  
\{T\in \Conf : T\leq S\}. 
\end{equation}

2. \ Let $(\Gamma,G)$ be a Cayley graph of an infinite monoid $\Gamma$ 
with respect to a finite generating system $G$. Then, by choosing 
$G\cup G^{-1}$ as the color set and $q:=\#(G\cup G^{-1})$ as the bound of valence, 
we define a  saturated subset of $\Conf$ by
\begin{equation}\label{eq:3.4.4}
\langle \Gamma,G\rangle\ :\ =\ \ \{\text{isomorphism classes of 
finite subgraphs of }(\Gamma,G)\}.
\end{equation}
\end{example}

 Obviously, the saturated subalgebra $\A\db[\langle S\rangle\db]$ consists 
 only of finite type elements, whereas the algebra 
 $\A\db[\langle\Gamma,G\rangle\db]$
 contains non-finite type elements. 
 This makes the latter algebra  
 interesting when we study limit elements in \S11.

\subsection{Completed tensor product of the configuration algebra.}
\label{subsec:3.5}

The tensor product over $\A$ of $m$-copies of $\A\cdot\Conf$ 
for $m\in\Z_{\ge0}$ is denoted by $\otimes^m (\A\cdot\Conf)$. 
In this section, we describe the completed tensor product 
$\widehat{\otimes}^m(\A\db[\Conf\db])$ of the 
completed configuration algebra,

\begin{definition} 
Let $\mathbb{A}$ be a commutative algebra with unit. 
For $m\in\Z_{\ge0}$, the completed $m$-tensor product 
$\widehat{\otimes}^m \mathbb{A}\db[\Conf\db]$
of the configuration algebra 
$\mathbb{A}\!\db[\Conf\db]$
is defined by the inverse limit
\begin{equation}
\label{eq:3.5.1}
\widehat{\otimes}^m \mathbb{A}\db[\Conf\db]\ :\ =\ \underset{n}{\varprojlim\limits}\  \otimes^m (\mathbb{A}\cdot \Conf)
/(\otimes^m\mathbb{A}\mathcal{J})_{n},
\end{equation}
where $(\otimes^m\mathbb{A}\mathcal{J})_{n}$ is the ideal 
in $\otimes^m(\A\cdot\Conf)$ given by
\begin{equation}
\label{eq:3.5.2}
(\otimes^m\mathbb{A}\mathcal{J})_{n}\ :\ =\ 
\sum_{n_{1}+ \cdots +n_{m} \geq n} \mathbb{A}
\mathcal{J}_{n_{1}} \otimes \cdots \otimes 
\mathbb{A}\mathcal{J}_{n_{m}},
\end{equation}
where  $\widehat{\otimes}^0 \mathbb{A}\db[\Conf\db]=\mathbb{A}$ and
 $\widehat{\otimes}^1 \mathbb{A}\db[\Conf\db]= \mathbb{A}\db[\Conf\db]$.
\end{definition}

We list some basic properties of 
$\widehat{\otimes}^m \mathbb{A}\db[\Conf\db]$ 
(proofs are left to the reader).

i) Since $\cap_{n=0}^\infty (\mathbb{A}\mathcal{J}^{\otimes m})_{n}=\{0\}$,
we  have the natural inclusion map 
\begin{equation}\label{eq:3.5.3}
\otimes^m(\A\cdot\Conf)\ \ \subset\ \ \widehat{\otimes}^m(\A\db[\Conf\db])
\end{equation}
whose image is a dense subalgebra with respect to the \eqref{eq:3.5.2}-adic topology.

\newpage
ii) There is a natural algebra homomorphism 
\begin{equation}\label{eq:3.5.4}
{\otimes}^m(\A\db[\Conf\db])\ \ \longrightarrow\ \ \widehat{\otimes}^m(\A\db[\Conf\db])
\end{equation}
with a suitable universal property. The image of an element 
$f_1\otimes\cdots\otimes f_m$ is denoted by 
$f_1\widehat\otimes\cdots\widehat\otimes f_m$. 
We also denote it by $f_1\otimes\cdots\otimes f_m$ 
if $f_i\in \A\cdot\Conf$ ($i=1,\cdots,m$) because of i).

iii) If $\Psi_i:{\otimes}^{m_i}(\A\cdot\Conf)\to
{\otimes}^{n_i}(\A\cdot\Conf)$ $(i=1,\cdots,l)$ are 
continuous homomorphisms, then one has the completed  
homomorphism 
\begin{equation}\label{eq:3.5.5}
\widehat\otimes_{i=1}^l \Psi_i\ :\ 
\widehat{\otimes}^{\sum_{i=1}^lm_i}(\A\db[\Conf\db])
\ \ \longrightarrow\ \ 
\widehat{\otimes}^{\sum_{i=1}^ln_i}(\A\db[\Conf\db])
\end{equation}
with some natural characterizing properties. 
In particular, the completed product map
$\A\db[\Conf\db]\widehat\otimes\A\db[\Conf\db]\to \A\db[\Conf\db]$
is sometimes denoted by M.

\subsection{Exponential and logarithmic maps.}\label{subsec:3.6}

\vspace{-0.1cm}
Let $\varphi(t)=\displaystyle{\sum^{\infty}_{n=0}}\varphi_{n}t^{n}
\in\mathbb{A}\db[t\db]$
\vspace{-0.4cm}
be a formal power series in the indeterminate $t$.
\vspace{-0.1cm}
Then the substitution of $t$ by an element $f$ of 
$\mathbb{A}\db[\Conf\db]_{+}$ to give  
$\varphi(f):=\displaystyle{\sum^{\infty}_{n=0}}
\varphi_{n}f^{n}\in\mathbb{A}\db[\Conf\db]$ defines a map 
$\varphi: \mathbb{A}\db[\Conf\db]_{+}\to \mathbb{A}\db[\Conf\db]$ 
(c.f.\ \eqref{eq:3.2.5}). 
The map  is equivariant with respect to any continuous 
endomorphism of the configuration algebra. 
The map can be restricted to any closed subalgebra of 
the configuration algebra to itself.
If $f$ is of finite type, then 
$\varphi(f)$ is also of finite type.

In particular, if $\mathbb{A}$ contains $\mathbb{Q}$, 
then we define the {\it exponential},
{\it logarithmic} and {\it power} (with an exponent 
$c\in\mathbb{A}$) maps as follows:
\begin{align}
\label{eq:3.6.1} &
\ \ \ \exp(f)\!&\!\!:\ =\ &\sum^{\infty}_{n=0}\ \frac{1}{n!}\ \ f^{n}
&  \text{for} \ f\in\mathbb{A}\db[ \Conf\db]_{+},
\\
\label{eq:3.6.2} &
\ \ \ \log(1+f)\!&\!\! :\ =\ &\sum^{\infty}_{n=1}\ \frac{(-1)^{n-1}}{n}\ \ f^{n}
& \text{for} \ f\in \mathbb{A}\db[ \Conf\db]_+, 
\\
\label{eq:3.6.3} &
\quad\ (1+f)^{c}\!&\!\!:\ =\ &\!\sum^{\infty}_{n=0}
\frac{c(c\!-\!1)\cdots (c\!-\!n\!+\!1)}{n!}\ f^n
\!\!\!\!
& \text{for} \ f\in\mathbb{A}\db[ \Conf\db]_{+}.
\end{align}
%
They satisfy the standard functional relations:
$\exp(f+g)= \exp(f)\cdot \exp(g)$, 
$\log((1+f)(1+g)) =\log(1+f)+\log(1+g)$,
$(1+f)^{c_{1}}\cdot (1+f)^{c_{2}} =(1+f)^{c+c_{2}}_{1}$ and
$\log((1+f)^{c})= c\cdot\log(1+f)$. 

%

\noindent
{\bf Fact.} {\it Let  $\mathcal{A}\!=\sum_{S\in\Conf_+}S\cdot A_S$ and 
$\mathcal{M}\!=\!\sum_{S\in \Conf_+} S\cdot M_S \in \mathbb{A}\db[ \Conf\db]$ 
by related by 
$\mathcal{A}\!=\!\exp(\mathcal{M})$ ($\Leftrightarrow  \mathcal{M}\!=\!\log(\mathcal{A})$).
Then their coefficients are related by
\begin{equation}
\label{eq:3.6.4} 
A_{S}\ =\  \sum_{m=0}^{\infty}
\sum_{\substack{S_1,\cdots,S_m\in\Conf_+\\
S=S_{1}^{k_{1}}\cdot\ldots\cdot S_{m}^{k_{m}}}}
\frac{1}{k_{1}!\ldots k_{m}!}
M_{S_{1}}^{k_{1}}\cdots M_{S_{m}}^{k_{m}},
\vspace{-0.4cm}
\end{equation}
\vspace{-0.2cm}
and 
\begin{equation}
\label{eq:3.6.5} 
M_{S}\ =\ \sum_{m=0}^\infty
\sum_{\substack{S_1,\cdots,S_m\in\Conf_+\\ 
S=S_{1}^{k_{1}}\cdot\ldots\cdot S_{m}^{k_{m}}}}
\!\!\!\!\!\!\!\!\!\frac{(k_{1}+\cdots +k_{m}\!-\!1)!(-1)^{k_{1}
+\cdots+k_{m}^{-1}}}{k_{1}!\ldots k_{m}!}
A_{S_{1}}^{k_{1}}
\cdots A_{S_{m}}^{k_{m}}.\!\!
\end{equation}
Here the summation index runs over the set of all decompositions of $S$:
\begin{equation}\label{eq:3.6.6}
S\ \ =\ \ S^{k_{1}}_{1} \cdot\ \ldots\ \cdot S^{k_{m}}_{m}
\end{equation}
for pairwise distinct $S_{i}\in \Conf_{+} \ (i=1,\ldots,m)$ 
(which may not necessarily be connected) 
and for positive integers $k_{i}\in\mathbb{Z}_{>0}$.
Two decompositions $S^{k_{1}}_{1}\cdot\ldots\cdot S^{k_{m}}_{m}$
and $T^{l_{1}}_{1}\cdot\ldots\cdot T^{l_{n}}_{n}$
are regarded as the same if $m=n$ and
there is a permutation $\sigma\in \frak{S}_{n} \ such\ that \ k_{i}=l_{\sigma(i)}$
and $S_{i}=T_{\sigma(i)}$ for $i=l,\ldots,m$.
The RHS's of \eqref{eq:3.6.4} and \eqref{eq:3.6.5}
are finite sums, since the $S_{i}$'s and $k_i$'s are bounded by $S$.
}

\smallskip
We omit the proof since it is a straightforward calculation of 
formal power series in the infinite generating system $\Conf_{0}$.



\begin{cor}
Let  $\mathcal{A}$ and $\mathcal{M}\in \mathbb{A}\db[ \Conf\db]$ be related as above. Then one has
\begin{equation}
\label{eq:3.6.7}
\qquad A_S=M_S \qquad \forall S\in \Conf_0.
\end{equation}
\end{cor}


\section{The Hopf algebra structure}\label{sec:4}

We construct a topological commutative Hopf algebra structure on
the configuration algebra $\mathbb{A}\db[\Conf\db]$.
More precisely, we construct in \ref{subsec:4.1} a sequence of co-products 
$\Phi_n$ ($n\in\Z_{\ge0}$) by the use of the covering coefficients 
and, in \ref{subsec:4.4}, the antipode $\iota$, which together 
satisfy the axioms of a topological Hopf algebra.

\subsection{Coproduct \ $\Phi_{m}$ for $m\in\Z_{\ge0}$.}\label{subsec:4.1}

For a non-negative integer $m\in\mathbb{Z}_{\ge0}$ and $U\in \Conf$,
define an element
\begin{equation}\label{eq:4.1.1}
\Phi_{m}(U)\ :\ =\ \sum_{S_{1}\in \Conf} \cdots \sum_{S_{m}\in \Conf}
{\small 
\begin{pmatrix}U\\ S_{1},\ldots,S_{m} \end{pmatrix}
}
\ S_{1}\otimes\ldots\otimes S_{m}
\end{equation}
in the tensor product ${\otimes}^m(\mathbb{Z}\cdot \Conf)$
of $m$-copies of the polynomial type configuration 
algebra.
Due to \ref{subsec:2.5} v), one has,
\begin{equation}\label{eq:4.1.2}
\Phi_{m}([\emptyset])\ \ =\ \ [\emptyset]\ \ (=1).
\end{equation}

The map $\Phi_{m}$ is {\it multiplicative.}
That is, for $U,V\in \Conf$, one has
\begin{equation}\label{eq:4.1.3}
\Phi_{m}(U\cdot V)\ \ =\ \ \Phi_{m}(U)\cdot\Phi_{m}(V).
\end{equation}

\begin{proof}
The decomposition rule \eqref{eq:2.7.1} implies the formula.
\end{proof}

Thus, the linear extension of $\Phi_{m}$ induces an 
algebra homomorphism  from $\mathbb{Z}\cdot\Conf$
to its $m$-tensor product $\otimes^m(\mathbb{Z}\cdot\Conf)$,  
which we denote by the same  $\Phi_{m}$ and call this the $m$th
{\it coproduct}.
The coproduct $\Phi_{m}$ can be further extended to a coproduct 
on the completed configuration algebra. 

\begin{assertion}
1. {\it The $m$th coproduct $\Phi_{m}$ ($ m \in \mathbb{Z}_{\geq 0}$)
on the polynomial type configuration algebra is continuous with respect to the adic topology. 
The induced homomorphism is denoted again by $\Phi_m$ 
and called the $m$th coproduct:}
\begin{equation}\label{eq:4.1.4}
\Phi_{m}\ :\ \mathbb{A}\db[\Conf\db]\ \longrightarrow \
\widehat{\otimes}^m\mathbb{A}\db[\Conf\db]:=\mathbb{A}\db[\Conf\db]\widehat{\otimes} \cdots \widehat{\otimes} \mathbb{A}\db[\Conf\db]
\end{equation}

2. The completed homomorphism $\Phi_{m}$ 
has the mulitiplicativity 
\begin{equation}
\label{eq:4.1.5}
\Phi_m(f\cdot g)\ =\ \Phi_m(f)\cdot\Phi_m(g)
\end{equation}
for any $f,g \in \mathbb{A}\db[\Conf\db]$,
\end{assertion}

3. Any saturated subalgebra $R$ of $\mathbb{A}\db[\Conf\db]$ 
is preserved by  $\Phi_{m}$:  
\begin{equation}\label{eq:4.1.6}
\Phi_m(R)\ \ \subset\ \ \widehat\otimes^m R.
\end{equation}
\begin{proof} 
1. Recall the fundamental system $({\otimes^ m}\A\mathcal{J})_{n}$ 
\eqref{eq:3.5.2} of neighborhoods of 
the $m$-tensor algebra $\otimes^m(\Z\cdot\Conf)$. Let us show the inclusion
\begin{equation}\label{eq:4.1.7}
\Phi_m(\A\mathcal{J}_{n})\ \ \subset\ \ ({\otimes^ m}\A\mathcal{J})_{n}
\end{equation}
for any $m,n\in \Z_{\ge0}$. The ideal $\mathcal{J}_{n}$ is generated  by 
$U \in \Conf$ with $\deg(U):=\#U\ge n$, and $\Phi_{m}(U)$ 
is a sum of monomials
$S_{1} \otimes \cdots \otimes S_{m}$ for $S_{i} \in \Conf$
such that (${S_{1},\ldots,S_{m}}\atop{U}$) $\neq 0$.
We have  $\sharp S_{1} +\cdots+ \sharp S{_m}\ge \sharp(U)\ge n$ because of 
(\ref{subsec:2.5}) i), implying $\Phi_m(U)\in (\otimes^m\mathcal{J})_{n}$.

2. The multiplicativity of the monomials \eqref{eq:4.1.3}  
implies the multiplicativity of the configuration algebra of polynomial type.
This extends to multiplicativity on infinite series \eqref{eq:3.2.4} 
because of the continuity of the product with respect to the 
adic topology. 

3. Let  $f$ be an element of $R$ and $f=\sum_S S f_S$ be its expansion. 
Then $\Phi_m(f)$ is a series of the form 
$\sum_S S_1\otimes \cdots \otimes S_m ({{S_{1},\ldots,S_{m}}\atop{S}})f_S$.
Thus,
$({{S_{1},\ldots,S_{m}}\atop{S}})f_S\not=0$ implies  
each factor $S_i$ satisfies $S_i\le S$ and $S\in \operatorname{Supp}(f)\subset\operatorname{ Supp}(R)$. By the definition of saturation,  $S_i\in  \operatorname{Supp}(R)$ and 
$\Phi_m(f)\in \widehat\otimes^mR$.
\end{proof}

{\bf Co-commutativity} of the coproduct $\Phi_{m}$.

\noindent
The symmetric group $\frak{S}_{m}$ acts naturally on the m-tensors 
\eqref{eq:3.5.1} by permuting the tensor factors.
The image of  $\Phi_{m}$ lies in the subalgebra consisting of
$\frak{S}_{m}$-invariant elements, because of \ref{subsec:2.5} ii):
$\Phi_{m}(\mathbb{A}\db[\Conf\db])\subset
(\widehat{\otimes}^m\mathbb{A}\db[\Conf\db])^{\frak{S}_{m}}$.
We shall call this property the {\it co-commutativity} of 
the coproduct $\Phi_{m}$.

\subsection{Co-associativity}\label{subsec:4.2}

\begin{assertion} 
For $m,n \in \mathbb{Z}_{\ge0}$, one has the formula:
\begin{equation}\label{eq:4.2.1}
(\ \underbrace{1 \widehat\otimes \cdots \widehat\otimes 1}_{n}\ \widehat\otimes \Phi_{m}) \circ \Phi_{n+1}\ \ =\ \ \Phi_{m+n}
\end{equation}
\end{assertion}

\begin{proof}
This follows immediately  from the composition rule \eqref{eq:2.6.1}.
\end{proof}

\noindent
Using the co-commutativity of $\Phi_{2}$, $\Phi_{3}$ can be expressed 
in two different ways:  
\begin{equation*}
(\Phi_{2} \widehat\otimes 1) \circ \Phi_{2}\ \ =\ \ (1 \widehat\otimes \Phi_{2}) \circ \Phi_{2}.
\end{equation*}
This equality is the {\it co-associativity} of the coproduct $\Phi_{2}$.
More generally, $\Phi_{m}$ is expressed by a composition of $m-1$ copies of
$\Phi_{2}$'s in any order.

\subsection{The augmentation map $\Phi_0$.}\label{subsec:4.3}

The augmentation map for the algebra is defined by $\Phi_0$ 
(recall \eqref{eq:2.5.2}):
\begin{equation}
\label{eq:4.3.1}
\aug:=\Phi_0 :\ \mathbb{A}\db[\Conf\db]\ \longrightarrow\ \mathbb{A},
\quad   S \in \Conf_{+} \mapsto 0,\ \ [\emptyset] \mapsto 1 
\end{equation}

\begin{assertion}
The map $\aug$ is the co-unit with respect to the coproduct $\Phi_{2}$.
\begin{equation}\label{eq:4.3.2}
({\aug \widehat\otimes \id}) \circ \Phi_{2}\ \ =\ \ {\rm id}_{\mathbb{Z}\db[\Conf\db]}.
\end{equation}
\end{assertion}
\begin{proof} This is the case $m=0$ and $n=1$ of the formula 
\eqref{eq:4.2.1}. Alternatively, for any $S \in \Conf_{+}$,
using (\ref{subsec:2.5}) iii) and iv), one calculates:
$
(\operatorname{aug} \widehat\otimes \id) \circ \Phi_{2} (S) = \sum_{T,U \in \Conf}
\footnotesize
{\begin{pmatrix}
S\\ T, U
\end{pmatrix}
}
T \cdot \operatorname{aug}(U) = \sum_{T \in \Conf}
\footnotesize{
\begin{pmatrix}
S\\ T, [\emptyset]
\end{pmatrix}
}
= S.
$
\end{proof}

\subsection{The antipodal map $\iota$}\label{subsec:4.4}

The coproduct and the co-unit exist 
both on the polynomial type and the completed configuration algebras.
The co-inverse (or antipode), which we construct in the present 
section, exists only on the localized configuration algebra.

\begin{assertion}{\it 
There exists an $\mathbb{A}$-algebra endomorphism
\begin{equation}\label{eq:4.4.1}
\iota\ :\ \mathbb{A}\db[\Conf\db]\ \longrightarrow\ \mathbb{A}\db[\Conf\db],
\end{equation}
satisfying following properties {\rm i)-v)}. 

\noindent
\ \ {\rm i)}\ 
$\iota$ is an involution. 
That is, $\iota$ is an automorphism with $\iota^{2} = \id_{\mathbb{A}\db[\Conf\db]}$.

\noindent
\ {\rm ii)}\
$\iota$ is the co-inverse map with respect to the coproduct $\Phi_{2}$, 
that is,
\begin{equation}\label{eq:4.4.2}
M\circ (\iota \widehat\otimes \id) \circ \Phi_{2}\ \ =\ \  \aug.
\end{equation}
 where $M$ is the product defined on the completed tensor product (recall {\rm \ref{subsec:3.5}}).

\noindent
{\rm iii)}\
$\iota$ is continuous with respect to  the adic topology.
More precisely,
\begin{equation}\label{eq:4.4.3}
\iota (\overline{\mathcal{J}_{n}})\ \ \subset\ \ \overline{\mathcal{J}_{n}}
\end{equation}
for $n \in \mathbb{Z}_{\ge0}$,
where $\overline{\mathcal{J}_{n}}$ is the closure of the ideal 
$\mathcal{J}_{n}$ \eqref{eq:3.2.1}.

\noindent
{\rm iv)}\
$\iota$ leaves any saturated subalgebra of $\mathbb{A}\db[\Conf\db]$ invariant.

\noindent
{\rm v)}\
Any $\mathbb{A}$-endomorphism of $\mathbb{A}\db[\Conf\db]$ satisfying 
{\rm ii)} and {\rm iii)} is equal to $\iota$. 
}
\end{assertion}

\begin{proof}
The proof is divided into two parts. 
In Part 1, we construct an  endomorphism $\varphi$ of the 
algebra $\mathbb{A}\db[\Conf\db]$, satisfying i), ii), iii) and iv). 
In Part 2, we show that any endomorphism  $\psi$ of $\mathbb{A}\db[\Conf\db]$ 
satisfying the properties ii) and iii)  
coincides with $\varphi$.

\smallskip
Part 1. 
Let us fix a bijection $i\in\Z_{\ge1}\mapsto S_i\in\Conf_0$ such that 
if $S_i\le S_j$ then $i\le j$ (such linearization 
exists since one can linearize the partially ordered structure on the 
finite set of configurations for  a fixed number of vertices). 
Note that this condition implies that the set
$\{S_1,S_2,\cdots,S_i\}$ for $i\in\Z_{>0}$ 
is saturated in the sense of \ref{subsec:3.4}.
Consider the increasing sequence $R_0:=\A$, $R_i:=\A\db[S_1,S_2,\cdots,S_i\db]$ 
($i\in\Z_{>0}$) of saturated subalgebras of $\A\db[\Conf\db]$. 
Let us show:

Claim: {\it there exists a sequence $\{\varphi_i\}_{i\in\Z_{\ge0}}$ of continuous endomorphisms
 $\varphi_i:R_i\to R_i$ satisfying the following relations:

a) \ $\varphi_i^2\ =\ \id_{R_i}$ \ for $i\in\Z_{\ge0}$.

b) \ $\varphi_i|_{R_{i-1}}\ =\ \varphi_{i-1}$ \ for $i\in\Z_{\ge1}$. 

c) \ $M\circ(\varphi_i \cdot \id) \circ \Phi_{2}|_{R_{i}}\ =\ \aug|_{R_{i}}$ \ for $i\in\Z_{\ge0}$.

d) \ $\varphi_i (\overline{\mathcal{J}_{n}}\cap R_i)\ \subset\ 
\overline{\mathcal{J}_{n}}\cap R_i$ \ for $i\in\Z_{\ge0}$ and for $n\in\Z_{\ge0}$.

e) \ $\varphi_i(S_k)\in \Z\db[\langle S_k\rangle\db]$ \ for $1\le k\le i$ 
(see \eqref{eq:3.4.1} and \eqref{eq:3.4.3} for notation).
}

 \medskip 
{\it Proof of} Claim. We construct the sequence $\varphi_i$  inductively.
Put $\varphi_0:=\id_{\A}$.
For $j\in \Z_{\ge0}$, suppose that $\varphi_{0},\cdots,\varphi_j$ 
satisfying a)-e) for $i\le j$ are given. 

For any given element 
$X\!\in\! \Z\db[\langle S_{j+1}\rangle\db]\cap \overline{\mathcal{J}_{\#(S_{j+1})}}$,
define an endomorphism  $\psi_X$ of $R_{j}[S_{j+1}]$ 
by $\psi_X|R_j=\varphi_j$ and $\psi_X(S_{j+1})=X$. 
Since $X\in  \overline{\mathcal{J}_{\#(S_{j+1})}}$, 
we have $\psi_X(\overline{\mathcal{J}_{n}} \cap R_{j}[S]) \subset \overline
{\mathcal{J}_{n}}$ $ ^{\forall}n \in \mathbb{Z}_{n}$. 
Hence the endomorphism is continuous in the adic topology and is extended to 
an endomorphism of $R_{j+1}=R_{j}\db[S\db]$.
We denote the extended homomorphism again by $\psi_X$.
Let us show that $\varphi_{j+1}\!:=\!\psi_X$ for a suitable choice of $X$ 
satisfies a)-e) for $i=j+1$. 
Actually, b), d) and e) are already satisfied by the construction. 
In order to satisfy a) and c), we have only to solve the following 
two equations on $X$: 

\centerline{a)$^*$ \  $\psi_X^2(S_{j+1})\!=\!S_{j+1}$ \quad and\quad
 c)$^*$ \ $ M\circ (\psi_X \widehat\otimes 1) \circ \Phi_{2}(S_{j+1}) = 0$. 
}
\noindent
In the following, we show an existence of a simultaneous solution of a)$^*$ and c)$^*$.

c)$^*$  Let us write down the equation c)$^*$ explicitly by using \eqref{eq:4.1.1}.
\[
\begin{array}{ll}
M\circ (\psi_X\widehat\otimes \id)\circ\Phi_{2}(S_{j+1})\ =\ \sum_{U,V\in \Conf}
{\footnotesize 
\begin{pmatrix}S_{j+1}\\ U,V 
\end{pmatrix}
}
\psi_X(U)\cdot V \ =\ 0 \ .
\end{array}
\]
The summation index $(U,V)$ runs over the finite 
set $\langle S_{j+1}\rangle\!\times\! \langle S_{j+1}\rangle$ 
(\ref{subsec:2.5} i)). 
Decompose the set into three pieces:
$\{S_{j+1}\}\!\times\! \langle S_{j+1}\rangle$, 
$(\langle S_{j+1}\rangle\!\setminus\!\{S_{j+1}\})\!\times\! (\langle S_{j+1}\rangle\!\setminus\!\{S_{j+1}\})$
and $(\langle S_{j+1}\rangle\!\setminus\!\{S_{j+1}\})\!\times\! \{S_{j+1}\}$. 
Since $\langle S_{j+1}\rangle\!\setminus\!\{S_{j+1}\}\!\subset\! \langle S_1,\cdots, S_{j}\rangle$, on which $\psi_X$ coincides with $\varphi_j$, 
the equation consists of three parts:
\begin{equation}\label{eq:4.4.4}
X \cdot \mathcal{A} (S_{j+1})  + \mathcal{B} (S_{j+1}) 
+  \varphi_j \big(\mathcal{A} (S_{j+1}) - S_{j+1} \big) \cdot S_{j+1}\ =\ 0,
\vspace{-0.2cm}
\end{equation}
where
\vspace{-0.4cm}
\begin{equation}\label{eq:4.4.5}
\begin{array}{ll}
\mathcal{A} (S_{j+1})\ :\ =\ \sum_{V \in \langle S_{j+1}\rangle}
{\footnotesize
\begin{pmatrix}
S_{j+1}\\ S_{j+1},V
\end{pmatrix}
}
V,
\vspace{-0.6cm}
\end{array}
\end{equation}
and
\vspace{-0.4cm}
\begin{equation}
\begin{array}{ll}
\mathcal{B} (S_{j+1})\ :\ =\ \sum_{U,V \in \langle S_{j+1}\rangle\setminus\{S_{j+1}\}}
{\footnotesize 
\begin{pmatrix}
S_{j+1}\\ U,V
\end{pmatrix}
}
\varphi_j(U) \cdot V. \tag*{$(4.4.5)^*$}
\end{array}
\end{equation}

We have the following facts concerning the equation \eqref{eq:4.4.4}.

i) By definition \eqref{eq:4.4.5}, each term of 
$\mathcal{A}(S_{j+1}) - S_{j+1}$ is an element 
$\langle S_{j+1}\rangle\setminus\{S_{j+1}\}$, i.e.\ 
is a monomial of $S_k$'s for $1\le k\le j$. 
Therefore, by the induction hypothesis e), 
$\varphi_j(\mathcal{A}(S_{j+1}) - S_{j+1})\in R_j$.

ii) By the hypothesis d) for $\varphi_j$,
$\varphi_j(U)$ belongs to $\overline{\mathcal{J}_{\sharp(U)}} \cap 
\mathbb{Z}\db[\langle U\rangle\db]$ for $^{\forall}U \leq S_{j+1}$ 
with $U \neq S_{j+1}$.
Hence, in view of 2.5 i), 
{\small
$\begin{pmatrix}
S_{j+1}\\ U,V
\end{pmatrix}\ne0$ 
}
implies $\varphi_j(U) \cdot V \in \overline{\mathcal{J}_{\sharp(S_{j+1})}}$,
and, therefore,
(4.5.5)* implies $\mathcal{B}(S_{j+1}) \in
\overline{\mathcal{J}_{\sharp(S_{j+1})}} \cap \mathbb{Z}\db[\langle S_{j+1}\rangle\db]$.

iii) By definition \eqref{eq:4.4.5}, 
$\mathcal{A} (S_{j+1})-1$ belongs to the augmentation ideal.
Hence, in view of the inclusion \eqref{eq:3.2.5}, 
the sum $\sum_{m \geq 0}(1- \mathcal{A} (S_{j+1}))^{m}$ converges in
$\mathbb{Z}\db[\langle S_{j+1} \rangle\db]$ to the inverse $\mathcal{A} (S_{j+1})^{-1}$.

Therefore the  equation \eqref{eq:4.4.4} for $X$ has a unique 
solution in $\mathbb{Z}\db[\langle S_{j+1} \rangle\db]$:
\begin{equation}\label{eq:4.4.6} 
X\ :\ =\ \frac{-1}{ \mathcal{A}(S_{j+1})} \big(\mathcal{B}(S_{j+1}) +
\varphi_j \big(\mathcal{A}(S_{j+1}) -S_{j+1} \big) \cdot S_{j+1} \big). 
\end{equation}

As a consequence of the above Facts i) and ii), we have 

\medskip
\noindent
$*$) \ \textit{The right hand side of \eqref{eq:4.4.6} belongs to 
$\overline{\mathcal{J}_{\sharp(S_{j+1})}} \cap \mathbb{Z}\db[\langle S_{j+1}\rangle\db]$.}

\medskip
This what we have asked for $X$ at the beginning. 
Thus, c)$^*$ is solved.

a)$^*$  We need to show:  $\psi_X^{2}(S_{j+1}) = S_{j+1}$ 
under the choice \eqref{eq:4.4.6}.
Apply $\psi_X$ to the equality \eqref{eq:4.4.4}.
Using the induction hypothesis a) and b), one gets
\begin{equation}\tag*{$**)$}
\psi_X^{2}(S_{j+1})\big(\varphi_j(\mathcal{A}(S_{j+1}) - S_{j+1}) + \psi_X(S_{j+1})\big) +
\varphi_j \mathcal{B}(S_{j+1}) + \big(\mathcal{A}(S_{j+1}) - S_{j+1}\big)X = 0
\end{equation}

Here, we have $\varphi_j \mathcal{B}(S) =  \mathcal{B}(S)$, by applying the symmetry
(\ref{subsec:2.5}) ii) and the induction hypothesis a) to the expression 
(4.4.5)*. 
Subtract \eqref{eq:4.4.4} from $**$):
\begin{equation*}
\big(\psi_X^{2}(S_{j+1}) - S_{j+1} \big) \big(\varphi_j (\mathcal{A}(S_{j+1}) - S_{j+1}) + 
\psi_X(S_{j+1}) \big) = 0.
\end{equation*}  
Since $\mathcal{A}(S_{j+1})-1 \in \mathcal{J}_{1}$ and  
$\varphi_j(\mathcal{A}(S_{j+1}) - S_{j+1}) + \psi_X(S_{j+1})-1 \in \mathcal{J}_{1}$,
$\varphi_j(\mathcal{A}(S_{j+1}) - S_{j+1}) + \psi_X(S_{j+1})$ 
is invertible in the algebra $R_{j+1}$.
This implies $\psi_X^{2}(S_{j+1}) - S_{j+1} = 0$. That is, $\psi_X$ for 
\eqref{eq:4.4.6} satisfies a)$^*$.

Thus, by the choice $\varphi_{j+1}\!:=\!\psi_X$, 
the induction step for $j\!+\!1$ is achieved. 
 $\Box$

We define the endomorphism $\varphi$ of the subalgebra 
$R:=\cup_{i=1}^\infty R_i$ by $\varphi|R_i:=\varphi_i$ for $i\in\Z_{\ge0}$.
Here, we note that $R$ comprise exactly the finite type elements, 
i.e.\ $R=\mathbb{A}\db[\Conf\db]_{finite}$, which is 
dense in the configuration algebra.
Then d) implies continuity of $\varphi$ on $R$, and 
therefore $\varphi$ extends to the full configuration algebra.
The extended homomorphism, denoted by $\varphi$ again, satisfies 
involutivity since this is so on the dense subalgebra $R$. 
This implies $\varphi$ is invertible.

Finally in Part 1, let us show that $\varphi$ satisfies iv).

We remark that, for any element $S\!\in\! \Conf$, one has 
$\varphi(S)\!\in\! \Z\db[\langle S\rangle\db]$ ({\it proof.} apply e) to each 
connected component of $S$).
Let $R$ be any saturated subalgebra of the configuration algebra. 
For any $f\!=\!\sum_S Sf_S\! \in\! R$, by applying the above considerations,
one has $\operatorname{Supp}(\varphi(f))\!\subset\!
\cup_{f_S\not=0} \operatorname{Supp}(\varphi(S))\!\subset \!
\cup_{f_S\not=0}$ semigroup generated by $\langle S\rangle
\subset \operatorname{Supp}(R)$.
That is, $\varphi(f)\in R$ and $\varphi(R)\subset R$.

Part 2. Uniqueness of $\varphi$.  Let $\psi$ be any endomorphism of 
the configuration algebra satisfying ii) and iii). 
Let $S_1,S_2,\cdots$ be the linear ordering of $\Conf_0$ used in Part 1.
We show that $\psi(S_j)=\varphi(S_j)$ by induction on $j\in\Z_{\ge1}$. 
Let $j\in \Z_{\ge 0}$, and assume $\varphi(S_i)=\psi(S_i)$ for $1\le i\le j$ 
(there is no assumption if $j=0$).
By ii), $\psi(S_{j+1})$ should satisfy the same equation  
\eqref{eq:4.4.4} for $X$, where, by the induction hypothesis, 
one has the equality:
$\varphi \big(\mathcal{A} (S_{j+1}) - S_{j+1} \big) 
= \psi \big(\mathcal{A} (S_{j+1}) - S_{j+1} \big) $. 
The uniqueness of the solution \eqref{eq:4.4.6} 
implies $\psi(S_i)=\varphi(S_i)$.
This implies the coincidence of $\varphi$ and $\psi$ on $\Z\cdot\Conf$. 
Then, by the continuity iii), we have the coincidence 
of $\varphi$ and $\psi$ on the completed configuration algebra.
\end{proof}

Equation \eqref{eq:4.4.4} for $n=1$ implies that $\iota$ preserves 
the augmentation ideal of $\A\db[\Conf\db]$. Hence, we have 
\begin{equation}\label{eq:4.4.7}
\aug \ \circ\ \iota \ \ =\ \ \aug.
\end{equation} 

Let us state an important consequence of our construction.

\begin{assertion}
{\it Any saturated subalgebra of the configuration algebra is a 
topological Hopf algebra. In particular, 
for any monoid $\Gamma$ with a finite generating system $G$ and commutative 
ring $\A$ with a unit, $\A\db[\langle \Gamma,G\rangle\db]$ is a Hopf algebra.
}
\end{assertion}
\begin{proof} We need only to remember that $\Phi_m$ ($m\ge0$) and 
$\iota$ preserve any saturated subalgebra (4.1 Assertion 3. and 
4.4 Assertion iv)).
\end{proof}

\subsection{Some remarks on $\iota$.}\label{subsec:4.5}
\begin{rem}
1. In section 5., the functions $\mathcal{A}(S)$ \  
$(S \in \Conf)$ will be 
re-introduced and investigated.
In particular we shall show the equality:
\begin{equation}\label{eq:4.5.1}
 \iota (\mathcal{A}(S)) \cdot \mathcal{A}(S)  \ \ =\ \ 1
\end{equation} 
for $S \in \Conf$ \eqref{eq:5.4.1}.
This can be also directly shown by use of \eqref{eq:2.7.1} 
and \eqref{eq:4.2.1}.
This relation gives a more natural definition of $\iota$.

2.
The polynomial ring $\A \cdot \Conf$ for any $\A$ 
is not closed under the map $\iota$. 
For example, let $X$ (resp.\ $Y$) be a graph of one (resp.\ two)
vertices. Then,
{\small
\[
\iota(X)\ =\ -\frac{X}{1 + X} \quad \text{and}\quad 
\iota(Y)\ =\ \frac{-Y + 2X^{2} + XY}{(1 + X)(1 + 2X + Y)}.
\vspace{-0.4cm}
\]
}
       
3.
Because of above {\it Remark} 2, the localization:
${ \big(\mathbb{Z} \cdot \Conf \big)}_{\mathfrak{M}} = \{f/g: f \in \mathbb{Z}
\cdot \Conf, g \in \mathfrak{M}\}$ for the multiplicative set
$\mathfrak{M}:= \{\mathcal{A}(S): S \in \Conf \}$ is the smallest necessary
extension of the algebra $\mathbb{Z} \cdot \Conf$ to define $\iota$.
However, the space 
${\big(\mathbb{Z} \cdot \Conf \big)}_{\mathfrak{M}}$ is still 
too small for our later applications (see \ref{subsec:6.3} {\it Remark}).

4.
There  is another coalgebra structure studied in combinatorics (\cite{R}).
\end{rem}


\section{Growth functions for configurations}\label{sec:5}

For any $S \in \Conf$, the sum of isomorphism classes of all subgraphs 
of a graph representing $S$ is denoted by $\mathcal{A}(S)$.
It is a group-like element in the Hopf algebra $\A\db[\Conf\db]$
and shall play a fundamental role in the sequel. We shall call it 
a {\it growth function} (one shuld not confuse with 
the same terminology in [M]).

\subsection{Growth functions}\label{subsec:5.1}

For $S$ and $T \in \Conf$, we introduce a numerical invariant 
\begin{equation}\label{eq:5.1.1}
A(S,T)\ \ :\ =\ \ \sharp \mathbb{A}(S, \mathbb{T}), 
\end{equation}
by the following steps i)-iii).

\ \ i) 
Fix a graph $\mathbb{T}$, with $[\mathbb{T}] = T$.

\ ii)
Put
\begin{equation}\label{eq:5.1.2}
\mathbb{A}(S,\mathbb{T})\ \ :\ =\ \ \sharp \{\ \mathbb{S} \ \mid \ \mathbb{S}\text{ is a subgraph of } \mathbb{T} \text{ such that } [\mathbb{S}]=S\}.
\end{equation}

iii)
Show that $\mathbb{A}(S,\mathbb{T}) \simeq \mathbb{A}(S,\mathbb{T}')$
if $[\mathbb{T}] = [\mathbb{T}']$.
(The proof is omitted.)

\noindent
We shall call $A(S,T)$ the growth coefficient of $T$ at $S\in \Conf$. 
\begin{align}\label{eq:5.1.3}
&A([\emptyset], T)\ =\ 1 \qquad {\rm for} \quad T \in \Conf,\\
&\ \ A(S, T) \ \neq\ 0 \qquad \text{if and only if} \quad S \in \langle T\rangle.
\label{eq:5.1.4}
\end{align}

Let us introduce the generating polynomial of the growth coefficients:
\begin{equation}
\label{eq:5.1.5}
\begin{array}{ll}
\mathcal{A}(T)\ \ :\ =\  \sum_{S \in \Conf} S \cdot A(S,T),
\end{array}
\end{equation}
and  call it the growth function of $T$.  
In fact, this is a finite sum and $\mathcal{A}(T)\in \Z\cdot\Conf$.
The definition of $\mathcal{A}(T)$ can be reformulated as:
\begin{equation}\label{eq:5.1.6}
\begin{array}{ll}
\mathcal{A}(T)\ \ =\ \ \sum_{S \in 2^{\mathbb{T}}} [\mathbb{S}],
\end{array}
\end{equation} 
where $2^{\mathbb{T}}$ denote the set of all subgraphs of $\mathbb{T}$
(cf. \ref{subsec:2.1} Definition 2.).

The following multiplicativity follows immediately from the expression
\eqref{eq:5.1.6}. For $T_{1}$ and $T_{2} \in \Conf$
\begin{equation}\label{eq:5.1.7}
 \mathcal{A}(T_{1} \cdot T_{2} )\ \ =\ \  \mathcal{A}(T_{1})\ \mathcal{A}(T_{2}).
\end{equation} 

\begin{rem}
1. By comparing the definition \eqref{eq:5.1.1} with \eqref{eq:2.4.1}, we see immediately 
${\small 
A(S,T) =
\begin{pmatrix}
T\\ T,S
\end{pmatrix}
}$
for $S$ and $T \in \Conf$.
Hence the two definitions \eqref{eq:4.4.5} 
and \eqref{eq:5.1.5} for $\mathcal{A}(T)$ coincide.

2.
By definition \eqref{eq:5.1.1}, we have additivity:
\begin{equation}\label{eq:5.1.8}
A(S,T_{1} \cdot T_{2})\ \ =\ \  A(S,T_{1}) + A(S,T_{2})
\end{equation} 
for $S \in \Conf_{0}$ and $T_{i} \in \Conf$.
\end{rem}

\subsection{A numerical bound of the growth coefficients}
\label{subsec:5.2}
In our later study on the existence of limit elements in \S10, 
the following estimates on the 
growth rates of growth coefficients play a crucial role.

\begin{lemma}
For $S, T \in \Conf$, we have
\begin{equation}\label{eq:5.2.1}
A(S,T)\ \ \leq\ \ \frac{1}{\sharp \Aut(S)} \cdot \sharp T^{n(S)}
\cdot (q-1)^{\sharp S - n(S)}.
\end{equation} 
Here $n(S):= \sharp$ of connected components of $S$, $q$ is 
the upper-bound of the number of edges at each vertex of $T$ \em{({\it recall} \ref{subsec:2.2})}, and $\Aut(S)$ means the isomorphism class of 
$\Aut(\mathbb{S})$ for a representative $\mathbb{S}$ of $S$ 
and we put $\#\Aut(S):=\#\Aut(\mathbb{S})$.
\end{lemma}

\begin{note}
In the original version [S2], the factor $q-1$ in 
\eqref{eq:5.2.1} 
was $q$. The author is grateful to the readers pointing out this 
improvement.
\end{note}

\begin{proof}
Let $\mathbb{S}$ and $\mathbb{T}$ be representatives of the $G$-colored graphs 
$S$ and $T$ respectively. We divide the proof into three steps.

i) 
Assume $S$ is connected. Let us show:
\begin{equation}\label{eq:5.2.2}
A(S,T)\ \ \leq\ \  \frac{1}{\sharp \Aut(S)} \sharp T
\cdot (q-1)^{\sharp S - 1}.
\end{equation}
\begin{proof}
Let $\mathbb{S}_{1}, \ldots, \mathbb{S}_{a}$ be an increasing sequence of 
connected subgraphs of $\mathbb{S}$ such that $\sharp \mathbb{S}_{i}=i$\ 
($i = 1, \ldots, a = \sharp \mathbb{S}$).
Put $\Emb(\mathbb{S}_{i}, \mathbb{T}):= \{\varphi : \mathbb{S}_{i} 
\rightarrow \mathbb{T}\mid$ embeddings as a $G$-colored graph\}.
Then, for $i\ge2$, the natural restriction map $\Emb(\mathbb{S}_{i}, \mathbb{T})
\rightarrow \Emb(\mathbb{S}_{i - 1}, \mathbb{T})$ has at most $q-1$
points in its fiber.
Hence $\sharp\Emb(\mathbb{S}_{i}, \mathbb{T}) \leq(q-1) \cdot
\sharp\Emb(\mathbb{S}_{i - 1}, \mathbb{T})$ ($i = 2, \ldots, a$).
On the other hand, since 

\smallskip
\centerline{
$A(S,T)\ =\ \sharp\Emb(\mathbb{S}, \mathbb{T}) / 
\sharp \Aut(\mathbb{S})$,
}

\smallskip
\noindent
one has the inequality: 

\smallskip
\!\!$A(S,T) = \sharp\Emb(\mathbb{S}_{a},
\mathbb{T}) / \sharp \Aut(\mathbb{S}_{a})$ 

\qquad \quad\  $\leq (q-1)^{a-1} \cdot
\sharp\Emb(\mathbb{S}_{1}, \mathbb{T}) / \sharp \Aut(\mathbb{S}_{a})
= (q-1)^{a-1} \cdot \sharp T/ \sharp \Aut(S)$.
\end{proof} 
ii)
Assume that $S$ decomposes as: $S=S^{k_{1}}_{1} \coprod .. \coprod
S^{k_{m}}_{m}$ for pairwise distinct $S_{i} \in \Conf_{0}$ 
($i=1,\cdots,m$) so that $\sum_{i=1}^m k_{i} = n(S)$. 
Let us show 
\begin{equation}\label{eq:5.2.3}
A(S,T)\ \ \leq\ \ \frac{1}{k_{1}! \ldots k_{m}!} \prod^{m}_{i = 1} 
A(S_{i},T)^{k_{i}},
\end{equation}
\begin{proof} For $1\le i\le m$, 
the subgraph of  $\mathbb{S} \in \mathbb{A}(S, \mathbb{T})$ 
 corresponding to the factor $S_i^{k_i}$,
denoted by $\SS|_{S_i^{k_i}}$, defines an off-diagonal element of
$( \prod^{k_{i}} \mathbb{A}(S_{i},\mathbb{T}) 
) /\frak{S}_{k_{i}}$ where $\frak{S}_{k_{i}}$ is the symmetric 
group of $k_i$ elements acting freely on the 
set of off-diagonal elements.
Then, the association $\SS\mapsto  (\SS|_{S_i^{k_i}})_{i=1}^m$ 
defines an embedding:
$
\mathbb{A}(S,\mathbb{T}) \rightarrow \prod^{m}_{i = 1} 
\Big( \Big( \prod^{k_{i}} \mathbb{A}(S_{i},\mathbb{T}) 
\Big) /\frak{S}_{k_{i}}\Big)
$ into the off-diagonal part.
\end{proof} 

iii) Let $S$ be as in ii). Then, $\Aut(S)=\prod_{i=1}^m\Aut(S_i^{k_i})$
and each factor $\Aut(S_i^{k_i})$ is a wreath direct product of
$\Aut(S_i)$ and $\frak{S}_{k_i}$.
Then \eqref{eq:5.2.1} is a consequence of a combination of 
\eqref{eq:5.2.2} and \eqref{eq:5.2.3}. 

This completes the proof of the Lemma. \end{proof}

\subsection{Product-expansion formula for growth coefficients}
\label{subsec:5.3}
The coefficients of a growth function of $T$ are not algebraically 
independent.
\begin{lemma}
Let $S_{1}, \ldots, S_{m} \ (m \geq 0)$ and $T \in \Conf$ be given.
Then,
\begin{equation}\label{eq:5.3.1}
\prod^{m}_{i = 1}A(S_{i},T)\ \ =\ \sum_{S \in \Conf} 
{\small
\begin{pmatrix}
S\\ S_{1}, \ldots, S_{m}
\end{pmatrix}
}
A(S,T).
\end{equation}
\end{lemma}

\begin{proof}
Let $\mathbb{T}$ be a graph representing $T$. For $m\in\Z_{\ge 0}$,
consider a map
\begin{equation*}
(\mathbb{S}_{1}, \ldots, \mathbb{S}_{m}) \in \prod^{m}_{i = 1}\mathbb{A}
(S_{i},\mathbb{T})\ \longmapsto\ 
\mathbb{S}:= \bigcup_{i = 1}^{m} \mathbb{S}_{i} \in 2^\mathbb{T},
\end{equation*}
whose fiber over $\mathbb{S}$ is ($\begin{smallmatrix}
S_{1}, \ldots, S_{m}\\
{\mathbb{S}}
\end{smallmatrix}$)
so that one has the decomposition
\begin{equation*}
\prod^{m}_{i = 1}\mathbb{A}(S_{i},\mathbb{T})\ \ \simeq\ \ 
{\bigcup_{\mathbb{S} \in 2^{\mathbb{T}}}} 
{\small
\begin{pmatrix}
\mathbb{S}\\ S_{1}, \ldots, S_{m}
\end{pmatrix}
}
.
\end{equation*}
By counting the cardinality of the both sides, one obtains the formula.
\end{proof}

\begin{remark}
The formula \eqref{eq:5.3.1} is trivial for $m\in\{0,1\}$, and 
can be reduced to the case $m=2$ for $m\ge2$ by an induction 
on $m$ as follows.

Multiply $A(S_{m+1}, T)$ to \eqref{eq:5.3.1} and apply the formula for $m=2$.
\begin{align*}
\begin{array}{ll}
\prod^{m+1}_{i=1} A(S_{i}. T) &\  = \sum_{S \in \Conf}
{\small
\begin{pmatrix}
S\\ S_{1}, \ldots, S_{m}
\end{pmatrix}
}
A(S,T)A(S_{m+1}, T)\\
&\  = \sum_{S \in \Conf}
{\small
\begin{pmatrix}
S\\ S_{1}, \ldots, S_{m}
\end{pmatrix}
}
\sum_{U \in \Conf} 
{\small
\begin{pmatrix}
U\\ S, S_{m+1}
\end{pmatrix}
}
A(U,T)\\
\intertext{Using the composition rule \eqref{eq:2.6.1}, this is equal to}
&\ = \sum_{U \in \Conf}
{\small
\begin{pmatrix}
U\\ S_{1}, \ldots, S_{m+1}
\end{pmatrix}
}
A(U,T).
\end{array}
\end{align*}
\end{remark}

\subsection{Group-like property of the growth function}\label{subsec:5.4}

An element $g \in \mathbb{A} \db[\Conf\db]$  is called 
group-like if it satisfies
\begin{equation}\label{eq:5.4.1}
\Phi_{m}(g) \ =\ \underbrace{g \ \widehat\otimes\ \cdots \ \widehat\otimes\ g}_m
\vspace{-0.3cm}
\end{equation}
for all $m \in \mathbb{\Z}_{\ge0}$.
This in particular implies the conditions
$\Phi_0(g)=1$ and $\iota(g)=g^{-1}$ (c.f.\ 
\eqref{eq:4.3.1} and \eqref{eq:4.4.2}). For any group-like elements
$g$ and $h$, the power product $g^ah^b$ for $a,b\in\A$ 
(c.f.\ \eqref{eq:3.6.3}) is also group-like.
 We put
\begin{align}
\label{eq:5.4.2}
& \frak{G}_{\mathbb{A}}\ \ :\ =\ \  \{ \text{the set of all group-like elements in }\mathbb{A} \db[\Conf\db]\}\\
& \frak{G}_{\mathbb{A}, {\rm finite}}\ \ :\ =\ \ \{g \in \frak{G}_{\mathbb{A}} \mid g \ \text{is of 
finite type.}\}
\end{align}

\begin{lemma}
The generating polynomial $\mathcal{A}(T)$ for any $T \in \Conf$ 
is group-like.
 That is, for any $m \in \mathbb{Z}_{\ge0}$ 
and $T \in \Conf$, we have 
\begin{equation}
\label{eq:5.4.4}
 \mathcal{A} (T) \otimes \cdots \otimes  \mathcal{A} (T)
\ \ =\ \ 
\Phi_{m}(\mathcal{A} (T)).
\end{equation}
\end{lemma}

\begin{proof} 
By the definition of $\mathcal{A}(T)$ \eqref{eq:5.1.3},
the tensor product of $m$-copies
\begin{equation} 
 \mathcal{A}(T) \otimes \cdots \otimes  \mathcal{A}(T)
\tag*{$*$)}
\end{equation}
can be expanded into a sum of $m$ variables $S_{1}, \ldots, S_{m}$:
\begin{equation} 
\sum_{S_{1} \in \Conf} \cdots \sum_{S_{m} \in \Conf} S_{1} \otimes
\cdots \otimes S_{m} \Big( \prod^{m}_{i=1} A(S_{i}, T) \Big). \tag*{$**$)} 
\end{equation}
By use of the product-expansion formula \eqref{eq:5.3.1},  this is 
equal to 
\begin{equation*} 
\sum_{S_{1} \in \Conf} \cdots \sum_{S_{m} \in \Conf} S_{1} \otimes
\cdots \otimes S_{m} \Big( \sum_{S \in \Conf} 
{\footnotesize
\begin{pmatrix}
S\\ S_{1}, \cdots, S_{m}
\end{pmatrix}
}
A(S,T)\Big)
\end{equation*}
Recalling the definition of the map $\Phi_{m}$ \eqref{eq:4.1.4}, 
this is equal to
\begin{align*}
\sum_{S \in \Conf} \Phi_{m} (S) \cdot A(S,T)\  
=\  \Phi_{m} \Big( \sum_{S \in \Conf}S \cdot A(S,T) \Big)
=\  \Phi_{m} (\mathcal{A}(T)). \tag*{$***$)}
\end{align*}
\vspace{-0.6cm}
\end{proof}

\subsection{A characterization of the antipode.}\label{subsec:5.5}
Equation \eqref{eq:5.4.4} 
provides formulae,
\begin{align}\label{eq:5.5.1}
 &  \iota (\mathcal{A}(T)) \mathcal{A}(T) 
\ \ =\ \ 1 \qquad {\rm for} \ T \in \Conf,\\
&\Phi_{m} \circ \iota\ \ =\ \ (\iota \widehat\otimes \cdots \widehat\otimes \iota) \circ \Phi_{m} 
\hskip1.2cm {\rm for} \ m \in \mathbb{Z}_{\ge0}.\label{eq:5.5.2}
\end{align}

\begin{proof}[Proof of \eqref{eq:5.5.1}]
Apply  \eqref{eq:5.4.1} to  $(\iota \cdot 1) \circ
\Phi_{2}(T) = \aug(T)$ \eqref{eq:4.4.2}.
\end{proof}
\begin{proof}[Proof of \eqref{eq:5.5.2}]
 It is enough to show the case $m=2$ due to \eqref{eq:4.2.1}.
Apply $\Phi_{2}$ to \eqref{eq:5.5.1}. Recalling \eqref{eq:5.4.1},
one obtains a relation
\begin{equation*}
\Phi_{2}( \iota(\mathcal{A}(T))\cdot (\mathcal{A}(T)\otimes ( \mathcal{A}(T)))\ =\ 1 \ .
\end{equation*}
Multiply $ \iota(\mathcal{A}(T))\iota(\mathcal{A}(T))$ and apply again
\eqref{eq:5.5.1} so that one obtains
\begin{align*}
\Phi_{2}(\iota(\mathcal{A}(T))) &\ =\  \iota(\mathcal{A}(T) \otimes 
\iota(\mathcal{A}(T))\\
& \ =\ (\iota \otimes \iota)\big(  \mathcal{A}(T) \otimes 
 \mathcal{A}(T)\big)\\
&\ =\ (\iota \otimes \iota) \Phi_{2} (\mathcal{A}(T)).  
\end{align*} 
Thus \eqref{eq:5.5.2} 
is true for $\mathcal{A}(T)$  ($T \in \Conf$).
Since  $\mathcal{A}(T)$ ($T \in \Conf$) span $\A\cdot\Conf$, 
which is dense in the whole algebra,  \eqref{eq:5.5.2}
holds on $\A\db[\Conf\db]$.
\end{proof}


\section{The logarithmic growth function}\label{sec:6}
The growth coefficients $A(S,T)$ in $S\in \langle T\rangle$ 
were bounded from above in \eqref{eq:5.2.1}.
However in the sequel, we also need to bound its lower terms.
This is achieved by introducing a logarithmic growth coefficient 
$M(S,T) \in \mathbb{Q}$ in $S\in \langle T\rangle$, 
and showing linear relations \eqref{eq:6.2.2}.

\subsection{The logarithmic growth coefficient}\label{subsec:6.1}

For $T \in \Conf$, define the logarithm  of the growth function:
\begin{align}\label{eq:6.1.1}
\mathcal{M}(T)\ \ :\ =\ \ & \log(\mathcal{A}(T)),\\
\intertext{in $\mathbb{Q}\db[\langle T \rangle\db]$ (cf \eqref{eq:5.1.5}
and \eqref{eq:3.6.2}). Expand $\mathcal{M}(T)$ in a series}  
\mathcal{M}(T)\ \ =&\ \ \sum_{S \in \Conf} S \cdot M(S,T).\label{eq:6.1.2}
\end{align}
The coefficient $M(S,T)$ is the {\it logarithmic growth coefficient} at 
$S\in\langle T\rangle$. 

By definition, $\mathcal{M}(T)$ does not have a constant term, i.e.\   
\begin{equation}\label{eq:6.1.3}
M([\emptyset], T)\ \ :\ =\ \ 0 \qquad {\rm for} \quad T \in \Conf.
\end{equation}

For later applications, we write the explicit relations among growth-functions 
and logarithmic growth-functions (cf. \eqref{eq:3.6.4} and \eqref{eq:3.6.5}).
\begin{align}\label{eq:6.1.4}
A(S, T)\ &= \sum_{S=S_{1}^{k_{1}} \coprod \cdots \coprod S_{m}^{k_{m}}}
\frac{1}{k_{1}! \ldots k_{m}!}M(S_{1},T)^{k_{1}} \cdots A(S_{m},T)^{k_{m}}\\
M(S, T)\ &= \sum_{S=S_{1}^{k_{1}} \coprod \cdots \coprod S_{m}^{k_{m}}}
\frac{(k_{1}+ \cdots +k_{m}-1)!(-1)^{k_{1} + \cdot 
+k_{m} -1}}{k_{1}! \ldots k_{m}!} \label{eq:6.1.5}\\
&\hskip4cm \times\ A(S_{1},T)^{k_{1}} \cdots A(S_{m},T)^{k_{m}}.\notag
\end{align}

\begin{rem}
1. From the formula, we see that for a connected $S \in \Conf_{0}$,
\vspace{-0.1cm}
\begin{equation}
\label{eq:6.1.6}
\vspace{-0.1cm}
A(S,T)\ \  =\ \  M(S, T).
\end{equation}
That is, {\it the logarithmic growth coefficients coincide with the 
growth coefficients at connected configurations.} 
This elementary fact shall be used repeatedly.

\medskip
2.
The multiplicativity of $\mathcal{A}(T)$ \eqref{eq:5.1.7} implies 
the additivity 
\begin{align}\label{eq:6.1.7}
\mathcal{M}(T_{1} \cdot T_{2})\ \  &=\ \ \mathcal{M}(T_{1}) + \mathcal{M}(T_{2})\\
\intertext{for $T_{i} \in \Conf$ and hence the additivity:}
M(S, T_{1} \cdot T_{2})\ \ &=\ \ M(S,T_{1}) 
+ M(S,T_{2})\quad \text{for $S \in \Conf.$}\tag*{$(6.1.7)^{*}$}
\end{align}

3. The invertibility \eqref{eq:5.5.1} implies
\begin{equation}
\iota (\mathcal{M}(T))\ \ =\ \ -\ \mathcal{M}(T).
\end{equation}
\end{rem}

\subsection{The linear dependence relations on the coefficients}
\label{subsec:6.2}
\begin{lemma}
The polynomial relation \eqref{eq:5.4.4} implies the linear relation:
\begin{equation}\label{eq:6.2.1}
\sum_{i=1}^{m} 1 \widehat\otimes \cdots \widehat\otimes 1 \widehat\otimes \overset{i{\rm th}}{\mathcal{M}(T)}
 \widehat\otimes 1 \widehat\otimes \cdots \widehat\otimes 1\ =\ \Phi_{m}(\mathcal{M}(T)), 
\end{equation}
on the logarithmic growth-function for $T \in \Conf$ and 
$m \in\Z_{\ge0}$.
\end{lemma}

\begin{proof}
Put $\mathcal{M}_{i}(T) := 1 \widehat\otimes \cdots \widehat\otimes 1 \widehat\otimes 
\overset{i{\rm th}}{\mathcal{M}(T)} \widehat\otimes 1 \widehat\otimes \cdots \widehat\otimes 1$
so that 
$\exp(\mathcal{M}_{i}(T)) = 1 \otimes \cdots \otimes 1 \otimes 
\mathcal{A}(T) \otimes 1 \otimes \cdots \otimes 1$.
Then \eqref{eq:5.4.4} can be rewritten as:
\begin{equation}
\exp(\mathcal{M}_{1}(T)) \cdot \ldots \cdot \exp(\mathcal{M}_{m}(T))
\ =\
\Phi_{m}(\exp(\mathcal{M}(T))) \tag*{$*$)}
\end{equation}
where the left hand side is equal to $\exp(\mathcal{M}_{1}(T) + \cdots +
\mathcal{M}_{m}(T))$ due to the commutativity of the $\mathcal{M}_{i}$ 's and 
the addition rule for $\exp$.
The right hand side of $*$) can be rewritten as 
$\Phi_{m}(\exp(\mathcal{M}(T))) = \Phi_{m}(\displaystyle
\sum^{\infty}_{n = 0}
\frac{1}{n!} \mathcal{M}(T)^{n}) = \sum^{\infty}_{n = 0}
\frac{1}{n!} \Phi_{m} (\mathcal{M}(T))^{n}$ $= \exp(\Phi_{m} (\mathcal{M}(T)))$.
By taking the logarithm of both sides,
we obtain \eqref{eq:6.2.1}.
\end{proof}

\begin{cor}
Let $m \geq 2$. For $S_{1}, \ldots, S_{m} \in \Conf_{+}$ and 
$T \in \Conf$, 
\begin{equation}\label{eq:6.2.2}
\sum_{S \in \Conf}
{\small
\begin{pmatrix}
S\\ S_{1},\ldots,S_{m}
\end{pmatrix}
}
M(S,T) = 0.
\end{equation}
\end{cor}

\begin{proof}
Expand both sides of \eqref{eq:6.2.1} in a series of the variables
$S_{i} := 1 \otimes \ldots \otimes 1 \otimes S_{i} \otimes 1 \otimes \ldots 
\otimes 1$ $(i = 1, \ldots, m)$.
Since the left hand side of \eqref{eq:6.2.1} does not have a mixed 
term $S_{1}\otimes\cdots \otimes S_m$ for $S_i\in \Conf_+$ 
and $m\ge2$, the corresponding 
coefficients in the right hand side vanish.
By \eqref{eq:4.1.1} and \eqref{eq:6.1.2}, this 
implies the formula \eqref{eq:6.2.2}.
\end{proof}

\begin{rem}
1. The formula \eqref{eq:6.2.2} is reduced to the case $m=2$ with $S_{i} \neq
\emptyset$ $(i = 1,2)$ by induction on $m$.
Recalling the composition rule (2.6) 
\begin{align*}
\sum_{S}
{\footnotesize
\begin{pmatrix}
S\\ S_{1}, \ldots, S_{m}
\end{pmatrix}
}
M(S,T) & = \sum_{S} \Big( \sum_{U \in \Conf}
{\footnotesize 
\begin{pmatrix}
U\\ S_{1}, \ldots, S_{m-1}
\end{pmatrix}
}
{\footnotesize 
\begin{pmatrix}
U\\ U, S_{m}
\end{pmatrix}
}
\Big) M(S,T)\\
& =\sum_{U \in \Conf_{+}}
{\footnotesize
\begin{pmatrix}
U\\ S_{1}, \ldots, S_{m-1}
\end{pmatrix}
}
\Big( \sum_{S}
{\footnotesize
\begin{pmatrix}
U\\ U, S_{m}
\end{pmatrix}
}
M(S,T) \Big) \\
& \quad + 
{\footnotesize
\begin{pmatrix}
\emptyset\\ S_{1}, \ldots, S_{m-1}
\end{pmatrix}
}
\Big( \sum_{S}
{\footnotesize
\begin{pmatrix}
U\\ \emptyset, S_{m}
\end{pmatrix}
}
M(S,T) \Big) 
= 0 + 0 = 0.
\end{align*}

2.
The linear dependence relations \eqref{eq:6.2.2} among $M(S,T)$'s for 
$S \in \Conf$ are the key facts of the present  paper.
The Hopf algebra structure was introduced only to deduce 
this relation.
We shall solve this relation in \eqref{eq:8.3.2} 
by use of 
kabi coefficients, which we introduce in the next paragraph 
$\S$ 7.
\end{rem}

\subsection{Lie-like elements}\label{subsec:6.3}
An element $\mathcal{M}$ satisfying \eqref{eq:6.2.1} 
has a name in Hopf algebra theory [9].

\begin{definition}
Let $\mathbb{A}$ be a commutative algebra with a unit. 
An element $\mathcal{M}$ of
$\mathbb{A}\db[\Conf\db]$ is called {\it Lie-like} 
if it satisfies the relation:
\begin{equation}\label{eq:6.3.1}
\Phi_{m}(\mathcal{M})\ \  =\ \  \sum^{m}_{i=1} 1 \widehat\otimes \cdots \widehat\otimes 1 
\overset{i{\rm th}} {\widehat\otimes \mathcal{M}\widehat \otimes} 1 \widehat\otimes \cdots \widehat\otimes 1
\end{equation}
for all $m \in \Z_{\ge0}$. This, in particular, implies the conditions
$\Phi_0(\mathcal{M})=0$ and $\iota(\mathcal{M})+\mathcal{M}=0$ 
(c.f.\ \eqref{eq:4.3.1} and \eqref{eq:4.4.2}). 
The linear combinations (over $\A$) 
of Lie-like elements are also Lie-like.
We put
\begin{equation}\label{eq:6.3.2}
\mathcal{L}_{\mathbb{A}}\ \ :\ =\ \ \{\text{all Lie-like elements in} \ 
\mathbb{A}\db[\Conf\db]\},
\vspace {-0.2cm}
\end{equation}
\vspace {-0.2cm}
and
\begin{equation}\label{eq:6.3.3}
\mathcal{L}_{\mathbb{A}, \ {\rm finite}}\ \ :\ =\ \ \{ M \in \mathcal{L}_{\mathbb{A}} 
\mid M \ \text{is of finite type}\}.
\end{equation}
\end{definition}

In this terminology, (\ref{subsec:6.2}) Lemma can be rewritten as:
{\it suppose $\mathbb{Q} \subset \mathbb{A}$, 
then one has $\mathcal{M}(T)\! \in\! \mathcal{L}_{\mathbb{A}, {\rm finite}}$} for $T\! \in\! \Conf$.

\begin{rem}
We shall see in \ref{subsec:8.4} that 
$\mathcal{L}_{\mathbb{\R}}$ is essentially 
an extension of  
$\mathcal{L}_{\mathbb{\R}, \ {\rm finite}}$ by
a space $\mathcal{L}_{\mathbb{\R}, {\infty}}$,  
which is the main objective of the present paper.
On the other hand, one has 
$\mathcal{L}_{\mathbb{A}}\cap \big(\A\cdot\Conf\big)_\mathfrak{M}
\subset\mathcal{L}_{\mathbb{A}, \ {\rm finite}}$ 
(actually equality holds, see \S8), 
since $\big(\A\cdot\Conf\big)_\mathfrak{M}$ consists only of finite 
type elements. 
\end{rem}


\section{Kabi coefficients}\label{sec:7}

We describe the inverse matrix 
of the infinite matrix
$
A\ :\ =\ (A(S,T))_{S,T\in \Conf_{0}}
$
explicitly in terms of kabi coefficients introduced in 
(\ref{subsec:7.2}).
The construction  shows that  
the inverse matrix  has only bounded number of nonzero entries 
(\ref{subsec:7.5}). This fact leads to the 
comparison of the two topologies on $\mathcal{L}_{\mathbb{A},finite}$, 
which plays a key role in the sequel in construction of the infinite space 
$\LL_{\A,\infty}$.

\subsection{The unipotency of $A$}\label{subsec:7.1}

The matrix $A$ is unipotent in the sense that i) $A(S,S)=1$ and
ii) $A(S,T)=0$ for $S\nleq T$ \eqref{eq:5.1.5}. Then a matrix
$A^{-1}:=E+A^*+A^{*2}+A^{*3}+\cdots$, where 
$E:=(\delta(U,V))_{U,V\in \Conf_0}$ and $A^*:=E-A$, is well
defined. Precisely, 
\begin{equation*}
A^{-1}(S,T)\ =\
\begin{cases}
0&\text{for $S\nleq T$}, \\
1&\text{for $S=T$},\\
\underset{k>0}{\sum}(-1)^k(\underset{S=S_0< \cdots < S_k=T}{\sum}
(\overset{k}{\underset{i=1}{\prod}}A(S_{i-1},S_i)))
&\text{for $S<T$}.
\end{cases}
\end{equation*}
The matrix $A^{-1}$ is unipotent in the same sense as $A$, and,
hence, the products $A^{-1}\cdot A$ and $A\cdot A^{-1}$ are well
defined and are equal to $E$.

\subsection{Kabi coefficients}\label{subsec:7.2}

\begin{definition}
1. A graph $\mathbb{U}$ is called a {\it kabi} over its subgraph
$\mathbb{S}$ if 
for all $x\in \mathbb{U}\backslash \mathbb{S}$, 
there exists $y\in \mathbb{S}$ 
such that $(x,y)$ is an edge.

2. Let $U\in \Conf_0$ and let $\mathbb{U}$ be a graph with 
$[\mathbb{U}]= U$. For $S\in \Conf_0$, put 
\begin{align}\label{eq:7.2.1}
\mathbb{K}(S,\mathbb{U})\  :\ =\ \ &\{\ \mathbb{S} \mid  \mathbb{S}\subset \mathbb{U}\ \text{ such that }
[\mathbb{S}]=S  \text{ and $\mathbb{U}$ is  kabi 
over  } \SS\},\\
K(S,U) \ :\ =\ \ &\# \mathbb{K}(S,\mathbb{U}).\label{eq:7.2.2}
\end{align}
We call 
$K(S,U)$ a {\it kabi-coefficient}.
The definition of the coefficient does not depend on the choice
of $\mathbb{U}$. If $K(S,U)\not= 0$, we say that $U$ has a 
kabi structure over $S$ or simply $U$ is  kabi over $S$.
\end{definition}

Directly from definition, we have
\begin{align}\label{eq:7.2.3}
K(S,U)\ &=\ 0 \ \ \text{\ for $S\nleq U$},\\
K(S,S)\ &=\ 1\ \ \text{\ for $S\in \Conf_0$}.\label{eq:7.2.4}
\end{align}

\begin{note}
The word ``kabi'' means ``mold'' in Japanese.
\end{note}

\subsection{Kabi inversion formula }\label{subsec:7.3}

\begin{lemma}
For $S\in \Conf_0$ and $T\in \Conf$, one has the formula:
\begin{equation}\label{eq:7.3.1}
\sum_{U\in \Conf_0}(-1)^{\#U-\#S}K(S,U)\cdot A(U,T)\ =\ \delta(S,T),
\end{equation}
where $\delta(S,T)$ means the \# of connected components of $T$ 
isomorphic to $S$.
\end{lemma}

\begin{proof}
The summation index $U$ on the left hand side runs over the range 
$S\le U\le T$ (otherwise $K(S,U)\cdot A(U,T)=0$).
Hence if $S\nleq T$, then the sum equals  $0$. If $S\!=\!T$,
the only term in the sum is $K(S,S)A(S,S)$ which equals $1$.

Let $S\in \Conf_0$ and $T\in \Conf$. Assume $S\le T$ and 
$S\not =T$. Let $\mathbb{T}$ be a $G$-colored 
graph with $T=[\mathbb{T}]$.
Applying the definition of $K(S,U)$ and $A(U,T)$ (cf. \eqref{eq:5.1.1}),
the left hand side of \eqref{eq:7.3.1} can be  rewritten as
\begin{align*}
&\sum_{U\in \Conf_0}(-1)^{\#U-\#S}K(S,U)\cdot \#\mathbb{A}(U,\mathbb{T})\\
=&\sum_{U\in \Conf_0}(-1)^{\#U-\#S}K(S,U)\cdot \#\{ \mathbb{U}\ \mid\
\mathbb{U}\subset \mathbb{T}\ \text{ such that } [\mathbb{U}]= U\}\\
=&\sum_{U\in \Conf_0}(-1)^{\#U-\#S}\#\left \{(\mathbb{S},\mathbb{U})
\ \mid \
\begin{matrix}
&\mathbb{S}\subset\mathbb{U}\subset\mathbb{T}\ \text{ such that } \\
& [\mathbb{S}]= S, [\mathbb{U}]= U 
\text{ and } \mathbb{U} \text{\ is kabi over } \SS
\end{matrix}\right \}
\intertext{Now we make a re-summation of this by fixing the subgraph 
$\mathbb{S}$ in $\mathbb{T}$.}
=&\sum_{\mathbb{S}\in A(S,\mathbb{T})}
\left (\sum_{U\in \Conf_0}(-1)^{\#U-\#S}\#\left \{ \mathbb{U}\mid
\begin{matrix}
&\mathbb{S}\subset \mathbb{U}\subset \mathbb{T} \ \text{ such that }
\\ 
& [\mathbb{U}]\simeq U \text{\ and }
\mathbb{S}
\text{\ is a kabi over } \SS. 
\end{matrix}\right \} \right )
\end{align*}
For a fixed subgraph $\mathbb{S}$ of $\mathbb{T}$, let 
$\mathbb{U}_{\max}$ be the maximal 
subgraph of $\mathbb{T}$ such that $\mathbb{U}_{\max}$ is
a kabi over $\SS$, i.e.\ $\mathbb{U}_{\max}$ consists of vertices of
$\mathbb{T}$, which are either in $\mathbb{S}$ or connected to 
$\mathbb{S}$ by an edge. Then a subgraph $\mathbb{U}$ of 
$\mathbb{T}$ becomes a kabi over $\mathbb{S}$, if and only if it is
a subgraph of $\mathbb{U}_{\max}$ containing $\mathbb{S}$.
Hence the sum is equal to 
\begin{align*}
&=\sum_{\mathbb{S}\in A(S,\mathbb{T})}
\Big(\sum_{\mathbb{S}\subset \mathbb{U}\subset \mathbb{U}_{\max}}
(-1)^{\#\mathbb{U}-\#\mathbb{S}}\Big)
\ \ =\sum_{\mathbb{S}\in A(S,\mathbb{T})}
\Big(\sum_{\mathbb{W}\subset \mathbb{U}_{\max} \backslash \mathbb{S}}
(-1)^{\#\mathbb{W}}\Big).
\end{align*}
where the last summation index $\mathbb{W}$ runs over all subsets of 
$\mathbb{U}_{\max} \backslash \mathbb{S}$. Hence the summation in the
parenthesis becomes $1$ or $0$ according to whether 
$\mathbb{U}_{\max} \backslash \mathbb{S}$ is $\emptyset$ or not.
It is clear that $\mathbb{U}_{\max} \backslash \mathbb{S}=\emptyset$
is equivalent to the fact that $\mathbb{S}$ is a connected 
component of $\mathbb{T}$. Hence the sum is equal to 
$ 
\ \delta(S,T).
$ 
\end{proof}

\subsection{Corollaries to the inversion formula.}\label{subsec:7.4}

The left inverse matrix of $A:=(A(S,T))_{S,T\in \Conf_0}$ is given by
\begin{equation}\label{eq:7.4.1}
A^{-1}\ =\ ((-)^{\# T-\# S}K(S,T))_{S,T\in \Conf_0}.
\end{equation}
Since the left inverse matrix of $A$ coincides with the right 
inverse, one has 
\begin{equation}\label{eq:7.4.2}
\sum_{U\in \Conf_0}(-1)^{\# T-\# U}A(S,U)\cdot K(U,T)\ =\ \delta(S,T)
\end{equation}
for $S\in\Conf_0$. 
Specializing $S$ in \eqref{eq:7.4.2} to $pt:=$ [one point graph], 
one gets,
\begin{equation}\label{eq:7.4.3}
\sum_{U\in \Conf_0}(-1)^{\# U}\# U\cdot K(U,T)\ =\ (-)^{\#T}\delta(pt,T).
\end{equation}

\subsection{Boundedness of non-zero entries of $K$}\label{subsec:7.5}

One of the most important consequences of \eqref{eq:7.4.1} is the boundedness of
the non-zero entries of the matrix $A^{-1}$, as follows.

Suppose $K(S,T)\not= 0$. Then, by definition, $T$ must have at least one
structure of kabi over $S$. This implies that for each fixed
$S$ and $q\ge 0$, there are only a finite number of $T\in \Conf_0$
with $K(S,T)\not =0$. Precisely,

\begin{assertion}
For $S\in \Conf_0$, $K(S,T)= 0$  unless 
$\# T\le \# S\cdot (q-1)+2$.
\end{assertion}

\begin{proof}
Let $\mathbb{T}$ be kabi over $\mathbb{S}$. Every vertex of
$\mathbb{S}$ is connected to at most $q$ number of points of $\mathbb{T}$.
Since $\mathbb{S}$ is connected, it has at least $\# S-1$ number 
of edges. Hence, $\#T-\#S\le \#$ \{edges connecting 
$\mathbb{S}$ and $\mathbb{T}\backslash \mathbb{S}\} \le
q\cdot \#S-2\cdot (\#S-1)$. This implies the Assertion.
\end{proof}

\begin{remark}
The above boundedness implies that $K$
induces a continuous map between the two differently completed modules of 
$\mathcal{L}_{\A,finite}$ (cf.\ \ref{subsec:8.4}).
\end{remark}


\section{Lie-like elements $\mathcal{L}_\mathbb{A}$}\label{sec:8}

Under the assumption $\mathbb{Q}\!\subset\! \mathbb{A}$, we introduce
two basis systems 
$\{\mathcal{M}(T)\}_{T\in\Conf_0}$ and $\{\varphi(S)\}_{S\in\Conf_0}$
for the module of Lie-like elements 
$\mathcal{L}_{\mathbb{A},finite}$, where the base change 
between them is given by the kabi-coefficients. 
The basis $\{\varphi(S)\}_{S\in\Conf_0}$ is compatible with the 
adic topology and gives a topological basis of 
$\mathcal{L}_\mathbb{A}$.

\subsection{The splitting map $\partial$}\label{subsec:8.1}

First, we introduce a useful but somewhat technical map $\partial$. 
One reason for its 
usefulness can be seen from the formula \eqref{eq:9.3.6}. 
For
$S\in \Conf_0$, let us define an $\mathbb{A}$-linear map
$\partial_S: \mathbb{A}\db[\Conf\db]\to \mathbb{A}$ by associating
to a series $f$ its coefficient
at $S$, \ i.e. \ $\partial_Sf:=f_S\in \mathbb{A}$ for $f$ given by 
\eqref{eq:3.2.4}. 
By the use of this, we define 
\begin{equation}\label{eq:8.1.1}
\begin{matrix}
\partial : \mathbb{A}\db[\Conf\db] &\ \longrightarrow  &
\underset{S\in \Conf_0}{\Pi}\mathbb{A}\cdot e_S .\\
f&\ \longmapsto &\ \sum_{S\in \Conf_0}(\partial_Sf)\cdot e_S
\end{matrix}
\end{equation}
Here, the right hand side is an abstract direct product module 
of rank one modules $\A\cdot e_S$ with the base $e_S$ 
for  $S\in \Conf_0$.
Let us verify that the map is well-defined.
First, define the map 
$\partial$ from the polynomial ring $\mathbb{A}\cdot \Conf$
to $\underset{S\in \Conf_0}{\oplus}\mathbb{A}\cdot e_S$.
Since $\partial(\mathcal{J}_n)\subset 
\underset{\substack{S\in \Conf_0\\ \# S\ge n}}\oplus
\mathbb{A}\cdot e_S$, 
the map is continuous with respect to the adic topology 
(3.2) on the LHS and the direct product topology on the RHS. 
Then, $\partial$ \eqref{eq:8.1.1} is obtained by
completing this polynomial map.

We note that the restriction of the map $\partial$  \eqref{eq:8.1.1} 
induces a map
\[
\partial  : \mathbb{A}\db[\Conf\db]_{finite}\longrightarrow 
\oplus_{S\in \Conf}\mathbb{A}\cdot e_S,
\]
even though the domain of this map is not a polynomial ring 
but the ring of elements of finite type (recall the definition 
in \ref{subsec:3.3}).

\subsection{Bases $\{\varphi(S)\}_{S\in\Conf_0}$ of $\mathcal{L}_{\mathbb{A},finite}$
and $\mathcal{L}_\mathbb{A}$}\label{subsec:8.2}

\begin{lemma}
Let $\mathbb{A}$ be a commutative algebra containing 
$\mathbb{Q}$. Then,

{\rm i)} The system $(\mathcal{M}(T))_{T\in \Conf_0}$ give a $\mathbb{A}$-free
basis for $\mathcal{L}_{\mathbb{A},finite}$.
\begin{equation}\label{eq:8.2.1}
\mathcal{L}_{\mathbb{A},finite}\ \ \simeq\ \ \oplus_{S\in \Conf_0}
\mathbb{A}\cdot \mathcal{M}(S).
\end{equation}

{\rm ii)} The map $\partial$ \eqref{eq:8.1.1} induces a bijection of 
$\mathbb{A}$-modules:
\begin{equation}\label{eq:8.2.2}
\partial |_{\mathcal{L}_{\mathbb{A},finite}}: 
\mathcal{L}_{\mathbb{A},finite}\ \ \simeq\ \ \oplus_{S\in \Conf_0}
\mathbb{A}\cdot e_S
\end{equation}
Put $\varphi(S):=\partial |_{\mathcal{L}_{\mathbb{A},finite}}^{-1}(e_S)$
for $S\in \Conf_0$ so that $\{\varphi(S)\}_{S\in \Conf_0}$ form
another $\mathbb{A}$-free basis of $\mathcal{L}_{\mathbb{A},finite}$.

{\rm iii)} The two basis systems $\{\mathcal{M}(S)\}_{S\in \Conf_0}$ and
$\{\varphi(S)\}_{S\in \Conf_0}$ for $\mathcal{L}_{\mathbb{A},finite}$
are related by the following formula.
\begin{align}\label{eq:8.2.3}
\mathcal{M}(T)\ =\ &\sum_{S\in \Conf_0}\varphi(S)\cdot A(S,T)\\
\varphi(S)\ =\ &\sum_{T\in \Conf_0}\mathcal{M}(T)\cdot (-1)^{\#T-\#S}K(T,S).
\label{eq:8.2.4}
\end{align}

{\rm iv)} $\mathcal{L}_{\mathbb{A},finite}$ is dense in $\mathcal{L}_{\mathbb{A}}$
with respect to the adic topology on the configuration algebra
(cf. (3.2)).

{\rm v)} The map $\partial$ induces an isomorphism of topological
$\mathbb{A}$-modules:
\begin{equation}\label{eq:8.2.5}
\mathcal{L}_\mathbb{A}\ \ \simeq\ \ 
\underset{S\in \Conf_0}{\Pi}\mathbb{A}\cdot e_s.
\end{equation} 
This means that any $\mathcal{M}\in \mathcal{L}_{\mathbb{A}}$
is expressed uniquely as an infinite sum 
\begin{equation}\label{eq:8.2.6}
\mathcal{M}\ \ =\ \ \sum_{S\in \Conf_0}\varphi(S)\cdot a_S
\vspace{-0.1cm}
\end{equation} 
for $a_S\in \mathbb{A}$
($S\in \Conf_0$).
That is, $(\varphi(S))_{S\in \Conf_0}$ is a topological basis of 
$\mathcal{L}_\mathbb{A}$.
We shall, sometimes, call $a_S$ the coefficient of $\mathcal{M}$ at $S\in\Conf_0$.
\end{lemma}
\vspace{-0.1cm}
\begin{proof}
That $\mathcal{M}(T)\in \mathcal{L}_{\mathbb{A},finite}$ for
$T\in \Conf$ is shown in (6.2) Lemma.

In the following a), b) and c), 
we prove i), ii) and iii) simultaneously.

a) {\it The restriction of the map 
$\partial $ \eqref{eq:8.1.1} on $\mathcal{L}_{\mathbb{A}}$ 
is injective.}

\begin{proof}
If $\mathcal{M}=\underset{S\in \Conf}{\sum}S\cdot 
M_S\in \mathbb{A}\db[\Conf\db]_+$ is Lie-like (6.3), then one has
\begin{equation}
\label{eq:8.2.7}
\sum_{S\in \Conf}
{\small
\begin{pmatrix}
S\\ S_1,\ldots,S_m
\end{pmatrix}
}
M_S\ =\ 0
\end{equation}
for any $S_1,\ldots,S_m\not =\emptyset$ and $m\ge 2$ (the proof is the 
same as that for \eqref{eq:6.2.2}).
We have to prove that $\partial \mathcal{M}=0$ implies $M_S=0$ for all 
$S\in \Conf$. This will be done by induction on $n(S)=\# \{$connected
components of $S\}$ as follows. The case $n(S)=1$
follows from the assumption $\partial \mathcal{M}=0$. Let
$n(S)>1$ and $S=S^{k_1}_1\sqcup\cdots \sqcup S^{k_l}_l$ be an irreducible 
decomposition of $S$ (so $S_i\in \Conf_0$ ($i=1,\cdots,l$) 
are pairwise distinct).
Apply \eqref{eq:8.2.7} for $m=k_1+\cdots +k_l (=n(S))$ and 
take $S_1,\ldots, S_1$ ($k_1$ times), $\ldots ,S_l,\ldots ,S_l$
($k_l$ times) for $S_1,\ldots S_m$.
\begin{equation}
\begin{array}{ll}
k_1!\ldots k_l! M_S+\sum_{\substack{T\in \Conf\\ n(T)<n(S)}}
{\small
\begin{pmatrix}
T\\ S_1,\ldots,S_m
\end{pmatrix}
}
M_T\ =\ 0
\tag*{$**$)}
\end{array}
\end{equation}
By the induction hypothesis, the second term in $**$) is $0$, 
and hence $M_S=0$.
\end{proof}

b) {\it For $T\in \Conf$, one has the formula:}
\begin{equation}\label{eq:8.2.8}
\partial(\mathcal{M}(T))\ \ =\sum_{S\in \Conf_0}e_S\cdot A(S,T).
\end{equation}

({\it Proof.}
Recall that $\mathcal{M}(T)=\log(\mathcal{A}(T))$ and the 
coefficients of $\mathcal{M}(T)$ and $\mathcal{A}(T)$ at a connected 
$S\in\Conf_0$ coincide (\eqref{eq:3.6.7} and \eqref{eq:6.1.6}). That is,
$\partial(\mathcal{M}(T))=\partial(\mathcal{A}(T))$. By definition,
$\partial(\mathcal{A}(T))=$ the right hand side of \eqref{eq:8.2.8}. $\Box$

c) {\it The map \eqref{eq:8.2.2} is surjective, and hence, is bijective.}

{\it Proof.}
It was shown in \S 7 that the infinite matrix 
$(A(S,T))_{S,T\in \Conf_0}$ is invertible as a unipotent matrix
(7.1). Then \eqref{eq:8.2.8} implies surjectivity. 

Using again \eqref{eq:8.2.8}, we have that the system 
$\{ \mathcal{M}(T)\}_{T\in \Conf_0}$ is $\mathbb{A}$-linearly
independent and spans $\mathcal{L}_{\A,finite}$,
i.e.\ i) holds.
The formula \eqref{eq:8.2.8} can be rewritten as \eqref{eq:8.2.3}.
Then \eqref{eq:8.2.4} follows from \eqref{eq:8.2.3} and
\eqref{eq:7.4.2} $\Box$.

Proof of iv) and v) is done in the following a), b) and c).

a) {\it $\mathcal{L}_\A$ is closed  
in $\A\db[\Conf\db]$ with respect to the adic topology},
since the co-product $\Phi_m$ is continuous (\ref{subsec:4.1}
Assertion). 
Thus: 
$(\mathcal{L}_{\A,finite})^{closure}\subset \mathcal{L}_\A$.

b) {\it The map \eqref{eq:8.2.2} is  
homeomorphic with respect to the topologies: 
the induced adic topology on the LHS and the restriction of the direct product 
topology on the RHS}.
 (To show this, 
it is enough to show the bijection:
\begin{equation}\label{eq:8.2.9}
\partial : (\mathcal{L}_{\A,finite})\cap 
\mathcal{I}_n\ \ \simeq\ \ \underset{\substack{S\in \Conf_0\\ \#S\ge n}}\oplus
\mathbb{A}\cdot e_S,
\vspace{-0.3cm}
\end{equation}
since the sets on the RHS for $n\in\Z_{\ge0}$ can be chosen as a 
system of fundamental neighborhoods for the direct 
product topology on $\underset{S\in \Conf_0}\oplus
\mathbb{A}\cdot e_S$.)

{\it Proof of} \eqref{eq:8.2.9}.
Due to the definition of the ideal $\mathcal{I}_n$
\eqref{eq:3.2.1}, the $\partial$-image of the left hand side 
is contained in the right hand side of \eqref{eq:8.2.9}. 
Thus, one has only to show surjectivity.
For $S\in \Conf_0$, let $\varphi(S)$ be the base of 
$\mathcal{L}_{\A,finite}$ such that 
$\partial (\varphi(S))=e_S$ as introduced in ii). It is enough
to show that if $\#S\ge n$ and $S\in \Conf_0$, then
$\varphi(S)$ belongs to $\mathcal{I}_n$. Expand 
$\varphi(S)=\sum U\cdot \varphi_U$.
We show that $\varphi_U\not= 0$ implies that $U$ is contained in 
the semi-group generated by $\langle S\rangle$ such that
$\# U\ge n$. More precisely, we show  
$(\begin{smallmatrix}U_1,\ldots,U_m\\ S\end{smallmatrix})\not= 0$,
where $U=U_1\sqcup \cdots \sqcup U_m$ is an irreducible
decomposition of $U$ (cf. (2.5) i)). 
The proof is achieved by induction on $m=n(U)$. 
For the case $n(U)\!=\!1$, $\varphi_U\!\not =\!0$ if and only if $U\!=\!S$ 
by the definition of $\varphi(S)$, and hence this is trivial.
If $n(U)>1$, then apply \eqref{eq:8.2.7}
similarly to $**$) for the irreducible 
decomposition of $U$. We get:
\begin{equation}
\begin{array}{ll}
k_1!\ldots k_l!\ \varphi_U+\sum_{\substack{T\in \Conf\\ n(T)<n(U)}}
{\small
\begin{pmatrix}
T\\ U_1,\ldots ,U_m
\end{pmatrix}
}
\varphi_T\ =\ 0
\tag*{$***$)}
\end{array}
\end{equation}
The fact that $\varphi_U\not= 0$
implies $\varphi_T\cdot(\begin{smallmatrix}U_1,\ldots,U_m\\ T\end{smallmatrix})\not= 0$
for some $T$. Since $\varphi_T\not=0$ with $n(T)<n(U)$, we 
apply the induction hypothesis to $T$, i.e.\ 
$(\begin{smallmatrix}T_1,\ldots,T_p\\ S\end{smallmatrix})\not= 0$
for an irreducible decomposition $T=T_1\sqcup \ldots \sqcup T_p$. 
Since $(\begin{smallmatrix}U_1,\ldots,U_m\\ T\end{smallmatrix})
\not= 0$, by composing the maps $U\to T\to S$, we conclude 
$(\begin{smallmatrix}U_1,\ldots,U_m\\ S\end{smallmatrix})
\not= 0$.
In particular, 
$U_i\!\in\! \langle S\rangle$ and $\# U\!=\!\sum \# U_i\!\ge\! \# T\!\ge\! \# S$.
This completes the proof of b). $\Box$)

c) By completing the map \eqref{eq:8.2.2}, one sees that the 
composition of the two injective maps
$(\mathcal{L}_{\mathbb{A},finite})^{closure}\subset \mathcal{L}_\mathbb{A}
\to \underset{p,q}{\varinjlim}
\underset{S\in \Conf_0}{\Pi}\mathbb{A}\cdot e_s$
is bijective. This shows that all the maps are bijective.
Hence,
$\mathcal{L}_{\mathbb{A},finite}$ is dense in $\mathcal{L}_\mathbb{A}$
and \eqref{eq:8.2.5} holds. The formula \eqref{eq:8.2.6} is another expression of
\eqref{eq:8.2.5}.

This completes the proof of the Lemma.
\end{proof}

\begin{rem}\label{rema:1}
1. It was shown in the above proof that for $S\in\Conf_0$
\begin{equation}\label{eq:8.2.10}
\varphi(S)\ \ \in\ \  \Z\db[\langle S\rangle\db]\ \cap\ \mathcal{J}_{\#S}. 
\end{equation}
In particular, $\varphi(U,S)=\delta(U,S)$ for $U\in \Conf_0$.

2. It was shown that the map $\partial|_{\mathcal{L}_{\mathbb{A},finite}}$ \eqref{eq:8.2.2} is a homeomorphism. But one
should note that \eqref{eq:8.2.1} is \textit{not} a homeomorphism.

3. In general, an element of $\mathcal{L}_\mathbb{A}$ cannot be expressed 
by an infinite sum of $\mathcal{M}(T)\ (T\in \Conf_0)$\ (cf. (9.4)).

4. The set of Lie-like elements of the localization 
$\mathbb{A}[\Conf]_\mathfrak{M}$ (cf. (4.6) Remark 4.) is equal
to $\mathcal{L}_{\mathbb{A},finite}$.
This is 
insufficient for our later application in \S10, so we employed the other 
localization \eqref{eq:3.2.2}.
\end{rem}

\subsection{An explicit formula for $\varphi(S)$}\label{subsec:8.3}

Let us expand $\varphi(S)$ for $S\in \Conf_0$ in the series:
\begin{equation}\label{eq:8.3.1}
\varphi(S)\ =\ \sum_{U\in \Conf}U\cdot \varphi(U,S)
\end{equation}
for $\varphi(U,S)\in \mathbb{Q}$. The formula \eqref{eq:8.2.3}
can be rewritten as a matrix relation
\begin{equation}\label{eq:8.3.2}
M(U,T)\ =\ \sum_{S\in \Conf_0}\varphi(U,S)\cdot A(S,T).
\end{equation}
We remark that  \eqref{eq:8.2.3} and \eqref{eq:8.3.2} 
are valid 
not only for $T\in \Conf_0$ but for all $T\in \Conf$, since both  
sides are additive with respect to $T$.

{\bf Formula.}
An explicit formula for the coefficients $\varphi(U,S)$.
\vspace{-0.3cm}
\begin{align}\label{eq:8.3.3}
&\!\!\!\sum_{\substack{U=U^{k_1}_1\sqcup \cdots \sqcup U^{k_m}_m\\
V\in\Conf,\ W\in\Conf_0}}\!\!\!\!\!\!\!\!\!\!\!\!\!\!\!
\frac{(|\underline{k}| -1)!(-1)^{|\underline{k}| -1+|W|+|S|}}
{k_1!\cdots k_m!}
{\footnotesize
\begin{pmatrix}
V\\ \underbrace{U_1,\ldots U_1}_{k_1},\ldots,\underbrace{U_m,\ldots U_{m}}_{k_m}
\end{pmatrix} 
A(V,W)K(W,S)}
.\!\!\!\!\!\!\!
\end{align}
\vglue -0.3cm
Here the summation index runs over all decompositions
$U=U_{1}^{k_{1}}\cdot\ldots\cdot U^{k_{l}}_{l}$ of $U$ in
the same manner explained at \eqref{eq:3.6.6},
where $|\underline{k}|=k_1+\cdots+k_k$.
\begin{proof}
By use of \eqref{eq:6.1.5}, rewrite the left hand side of 
(8.2.3)$^*$ into a polynomial of
$A(U_i,T)$. Then apply the product 
expansion formula \eqref{eq:5.3.1} to each monomial 
so that the left hand side is expressed linearly by  
$A(S,T)$'s. Using the invertibility of 
$\{ A(S,T)\}_{S,T\in \Conf}$ \eqref{eq:7.4.2}, one deduces  
\eqref{eq:8.3.3}.
\end{proof}

\begin{rem}
As an application of \eqref{eq:8.3.3}, we can explicitly determine the 
coefficients $\{M_U\}_{U\in\Conf}$ of any Lie-like element 
 $\mathcal{M}\ =\underset{U\in \Conf}{\sum}U\cdot M_U$
from the subsystem $\{M_S\}_{S\in\Conf_0}$ by the 
relation
$M_U\ =\sum_{S\in \Conf_0}
\varphi(U,S)\cdot M_S.$
Here, the summation index $S$ runs only over the finite set with 
$\#S\le \#U$.
\end{rem}

\subsection{Lie-like elements $\LL_{\A,\infty}$ at infinity}
\label{subsec:8.4}

 We introduce the space $\LL_{\A,\infty}$ of Lie-like elements at infinity 
for a use after \S10.

 Recall that the kabi coefficients relate the two basis systems of
 $\mathcal{L}_{\mathbb{A},finite}:
 \{ \varphi (S)\}_{S\in\Conf_0}$ and $\{\mathcal{M}(T)\}_{T\in\Conf_0}$
 (cf.(\ref{subsec:8.2}) lemma).
The  map:
\begin{align*}
\ \mathcal{L}_{\mathbb{A},finite}\quad &\ \  = \qquad \mathcal{L}_{\mathbb{A},finite}\\
K : \sum_{S\in \Conf_0}\varphi (S)\ a_S &\  \longmapsto \sum_{T\in\Conf_{0}}
\mathcal{M}(T)\sum_{S\in\Conf_0}(-1)^{\sharp T -\sharp S} K(T,S)\ a_S
\end{align*}
is the identity homomorphism between the same modules.
We define topologies on the modules of both sides: the fundamental 
system of neighborhoods of 0 are the linear subspaces spanned by the all
bases except for finite ones. Actually, the topology on the LHS coincides with the
 adic topology, which we have been studying (\ref{subsec:8.2} Lemma).
The map $K$ is continuous with respect to the topologies,
since for any base $\mathcal{M}(T)$, there are only a finite number
of bases $\varphi (S)$ whose image $K(\varphi (S))$ contains the
term $\mathcal{M}(T)$, namely $K(T,S)\neq 0$ only for such $S$ satisfying
$\sharp T \ge \frac{1}{q-1}(\sharp S-2)$, \ref{subsec:7.5} Assertion.
Let us denote by  $\overline{K}$ the map between the completed modules and call
it the $kabi~ map$.
\begin{equation}\label{eq:8.4.1}
\overline{K}\ :\ \mathcal{L}_{\mathbb{A}}\ \longrightarrow 
\prod_{T\in\Conf_{0}} \mathbb{A}\cdot\mathcal{M}(T).
\end{equation}
We consider the set of Lie-like elements which are annihilated by the 
kabi map:
\begin{equation} \label{eq:8.4.2}
\mathcal{L}_{\mathbb{A},\infty}\ \ :\ =\  
\ker(\overline{K}),
\end{equation}
and call it the space of {\it Lie-like elements at infinity}.
In fact, $\mathcal{L}_{\mathbb{A},\infty}$
 does not contain a non-trivial finite type element, i.e.\ 
$\mathcal{L}_{\mathbb{A},finite} \cap \mathcal{L}_{\mathbb{A},\infty}
 =\{0\}. 
$
However, the direct sum $\mathcal{L}_{\mathbb{A},finite}\oplus \mathcal{L}_{\mathbb{A},\infty}$ 
is a small submodule of $\mathcal{L}_{\mathbb{A}}$, and {\it one 
looks for a  submodule, 
say  $\mathcal{L}'$, of $\mathcal{L}_{\mathbb{A}}$ 
containing $\mathcal{L}_{\mathbb{A},finite}$, with a splitting }
$\mathcal{L}_{\mathbb{A}}=\mathcal{L}'\oplus \mathcal{L}_{\mathbb{A},\infty}$.
However, there is some difficulty in finding
such $\mathcal{L}'$ for general $\A$: 
an {\it infinite sum} 
 $\sum_{T\in\Conf_{0}}a_T\mathcal{M}(T)$ $\in Im(\overline{K})$ 
never converges  in $\mathcal{L}'\ (\simeq 
Im(\overline{K}))$ with respect to the adic topology. 
We shall come back to this problem in (\ref{subsec:10.2}) 
for the case $\A=\R$,
where the classical topology plays the crucial role.


\section{Group-like elements $\frak{G}_\mathbb{A}$}\label{sec:9}

We determine the groups $\frak{G}_\mathbb{A}$ and
$\frak{G}_{\mathbb{A},finite}$ of group-like elements in $ \mathbb{A}\db[\Conf\db]$ and 
$ \mathbb{A}\db[\Conf\db]_{finite}$, respectively, if 
$\mathbb{A}$ is $\Z$-torsion free. In particular, if $\A=\Z$, 
the group $\frak{G}_{\mathbb{Z},finite}$ is, 
 by the correspondence $\mathcal{A}(S) \leftrightarrow S$, 
isomorphic to $\langle\Conf \rangle=$the abelian group associated to 
the semi-group $\Conf$, 
and it forms a ``lattice in the continuous group'' $\frak{G}_{\R}$. 
Then, we introduce the set $\mathrm{EDP}$ of equal division 
points inside the positive cone in $\frak{G}_{\R}$
spanned by the basis $\{\mathcal{M}(S)\}$.

\subsection{$\frak{G}_{\A,finite}$ and $\frak{G}_\A$ for the case $\Q\subset \A$}\label{subsec:9.1}
We start with a general fact:
{\it Assume $\mathbb{Q}\subset \mathbb{A}$.
Then one has isomorphisms:}
\begin{equation}\label{eq:9.1.1}
\begin{matrix}
\exp\ :\ &\!\!\!\! \mathcal{L}_{\mathbb{A}}\ \ \ \ \  \ \simeq\ \ \ \  \frak{G}_\mathbb{A}.\\
\exp\ :\ &\mathcal{L}_{\mathbb{A},finite}\ \ \simeq\ \ \frak{G}_{\mathbb{A},finite}.\\
\end{matrix}
\end{equation}
\begin{proof}
Since $\aug(g)=1$, $\log(g)$ \eqref{eq:3.6.2} is well defined for 
$\mathbb{Q}\subset \mathbb{A}$. That $g$ is group-like 
\eqref{eq:5.4.1} implies that $\log(g)$ is Lie-like and belongs to 
$\mathcal{L}_\mathbb{A}$ (cf. proof of (6.2) Lemma). Then $g$ is
of finite type, if and only if $\log (g)$ is so (cf. (3.6)). 
Thus \eqref{eq:9.1.1} is shown. The homeomorphism
follows from that of exp (3.6).
\end{proof}

\subsection{Generators of $\frak{G}_{\A,finite}$ and $\frak{G}_{\A}$ 
for a $\Z$-torsion free $\A$.}\label{subsec:9.2}
\begin{lemma}
Let $\mathbb{A}$ be a commutative $\mathbb{Z}$-torsion-free
algebra with unit.

{\rm i)} Any element $g$ of $\frak{G}_{\mathbb{A},finite}$ is uniquely
expressed as
\begin{equation}\label{eq:9.2.1}
g\ =\ \prod_{i\in I}\mathcal{A}(S_i)^{c_i}
\end{equation}
for $S_i\!\in\! \Conf_0$ and $c_i\!\in\! \mathbb{A}$ ($i\!\in\! I$) 
with $\#I\!<\!\infty$. 
That is, one has an
isomorphism:
\begin{equation}\label{eq:9.2.2}
 \quad  
\langle \Conf\rangle \otimes_\mathbb{Z}\mathbb{A}\  \simeq\  \frak{G}_{\mathbb{A},finite}
\ ,\ \ S \leftrightarrow  \mathcal{A}(S)\ ,\!\!\!
\end{equation}
where $\langle \Conf\rangle$ is the group associated to the semi-group $\Conf$.

{\rm ii)} $\frak{G}_{\mathbb{A},finite}$ is dense in $\frak{G}_\mathbb{A}$ with respect to the adic topology.

{\rm iii)} We have the natural inclusion:
\begin{equation}\label{eq:9.2.3}
\{\exp(\varphi(S))\mid S\in \Conf_0\}\subset \frak{G}_{\mathbb{Z},finite}
\end{equation}
The set $\{ \exp(\varphi(S))\}_{S\in \Conf_0}$
is a topological free generating system of $\frak{G}_\mathbb{A}$.
This means that any element $g$ of $\frak{G}_\mathbb{A}$ 
is uniquely
expressed as an infinite product:
\begin{equation}\label{eq:9.2.4}
\prod_{S\in \Conf_0}\exp(\varphi(S)\cdot a_S)
\ \ :\ =\ \ \lim_{n\to \infty}
({\prod}_{\substack{S\in \Conf_0\\ \# S<n}}\exp(\varphi(S)\cdot a_S))
\end{equation}
for some $a_S\in \mathbb{A}$ ($S\in\Conf_0$).  
\vspace{-0.4cm}
\end{lemma}
\begin{proof}
If $\mathbb{Q}\subset \mathbb{A}$, then due to the isomorphisms
\eqref{eq:9.1.1} and \eqref{eq:6.1.1}, the Lemma is reduced to 
the corresponding statements for 
$\mathcal{L}_\mathbb{A}$ and $\mathcal{L}_{\mathbb{A},finite}$ 
in (\ref{subsec:8.2}) Lemma, where, due to \eqref{eq:8.2.4}, \eqref{eq:8.2.10} 
and the integrality of kabi $K$, 
we have
\[
\quad \exp(\varphi(S))\ =\! \prod_{T\in \Conf_0}\mathcal{A}(T)^{ (-1)^{\#T-\#S}K(T,S)}\ \in \frak{G}_{\Z,finite}\cap \{ 1+\mathcal{J}_{\#S}\},\!\!\!
\vspace{-0.2cm}
\]
where we note $\mathcal{A}(T)\in \frak{G}_{\Z,finite}$ (c.f.\ \eqref{eq:5.1.6} and 
\eqref{eq:5.4.4}).

Assume $\mathbb{Q}\not\subset\mathbb{A}$ and let
$\tilde{\mathbb{A}}$ be the localization of $\mathbb{A}$ with respect to 
$\mathbb{Z}\backslash\{ 0\}$. Since $\mathbb{A}$ is torsion
free, one has an inclusion $\mathbb{A}\!\subset\! \tilde{\mathbb{A}}$,
which induces inclusions $\frak{G}_\mathbb{A}\!\!\subset\! \frak{G}_{\tilde{\mathbb{A}}}$
and $\frak{G}_{\mathbb{A},finite}\!\!\subset\! \frak{G}_{\tilde{\mathbb{A}},finite}$, and
the Lemma is true for $\frak{G}_{\tilde{\mathbb{A}},finite}$ 
and $\frak{G}_{\tilde{\mathbb{A}}}$.

i)  Let us express an element
$g\in \frak{G}_{\mathbb{A},finite}$ as 
$\underset{i\in I}{\Pi}\mathcal{A}(S_i)^{c_i}$, where
$c_i\in \tilde{\mathbb{A}}$ for $i\in I$ and $\# I<\varpropto$.
We need to show that $c_i\in \mathbb{A}$ for $i\in I$. Suppose not.
Put $I_1:=\{ i\in I: c_i\notin \mathbb{A}\}$ and let $S_1$ be a 
maximal element of $\{ S_i: i\in I_1\}$ with respect to the 
partial ordering $\le $. Put 
$g_1 :=\underset{i\in I_1\backslash \{ 1\}}{\Pi}\mathcal{A}(S_i)^{c_i}$
and 
$g_2 :=\underset{i\in I\backslash I_1}{\Pi}\mathcal{A}(S_i)^{c_i}$.
Then
$
g_1\mathcal{A}(S_1)^{c_1} = g\cdot g_2^{-1}\ \in 
\frak{G}_{\mathbb{A},finite}$. In the left
hand side, $g_1$ does not contain the term $S_1$, whereas
$\mathcal{A}(S_1)^{c_1}$ contains the term $c_1S_1$. Hence, the left
hand side contains the term $c_1S_1$.

ii)
Let any $g\in \frak{G}_\mathbb{A}$ be given. For a fixed integer 
$n\in \mathbb{Z}_{\ge0}$, we calculate
\begin{align*}
\log(g)\ &=\sum_{S\in\Conf_0}\varphi(S)\cdot a_S
\qquad \text{for $a_S\in \tilde{\mathbb{A}}$\qquad  (c.f.\ \eqref{eq:8.2.6})}\\
&=\sum_{T\in\Conf_0}\mathcal{M}(T)\cdot c_{T,n}+R_n,\quad \text{where}\\
c_{T,n}\ \ :\ =\ &{\sum}_{\substack{S\in\Conf_0\\ \#S<n}}
(-1)^{\#T-\#S}K(T,S)\cdot a_S\in \tilde{\mathbb{A}},\tag*{$*)$}
\quad (c.f.\ \eqref{eq:8.2.4})\\
R_n\ \ :\ =\ &{\sum}_{\substack{S\in\Conf_0\\ \#S\ge n}}\varphi(S)\cdot a_S.
\tag*{$**)$}
\end{align*}
Here we notice that

$*$)\ \quad  $c_{T,n}\not= 0$ implies $\#T<n$, since $K(T,S)\not= 0$ implies $T\le S$
\quad \eqref{eq:7.2.3}.

$**$)\quad $R_n\in \mathcal{J}_n$, since $\#S\ge n$ implies 
$\varphi(S)\in \mathcal{J}_n$\quad \eqref{eq:8.2.10}.

\noindent
Therefore
\vspace{-0.2cm}
\begin{align*}
g\ =\ &{\sum}_{\substack{T\in\Conf_0\\ \#T<n}}
\mathcal{A}(T)^{c_{T,n}}\cdot \exp(R_n)\\
=\ &{\sum}_{\substack{T\in\Conf_0\\ \#T<n}}
\mathcal{A}(T)^{c_{T,n}}\quad \mod \mathcal{J}_n.
\end{align*}
Let us show that $c_{T,n}\in \mathbb{A}$ for all $T\in \Conf_0$. 
Suppose not, and let $T_1$ be a maximal element of 
$\{ T\in \Conf_0: \#T<n$ and $c_{T,n}\not\in \mathbb{A}\}$. 
Then similar to the proof of i), the coefficient of $g$ at
$T_1\equiv c_{T_1,n} \mod \mathbb{A} \not\equiv 0
\mod \mathbb{A}$. This is a contradiction. Therefore, 
$c_{T,n}\in \mathbb{A}$ for all $T$ and hence,  
$g\in \frak{G}_{\mathbb{A},finite}\mod \mathcal{J}_n$.

iii) Applying \eqref{eq:7.4.2} to the relation $*$) in the proof of ii), 
one gets
\[
a_S\ ={\sum}_{\substack{T\in\Conf_0\\ \#T<n}}
A(S,T)\cdot c_{T,n}
\]
for $\#S<n$. Here the right hand side belongs to $\mathbb{A}$
due to the proof ii). On the other hand, the left hand side
$(= a_S)$ does not depend on $n$. Hence, by moving $n\in \mathbb{Z}_{\ge0}$,
one has proven that $a_S\in \mathbb{A}$ for all $S\in \Conf_0$. 
\end{proof}

\subsection{Additive characters on $\frak{G}_\mathbb{A}$}\label{subsec:9.3}

\begin{definition}\!\!
An {\it additive character} on $\frak{G}_{\A}$ is an additive homomorphism
\begin{equation}\label{eq:9.3.1}
\mathcal{X}\ :\  \frak{G}_\mathbb{A}\  \longrightarrow \ \mathbb{A},
\end{equation}
which is continuous with respect to the adic topology on $\frak{G}_\mathbb{A}$
such that
\[
\mathcal{X}(g^a)\ =\ \mathcal{X}(g)\cdot a
\]
for all $g\in \frak{G}_\mathbb{A}$ and $a\in \mathbb{A}$. 
The continuity of $\mathcal{X}$ \eqref{eq:9.3.1} is
equivalent to the statement that there exists $n\ge0$ such that 
$\mathcal{X}(\exp(\varphi(S)))=\mathcal{X}(1)=0$ for 
$S\in \mathcal{J}_n\cap \Conf_0$. Hence it is equivalent to
$\#\{ S\in \Conf_0: 
\mathcal{X}(\exp(\varphi(S)))\not =0\}<\infty$.

The set of all
additive characters will be denoted by
\begin{equation}\label{eq:9.3.2}
\Hom_\mathbb{A}(\frak{G}_\mathbb{A},\mathbb{A}).
\end{equation}
\end{definition}

\begin{assertion}
1. For any fixed $U\in \Conf_0$, the correspondence
\begin{equation}\label{eq:9.3.3}
\mathcal{X}_U\ :\ \mathcal{A}(S)\in \frak{G}_{\mathbb{Z},finite}
\ \longmapsto\
A(U,S)\in \mathbb{Z}
\end{equation}
extends uniquely to an additive $\mathbb{A}$-character on $\frak{G}_\mathbb{A}$,
denoted by $\mathcal{X}_U$. Then
\begin{equation}\label{eq:9.3.4}
\quad \quad \mathcal{X}_U(\exp(\varphi(S)))\ =\ \delta(U,S)
\qquad \text{ for $U,S\in \Conf_0$.}
\end{equation}

2. There is a natural isomorphism
\begin{equation}\label{eq:9.3.5}
\begin{array}{llllllll}
\Hom_\mathbb{A}(\frak{G}_\mathbb{A},\mathbb{A})&\quad\simeq&\!&\!
\underset{U\in \Conf_0}\otimes 
\mathbb{A}\cdot \mathcal{X}_U\\
\mathcal{X}&\ \longmapsto & &\sum_{U\in \Conf_0}\mathcal{X}_U(\exp(\varphi(S)))
\cdot \mathcal{X}_U.
\end{array}
\end{equation}

3. If $\mathbb{Q}\subset \mathbb{A}$, then for any 
$\mathcal{M}\in\mathcal{L}_\mathbb{A}$ and $U\in \Conf_0$ one has
\begin{equation}\label{eq:9.3.6}
\mathcal{X}_U(\exp(\mathcal{M}))\ =\ \partial _U\mathcal{M}.
\end{equation}
\end{assertion}

\begin{proof}
1. First we note that $A(U,S)$ for fixed $U\in \Conf_0$ is
additive in $S$ \eqref{eq:5.1.8}, so that $\mathcal{X}_U$ naturally
extends to an additive homomorphism on $\frak{G}_{\mathbb{A},finite}$.
For continuity (i.e.\ the finiteness of
$S$ with $\mathcal{X}_U(\exp(\varphi(S)))\not = 0$), 
it is enough to show \eqref{eq:9.3.4}. Recalling \eqref{eq:8.2.4}
and \eqref{eq:7.4.2}, this proceeds as:
\begin{align*}
\mathcal{X}_U(\exp(\varphi(S))\ =\ &\ \mathcal{X}_U(\exp({\sum}_{T\in \Conf_0}
\mathcal{M}(T)(-1)^{\#T-\#S}K(T,S)))\\
=\ &{\sum}_{T\in \Conf_0}\mathcal{X}_U(\exp(\mathcal{M}(T)))\cdot
(-1)^{\#T-\#S}K(T,S)\\
=\ &{\sum}_{T\in \Conf_0}\mathcal{X}_U\mathcal{A}(T)\cdot 
(-1)^{\#T-\#S}K(T,S)\\
=\ &{\sum}_{T\in \Conf_0}A(U,T)\cdot (-1)^{\#T-\#S}K(T,S)
\ \ =\ \delta(U,S).
\end{align*}

2. The continuity of $\mathcal{I}$ implies the sum in the
target space is finite.

3. Both sides of \eqref{eq:9.3.6} 
take the same values for the
basis $(\varphi (S))_{S \in \operatorname{Conf}_{0}}$.
\end{proof}

\subsection{Equal division points of $\frak{G}_{\mathbb{Z},finite}$}
\label{subsec:9.4}

Recalling $\langle \Conf \rangle \simeq \frak{G}_{\mathbb{Z},finite}
$ \eqref{eq:9.2.2}, we regard 
$\langle \Conf \rangle$ as a ``lattice'' 
in $\frak{G}_{\mathbb{R},finite}$. 
In the positive rational cone
$\frak{G}_{\mathbb{Q},finite}\cap \big(\prod_{S\in\Conf_0}\mathcal{A}(S)^{\R_{\ge0}}\big)$, we consider a particular point, which we call
the {\it equal division point} for $S \in \Conf$:
\begin{equation}\label{eq:9.4.1}
\mathcal{A}(S)^{1/ \sharp (S)}.
\end{equation}
Here, the exponent $1/ \sharp (S)$ is chosen so that we get the normalization:
\begin{equation}\label{eq:9.4.2}
\mathcal{X}_{pt} \Bigl( \mathcal{A}(S)^{1/ \sharp (S)}\Bigr)\ =\ 1.
\end{equation}
The set of all equal division points is denoted by 
\begin{equation}\label{eq:9.4.3}
\mathrm{EDP}\ :=\ \{\mathcal{A}(S)^{1/ \sharp (S)}
\mid S \in \Conf \}.
\end{equation}

The formulation of \eqref{eq:9.4.1} is inspired from the free
energy of Helmholtz in statistical mechanics.
Instead of treating equal division points in the form \eqref{eq:9.4.1}
in $\frak{G}_\mathbb{R}$, we shall treat their logarithms in 
$\mathcal{L}_\mathbb{R}$ in the next paragraphs.

\subsection{A digression to $\mathcal{L}_\mathbb{A}$ with $\mathbb{Q} 
\not\subset\mathbb{A}$}
\label{subsec:9.5}

We have determined the generators of $\frak{G}_{\A,finite}$ and $\frak{G}_{\A}$ without assuming 
$\Q\subset \A$ but assuming only $\Z$-torsion freeness of $\A$. 
The following Assertion seems to suggest that the 
Lie-like elements behave differently 
to the group-like elements.
However we do not pursue this subject any further in the present paper.

\begin{assertion}
Let $\mathbb{A}$ be a commutative algebra
with unit.
If there exists a prime number 
$p$ such that $\mathbb{A}$ is $p$-torsion free and 
$1/p \notin \mathbb{A}$, then $\mathcal{L}_\mathbb{A}$ is divisible by $p$ 
(i.e. $\mathcal{L}_\mathbb{A}=p \mathcal{L}_\mathbb{A}$).
In particular, if $\mathbb{A}$ is noetherian, 
$\mathcal{L}_\mathbb{A}  = \{ 0\}$.
\end{assertion} 

\begin{proof}[A sketch of the proof]
Consider an element $\mathcal{M}=\sum_{U\in\Conf}U\cdot M_U \in \mathcal{L}_{\mathbb{A}}$.
As an element of $\mathcal{L}_{\tilde{\mathbb{A}}}$, it can be  
expressed as $\textstyle\sum\limits_{S \in \operatorname{\Conf}_{0}} 
\varphi (S) \cdot a_S$ where $a_S\! =\! \partial_{S}\mathcal{M}
\!=\! M_S\!\in\!\mathbb{A}$ 
for $S \in \Conf_{0}$.
Recall the expression \eqref{eq:8.3.3} for $\varphi(U,S)$ ($U\in\Conf$) 
and the remark following it. We see that $ M_U$ is expressed as:

$$\kern-8ex\sum_{\kern8ex\substack{U=U^{k_1}_1\sqcup \cdots \sqcup U^{k_m}_m\\
V\in\Conf,\ W\in\Conf_0\\S\in \Conf_0}}\kern-10ex
\frac{(-1)^{|\underline{k}| -1+|W|+|S|}(|\underline{k}| -1)!}
{k_1!\cdots k_m!}
{\scriptsize
\begin{pmatrix}
[5pt] V\\ \underbrace{U_1,\ldots, U_1}_{k_1},\ldots,\underbrace{U_m,\ldots, U_{m}}_{k_m}
\end{pmatrix}}
A(V,W)K(W,S)a_S.
$$

Apply this formula for $U = T^{p}$ for a fixed $T\in\Conf_0$.
The summation index set is 
$\{ (k_1,k_2,\ldots )\! \in\! (\mathbb{Z}_{\ge 0})^{\mathbb{Z}_{\ge1}}\mid
p\! =\! \sum_{i\ge1}i \cdot k_{i}\}$,
as explained in \ref{subsec:3.6} {\it Example}.
Except for the case $k_1 = p$ and $k_i = 0~ (i > 1)$,
the denominator $k_1 !\ldots k_m !$ is a product of prime 
numbers smaller than $p$.
The coefficient 
${\scriptsize
\begin{pmatrix}
V\\ \underbrace{U_1,\ldots U_1}_{k_1},\ldots,\underbrace{U_m,\ldots U_{m}}_{k_m}
\end{pmatrix}}$
for this case
(i.e.\ $k_1 = p, k_i = 0~ (i > 1)$  and for $V=[\mathbb{V}]$ 
is equal to the cardinality of the set 
$\{ ({\mathbb{U}_{1}},\ldots ,{\mathbb{U}_{p}})\mid
\mathbb{U}_{i}$
is a subgraph of $\mathbb{V}$ such that
$[\mathbb{U}_{i}]=T$ and
$\cup_{i=1}^p\mathbb{U}_i = \mathbb{V}\}$.
Since the cyclic permutation of 
${\mathbb{U}_{1}},\ldots ,{\mathbb{U}_{p}}$ acts on the set, and 
the action has no fixed points except for $V=T$,
we see that the covering coefficient is
divisible by $p$ except for the case $V=T\in \Conf_0$. 
In that case $\sum_{W\in\Conf_0}(-1)^{|W|+|S|}A(T,W)K(W,S)=\delta(T,S)$.
Therefore
$ 
\frac{(-1)^p}{p}a_{T}\ \equiv\ 0\ 
\bmod \ \mathbb{A}_{\text{loc}}
$ 
where $\mathbb{A}_{\text{loc}}$ is the localization of the algebra
$\mathbb{A}$ with respect to the 
prime numbers smaller than $p$.
Hence 
$a_{T} \in p\mathbb{A}_{\text{loc}}\cap\mathbb{A} = p\mathbb{A}$. 
\end{proof}


\section{Accumulation set 
of logarithmic equal division points}\label{sec:10}

We consider the space of Lie-like elements $\LL_{\R}$ over 
the real number field $\R$ which is equipped with the classical topology. 
The set in $\mathcal{L}_{\mathbb{R}}$ 
of accumulation points of the logarithm of $\mathrm{EDP}$ (\ref{subsec:9.4}), 
denoted by $\Omega:=\overline{\log (\operatorname{EDP})}$,  
becomes a compact convex set. 
We decompose  $\Omega=\overline{\log (\operatorname{EDP})}$ 
into a join of
the finite (absolutely convergent) 
part $\Omega_{abs}:=\overline{\log (\operatorname{EDP})}_{abs}$ 
and the infinite part 
$\Omega_\infty:=\overline{\log (\operatorname{EDP})}_\infty$.

\subsection{The classical topology on $\mathcal{L}_\mathbb{R}$}
\label{subsec:10.1}

We equip the $\mathbb{R}$-vector space 
\begin{equation}\label{eq:10.1.1}
\mathcal{L}_\mathbb{R} \ =\
\varprojlim_{n}
\mathcal{L}_\mathbb{R} / 
\overline{\mathcal{J}_{n}}\cap\mathcal{L}_\mathbb{R}
\end{equation}
with the {\it classical topology} defined by the projective limit 
of the classical topology on the finite quotient $\mathbb{R}$-vector spaces.
Since the quotient spaces are 
\begin{equation*}
\mathcal{L}_\mathbb{R}/
\overline{\mathcal{J}_{n}}\cap\mathcal{L}_\mathbb{R} 
\ \simeq\ 
\oplus_{S\in\Conf_0,\#S<n}\R\cdot\varphi(S) 
\ \simeq \ \mathcal{L}_{\mathbb{R},finite}/
\mathcal{J}_{n}\cap\mathcal{L}_{\mathbb{R},finite},
\end{equation*}
we see that 1) $\mathcal{L}_\mathbb{R}$ is  homeomorphic to the direct 
product $\prod_{S\in\Conf_{0}}\mathbb{R}\cdot\varphi (S)$ (recall 
\eqref{eq:8.2.5}),  and 
2) $\mathcal{L}_{\mathbb{R},finite}\simeq 
\oplus_{S\in\Conf_0}\R\cdot\varphi(S)$ is dense in
$\mathcal{L}_\mathbb{R}$ with respect to the classical topology.
That is,
{\it the classical topology on $\LL_{\R}$ is the  
topology of the coefficient-wise convergence with respect 
to the basis $\{\varphi(S)\}_{S\in \Conf_0}$}. 
It is weaker than the adic topology.

Similarly, we equip $\mathbb{R}\db[\Conf\db]$ with the classical topology 
defined by
\begin{equation}\label{eq:10.1.2}
\mathbb{R}\db[\Conf\db]\ =\
\varprojlim_{n} \ 
\mathbb{R}\cdot\Conf/
\mathcal{J}_{n} \ =\   \prod_{S\in\Conf}
\mathbb{R}\cdot S .
\end{equation}
So, the classical topology on $\mathbb{R}\db[\Conf\db]$ is 
the same as the topology of coefficient-wise convergence with respect to
the basis $\{S\}_{S\in\Conf}$.
The next relation ii) between the two topologies
\eqref{eq:10.1.1} and \eqref{eq:10.1.2} is a consequence of \eqref{eq:8.3.3}.

\begin{assertion}
{\rm i)} The product and coproduct on 
$\mathbb{R}\db[\Conf\db]$ are
continuous with respect to the classical topology.

{\rm ii)} The classical topology on $\mathcal{L}_\mathbb{R}$ is
homeomorphic to the topology induced from that on
$\mathbb{R}\db[\Conf\db]$.

{\rm iii)} Let us equip $\frak{G}_\mathbb{R}$ with the classical topology 
induced from that on $\mathbb{R}\db[\Conf\db]$. Then
$
\exp : \mathcal{L}_\mathbb{R} \rightarrow  \frak{G}_\mathbb{R}
$
is a homeomorphism.
\end{assertion}

\begin{proof}
i) The product and coproduct are continuous with respect to
the adic topology (cf. (\ref{subsec:3.2}) and (\ref{subsec:4.2})),
which implies the statement.

ii) For a sequence in $\mathcal{L}_\mathbb{R}$,
we need show the equivalence of convergence in
$\mathcal{L}_\mathbb{R}$ and in 
$\mathbb{R}\db[\Conf\db]$.
This is true due to \eqref{eq:8.3.3}.

iii) The maps $\exp$ and $\log$  are bijective (cf. (\ref{subsec:9.2}) Assertion)
and homeomorphic with respect to the adic topology, which implies the statement.
\end{proof}

\subsection{Absolutely convergent sum in $\mathcal{L}_{\R}$}
\label{subsec:10.2}

Recall the problem posed in \ref{subsec:8.4}: find a subspace of $\mathcal{L}_{\A}$
containing $\mathcal{L}_{\A,finite}$ which is complementary to 
the subspace at infinity  $\mathcal{L}_{\A,\infty}$ \eqref{eq:8.4.2}.
In the present paragraph, we answer this problem 
for the case $\A=\R$ by introducing a sufficiently large submodule 
$\mathcal{L}_{\R,abs}$, which contains $\mathcal{L}_{\mathbb{R},finite}$
but does not intersect with $\mathcal{L}_{\mathbb{R},\infty}$ 
so that we obtain a splitting submodule 
$\mathcal{L}_{\R,abs}\oplus\mathcal{L}_{\mathbb{R},\infty}$ 
of $\mathcal{L}_{\R}$.

\begin{definition}
We say a formal sum $\sum_{T\in \Conf_0}a_T\mathcal{M}(T)
\in\prod_{T\in\Conf_{0}} \mathbb{\R}\cdot\mathcal{M}(T)$  
is {\it absolutely convergent} if, for any $S\!\in\! \Conf$,
the sum $\sum_{T\in\Conf_0} a_{T}M(S,T)$ of 
its coefficients at $S$ is absolutely 
convergent, i.e.\
$\sum_{T\in \Conf_0}|a_T|M(S,T)\!<\!\infty$ for all $S\!\in\!\Conf$.
Then, any series 
$\sum_{i=1}^\infty a_{Ti}\mathcal{M}(T_i)$ defined by 
any linear ordering $T_1\!<\!T_2\!<\!\cdots$ of the index set  $\Conf_0$  
converges in $\LL_{\R}$ to the same element with respect to the classical topology.
We denote the limit by  
$\sum_{T\in \Conf_0}^{abs}a_T\mathcal{M}(T)$. 
Define the space of absolutely convergent elements: 
\begin{equation}
\label{eq:10.2.1}
\begin{array}{ll}
\mathcal{L}_{\R,abs}\ :=\ \{ \text{ all absolutely convergent sums 
 }\sum_{T\in \Conf_0}^{abs}a_T\mathcal{M}(T)\ \} . 
\end{array}
\end{equation}
\end{definition}
\noindent 
By definition, $\mathcal{L}_{\R,abs}$ 
is an $\R$-linear subspace of $\LL_{\R}$ such that 
$\mathcal{L}_{\R,abs}\cap \mathcal{L}_{\R,\infty}=\{0\}$ and 
$\mathcal{L}_{\R,abs}\supset \mathcal{L}_{\R,finite}$. 
Hence, the restriction $K|_{\mathcal{L}_{\R,abs}}$ of 
the kabi-map \eqref{eq:8.4.1} is injective.
 We give a criterion for the absolute convergence, 
which guarantees that $\mathcal{L}_{\R,abs}$ will be large enough for 
our purpose \eqref{eq:10.4.3}.

\begin{assertion} A formal sum $\sum_{T\in \Conf_0}\!a_T\mathcal{M}(T)$ is
absolutely convergent if and only if the sum 
 $\sum_{T\in \Conf_0}|a_T| \#(T)$ is convergent. 
 The $\mathcal{L}_{\R,abs}$ is a Banach space
with respect to the norm 
$\big|\sum_{T\in \Conf_0}^{abs}a_T\mathcal{M}(T)\big|:=
\sum_{T\in \Conf_0}|a_T| \#(T)$.
\end{assertion}
\vspace{-0.2cm}
\begin{proof} The coefficient of $\mathcal{M}(T)$ at $[$one point graph$]$ is equal to 
$\#(T)$. So absolute convergence implies
the convergence of $\sum_{T\in \Conf_0}|a_T| \#(T)$. 

Conversely, under this assumption, let us show the absolute convergence of 
the sum  $\sum_{T\in\Conf_0} a_{T}M(S,T)$ for any $S\in \Conf$. 
We prove this by induction on  $n(S)=$ the number of connected 
components of $S$. 
If $S$ is connected (i.e.\ $n(S)=1$), then $A(S,T)=M(S,T)$ and by the use of 
\eqref{eq:5.2.1}, we have
$\sum_{T\in\Conf_0}|a_{T}|M(S,T)
\le (\sum_{T\in\Conf_0}  |a_{T}|\#T)\frac{(q-1)^{\#S-1}}{\#\Aut(S)}
$
which converges absolutely.
If $S$ is not connected, decompose it into connected components 
as $S=\prod_{i=1}^m S_i$ and apply \eqref{eq:6.2.2}.
Since  
${\footnotesize 
\begin{pmatrix}
S'\\ S_{1},\ldots,S_{m}
\end{pmatrix}}
\not=0$ implies either $n(S')<n(S)$ or $S'=S$,  
$M(S,T)$ is expressed as a finite linear combination of $M(S',T)$ 
for $n(S')<n(S)$ (independent of $T$). 
We are now done by the  induction hypothesis.
\end{proof}

\subsection{Accumulating set  $\Omega:=\overline{\log(\text{EDP})}$}
\label{subsec:10.3}

Recall that an equal dividing point in
$\frak{G}_\mathbb{Q}$ \eqref{eq:9.4.1} is, by definition, an element 
of the form
$\mathcal{A}(S)^{1/\sharp (S)}$ for a
$S\in \Conf_+$.
Let us consider the set in $\mathcal{L}_\mathbb{Q}$ of their logarithms 
(by use of the  homeomorphism in \ref{subsec:10.1} Assertion iii):
\begin{equation}\label{eq:10.3.1}
\log (\text{EDP} )\ :\ =\ \ 
\{\mathcal{M}(T)/\sharp T \mid T\in \Conf_+\}
\end{equation}
and its closure $\Omega:=\overline{\log(\text{EDP})}$ 
in $\mathcal{L}_\mathbb{R}$ with respect to the classical topology.
So, any element $\omega\in \Omega=\overline{\log(\text{EDP})}$ has an expression:
\begin{equation}\label{eq:10.3.2}
\omega\ \ :=\ \ \underset{n\to\infty}{\ \ {\lim}^{cl}}\ 
\frac{\mathcal{M}(T_n)}{\sharp T_n}
\end{equation}
for a sequence $\{T_n\}_{n\in\Z_{>0}}$ in $\Conf_+$, where 
we denote by ${\lim}^{cl}$ the limit with respect 
to the classical topology.
Recalling that the topology on $\mathcal{L}_\R$ is defined by 
the coefficient-wise convergence with respect to the basis $\{\varphi(S)\}_{S\in \Conf_0}$ 
and 
using \eqref{eq:8.2.3}, one has  
$\omega\!=\! \sum_{S\in\Conf_0}\varphi(S)\cdot a_S$ where 
$a_S=\underset{n\to\infty}{\ \ {\lim}}\frac{A(S,T_n}{\# T_n}$.

\begin{assertion}
{\rm 1.} The set $\Omega \overline{\log(\operatorname{EDP})}$ is  compact and convex.

{\rm 2.} Expand any element $\omega\in \Omega=\overline{\log(\operatorname{EDP})}$ as 
$\sum_{S \in \Conf_{0}}\varphi (S) \cdot a_S$.
Then

\vspace{0.1cm}
\noindent
{\rm i)}\quad \  $0 \ \le\ a_S\ \le\ (q-1)^{\sharp S-1}/\sharp \Aut(S)$ \ 
for \ $S \in \Conf_{0}$,

\noindent
{\rm ii)}\quad  $(q-1)^{\sharp S - \sharp S'} a_{S'} \ge a_S$ \ for \ $S' \le S$. 
In particular, if \ $a_S \neq 0$ \ then $a_{S'} \neq 0$.
%
%
\end{assertion}

\begin{proof}
1.  
%
Compactness:  it is enough to show
that the range of coefficients
$a_S$ for $\omega\! \in\! \log(\text{EDP})$ is bounded for each 
$S\! \in\! \Conf_{0}$.
Recalling the expansion formula \eqref{eq:8.2.3},
this is equivalent to the statement that
$\{A(S,T)/\sharp T\mid  T\!\in\! \Conf_{0}\}$
is bounded for any $S \in \Conf_{0}$.
Applying the inequality \eqref{eq:5.2.2}, we have
\begin{equation*}
0\ \le\ A(S,T)/\sharp T\ \le\ (q-1)^{\sharp S-1}/\sharp \Aut(S),
\end{equation*}
which clearly gives a universal bound for $A(S,T)/\sharp T$ independent of $T$.

Convexity: since for any $T$, $T'(\neq[\emptyset])$
and $r\in\Q$ with $0<r<1$, one can find positive
integers $p$ and $q$ such that for
$T'':= T^p \cdot T'^q$ one has
\begin{align*}
\mathcal{M}(T'')/\sharp T'' 
\ & =\ (p\cdot\mathcal{M}(T) + q\cdot\mathcal{M}(T'))/
(p\cdot\sharp T + q\cdot\sharp T')\\
\ & =\ r\cdot\mathcal{M}(T)/\sharp T + (1-r)\cdot\mathcal{M}(T')/\sharp T'.
\end{align*}
2.  \ \ i) This is shown already in 1.

ii) If $S' \le S$ and $S\in \Conf_{0}$,
then for any $T \in \Conf$ one has an inequality
$(q-1)^{\sharp S - \sharp S'}A(S',T)\ge A(S,T)$.
(This can be easily seen by fixing representatives of
$S$ and $S'$ as in proof of \eqref{eq:5.2.2}).
Therefore $(q-1)^{\sharp S - \sharp S'} a_{S'} \ge a_S$.
\end{proof}
\begin{remark}
The condition \eqref{eq:9.4.2} on EDP implies 
$a_{pt}=1$ for any element $\omega\in \Omega=\overline{\log(\text{EDP})}$. 
In particular, this implies $0\not\in \Omega=\overline{\log(\text{EDP})}$.
\end{remark}

\subsection{Join decomposition  
$\Omega=\overline{\log(\text{EDP})}_{abs}*\overline{\log(\text{EDP})}_{\infty}$}
\label{subsec:10.4}
We show that $\overline{\log(\text{EDP})}$ is embedded in
$\LL_{\R,abs}\oplus \LL_{\R,\infty}$, and, accordingly,
decompose $\overline{\log(\text{EDP})}$ into the join of 
a finite part 
and an infinite part, where the finite part is an infinite 
simplex with the vertex set 
$\{\frac{\mathcal{M}(T)}{\#T}\}_{T\in\Conf_0}$.

\begin{definition}
Define the {\it finite part} and the {\it infinite part} of 
$\overline{\log(\text{EDP})}$ by
\begin{eqnarray}\label{eq:10.4.1}
\overline{\log(\text{EDP})}_{abs}&\ : =\ & \overline{\log(\text{EDP})} \cap \LL_{\R,abs},\ \\
\label{eq:10.4.2}
 \overline{\log(\text{EDP})}_{\infty}&\  : =\ & \overline{\log(\text{EDP})} \cap \LL_{\R,\infty}.
\end{eqnarray}
\end{definition}

\begin{lemma}\!\!\!\!
{\bf 1.} $\overline{\log(\text{EDP})}$ is the join of the finite 
part and the infinite part:
\begin{equation}
\label{eq:10.4.3}
\begin{array} {lll}
\overline{\log(\text{EDP})}\ \ =\ \  \overline{\log(\text{EDP})}_{abs}
\ * \ \overline{\log(\text{EDP})}_{\infty}.
\end{array}
\end{equation}
Here, the join of subsets $A$ and $B$ in real vector spaces $V$ and $W$ 
is defined by 

\centerline{
$A*B:=\{\lambda p+(1-\lambda)q\in V\oplus W 
\mid p\in A, q\in B, \lambda\in [0,1]\}$.
}

{\bf 2.}   The finite part
is the infinite simplex of the vertex set $\{\frac{\mathcal{M}(S)}{\#S}\}_{S\in\Conf_0}$:
\[
\begin{array} {lll}
\overline{\log(\text{EDP})}_{abs}
=\left\{ {\sum}^{abs}_{S\in Conf_0} \mu_S \frac{\mathcal{M}(S)}{\#S}
\mid \mu_S\in \R_{\ge0} \ \text{ and }\ {\sum}_{S\in\Conf_0}\mu_S=1
\right\}.
\end{array}
\]
\end{lemma}

\begin{proof} 
We prove 1. and 2. simultaneously in two steps A. and B.
We show only the inclusion LHS \!$\subset$\! RHS since 
the opposite inclusion LHS \!$\supset$\! RHS is trivial 
due to the closed compact convexity of 
$\overline{\log(\text{EDP})}$ (\ref{subsec:10.3} Assertion 1.).

\medskip
 A. {\bf Finite part.} \ Let us consider an element 
$\omega\in \overline{\log(\text{EDP})}$ of the expression 
\eqref{eq:10.3.2}.
For $S\in \Conf_0$, recall that $\delta(S,T_n)$ is the \# of connected 
components of $T_n$ isomorphic to $S$. Let us show that {\it the limit 
\begin{equation}
\label{eq:10.4.4}
\mu_S\ :\ =\ \ {\#S}\ \underset{n\to \infty}{\lim}\frac{\delta(S,T_{n})}{\sharp T_{n}}\
\end{equation}
converges to a finite real number $\mu_S$ such that 
}
\begin{equation}\label{eq:10.4.5}
0\ \le\ {\sum}_{S\in\Conf_0}\mu_S\ \le\ 1 \ .
\end{equation}

Note that the kabi-map $\overline{K}$ \eqref{eq:8.4.1} is also continuous 
with respect to the classical topology.
So, it commutes with the classical limiting process $\underset{n\to\infty}{\ \ \lim^{cl}}
\frac{\mathcal{M}(T_n)}{\sharp T_n}$. 
Recalling the kabi-inversion formula \eqref{eq:7.3.1}, we calculate
\begin{equation*}
\begin{array}{l}
\overline{K}(\omega )\ =\ 
\overline{K}(\underset{n\to \infty}{\ \ \lim^{cl}}\frac{\mathcal{M}(T_{n})}{\sharp T_{n}})
\ =\ 
\underset{n\to\infty}{\ \ \lim^{cl}}\frac{\overline{K}(\mathcal{M}(T_{n}))}{\sharp T_{n}}
\ =\ 
\underset{n\to\infty}{\lim}\!\sum_{S\in\Conf_{0}}\!\!\frac{\delta(S,T_{n})}{\sharp T_{n}}
\mathcal{M}(S).
\end{array}
\end{equation*}
Here, the convergence on the RHS is the coefficient-wise convergence with
respect to the basis $\mathcal{M}(S)$ for $S\in\Conf_{0}$.
This implies the convergence of \eqref{eq:10.4.4}.

Let $C$ be any finite subset of $\Conf_{0}$.
For any $n \in \mathbb{Z}_{\ge0}$, one has
\begin{equation*}
{\sum}_{T\in C}\delta (T,T_{n})\cdot\sharp T\ \  \le\ \  \sharp T_{n}
\end{equation*}
since the LHS is equal to the cardinality of the vertices of the 
union of connected components of $T_{n}$
which is isomorphic to an element of $C$.
Dividing both sides by $\sharp T_{n}$ and taking the limit
$n \to \infty$, one has \eqref{eq:10.4.5}.

Define the {\it finite part of} $\omega$ by the absolutely convergent sum
\begin{equation}\label{eq:10.4.6}
\omega_{finite}\ :\ =\ {\sum}^{abs}_{S\in\Conf_{0}}\mu_S \frac{\mathcal{M}(S)}{\sharp S}
\end{equation} 
(apply {\rm \ref{subsec:10.2}} Assertion to \eqref{eq:10.4.5}). 
We remark that the coefficients $\mu_S$ are uniquely determined from $\omega$  
and are independent of the sequence $\{T_n\}_{n\in\Z_{\ge0}}$, 
due to the formula:
\begin{equation}
\label{eq:10.4.7}
\begin{array}{l}
\overline{K}(\omega )\ =\ \sum_{S\in\Conf_0}\mu_S \frac{\mathcal{M}(S)}{\sharp S}.
\end{array}
\end{equation}

B. {\bf Infinite part.} \ 
Put $\mu_\infty:=1-\sum_{S\in\Conf_0}\mu_S$.
Let us show that 

i) {\it if $\mu_\infty= 0$, then we have
$\omega \ = \ \omega_{finite}$,}
and 
ii) {\it if $\mu_\infty>0$, then there exists a 
unique element $\omega_\infty\in \LL_{\R,\infty}$ so that
}
$
\omega\ =\ \mu_\infty \ \omega_\infty\ +\ \omega_{finite}
$

For any $S\in \Conf_0$, let us denote by
$T_n (S)$ the isomorphism class of 
the union of the connected components of $T_n$ 
isomorphic to $S\in\Conf_0 $. Thus,
$\sharp T_n (S)= \delta(S,T_n)\sharp S$ and $\sharp T_n (S)/\#T_n \to \mu_S$ as $n\to\infty$.
For any finite subset $C$ of $\Conf_0$, 
put $T_n^* (C^c):=T_n \backslash\textstyle\bigcup\limits_{S\in C}
T_n (S)$ so that one has 
%
\vspace{-0.1cm}
\[
\frac{\mathcal{M}(T_n )}{\sharp T_n}\ =\
\frac{\mathcal{M}(T_n^* (C^c))}{\sharp T_n}+
\sum\limits_{S\in C}
\frac{\delta (S,T_n )\sharp S}{\sharp T_n }\cdot
\frac{\mathcal{M}(S)}{\sharp S}.
\tag*{$*$)}
\vspace{-0.1cm}
\]
%
For the given $C$ and for $\varepsilon >0$, 
there exists $n(C,\varepsilon )$ such that 

\noindent
a)
\centerline{ 
$\textstyle\sum\limits_{S\in C}|\mu_S-\sharp T_n (S)/\sharp T_n 
|\ <\ \varepsilon$ }

\noindent
for $n\ge n(C,\varepsilon )$. 
This implies
$|\mu_\infty - \sharp T_n^* (C^c)/\sharp T_n |<\varepsilon +
\textstyle\sum\limits_{S\in\Conf_0 \backslash C} \mu_S$.

Let $\{\varepsilon_m\}_{m\in\mathbb{\Z}_{\ge0}}$ be any sequence of positive real
numbers with $\varepsilon_m \downarrow 0$. Choose an increasing 
sequence $\{C_m\}_{m\in\mathbb{Z}_{\ge0}}$ of finite subsets of $\Conf_0$ 
satisfying 

\noindent
b)
\centerline{
$\textstyle\bigcup\limits_{m\in\mathbb{Z}_{\ge0}}C_m\ =\ \Conf_0$ 
\quad and\quad 
$\textstyle\sum\limits_{S\in\Conf_0 \backslash C_m}\mu_S\ <\ \varepsilon_m$.
}

\noindent
Put $n(m):= n(C_m , \varepsilon_m )$.
Then,  by definition of $\mu_\infty$ and by a) and b), one has

\noindent
c)
\centerline{
$|\mu_\infty-\sharp T_{n(m)}^* (C_m^c )/\sharp T_{n(m)}|\ <\ 2\varepsilon_m $.
}

Substituting $n$ and $C$ in $*)$ by $n(m)$ and $C_m$, respectively, 
we obtain a sequence of equalities indexed by $m\in\Z_{\ge0}$.
Let us prove:

i) {\it the second term of $*)$ absolutely converges to $\omega_{finite}$.}

ii) {\it if $\mu_\infty = 0$, then the first term of $*)$ converges to 0.}

iii) {\it if $\mu_\infty \neq 0$, then $T_m^* : = T_{n(m)}^* (C_m^c)\neq \emptyset$
for large $m$ and $\mathcal{M}(T_m^* )/\sharp T_m^*$ converges to an
element $\omega_\infty\in\overline{\log(\text{EDP})}\cap\LL_{\R,\infty}$.
}

{\it Proof of} i). For $m\in\Z_{\ge0}$, 
the difference of $\omega_{finite}$ and the second term of $*)$ 
is 
$
 \sum_{S\in\Conf_0 } c_S \frac{\mathcal{M}(S)}{\sharp TS }
$
where  $c_S :=\mu_S -\frac{\delta (S,T_{n(m)})\sharp S}
{\sharp T_{n(m)}}$ for $S\in C_m$ and $c_S :=\mu_S$ for 
$S\in\Conf_0 \backslash C_m$.
Therefore, using a) and the latter half of b), one sees that 
the sum $\sum_{S\in \Conf_0}|c_S|$ is bounded by 
$2\varepsilon_m$.
Then, due to a criterion in \ref{subsec:10.2} Assertion, 
the difference tends to 0 absolutely as $m \uparrow \infty$. $\Box$

{\it Proof of} ii).
Recall $c)$ $|\sharp T_{n(m)}^* (C_m^c)/\sharp T_{n(m)}
|<2\varepsilon_m$. The first term of $*)$ is given by
$\frac{\mathcal{M}(T_{n(m)}^* (C_m ))}{\sharp T_{n(m)}} =
\textstyle\sum\limits_{S\in\Conf_0 } \varphi (S)
\frac{A(S,T_{n(m)}^* (C_m ))}{\sharp T_{n(m)}}$,
where the coefficient of $\varphi (S)$ is either 0 if
$T_{n(m)}^* (C_m )\!=\!\emptyset$ or equal to
$\frac{\sharp T_{n(m)}^* (C_m )}{\sharp T_{n(m)}}
\frac{A(S,T_{n(m)}^* (C_m ))}{\sharp T_{n(m)}(C_m )}$
otherwise, which is bounded by $2\varepsilon_m q^{\sharp S-1}/\sharp\Aut(S)$.
So it converges to 0 as $m \uparrow \infty$. $\Box$

{\it Proof of} iii).
The sequence of the first term of the RHS of $*)$
converges to $\omega - \omega_{finite}$, since the LHS of $*)$ and the second term
of the RHS of $*)$ converge to $\omega$ and $\omega_{finite}$,
respectively.
On the other hand, due to c), one has 
$\sharp T_{n(m)}^* (C_m^c )/\sharp T_{n(m)}
>\mu_\infty -2\varepsilon_m$ 
for sufficiently large $m$, and hence one has 
$T_{n(m)}^* (C_m^c )\neq\emptyset$.
The first term is decomposed as:
\begin{equation*}
\frac{\mathcal{M}(T_{n(m)}^* (C_m^c ))}{\sharp T_{n(m)}}\ =\
\frac{T_{n(m)}^* (C_m^c )}{\sharp T_{n(m)}}
\frac{\mathcal{M}(T_{n(m)}^* (C_m^c ))}{\sharp T_{n(m)}^* (C_m^c )} ,
\end{equation*}
whose first factor converges to $\mu_\infty \neq 0$ due to $c)$.
Therefore, the second factor converges to some $\omega_\infty :=(\omega -\omega_{finite} )
/\mu_\infty$, which belongs to $\overline{\log(\text{EDP})}$ by
definition. Since $\overline{K}(\omega) = \overline{K}(\omega_{\infty})$, $\omega_\infty$ belongs to $\ker(\overline{K})$. $\Box$

These complete a proof of the Lemma.
\end{proof}

\subsection{Extremal points in $\Omega_\infty=\overline{\log(\text{EDP})}_{\infty}$.}
\label{subsec:10.5}

A point $\omega$ in a subset $A$ in a real vector space 
is called
an $extremal~ point$ of $A$ whenever an interval $I$  contained in $A$ 
contains $\omega$ then $\omega$  is a terminal point of $I$.

\vspace{-0.2cm}
\begin{assertion}
The extremal point of $\overline{\log(\text{EDP})}$ 
is one of the following:

{\rm i)}\ \   $\frac{\mathcal{M}(S)}{\sharp S}$ for an element  $S\in\Conf_0$, 

{\rm ii)} $\underset{n\to\infty}{\ \ \lim^{cl}}\frac{\mathcal{M}(T_n)}{\sharp T_n}$ for a sequence $T_n\in \Conf_0$
with $\#T_n\to\infty$ ($n\to\infty$).
\end{assertion}

\begin{proof}
For $\omega\in \overline{\log(\text{EDP})}$, if 
$\mu_\infty\not=0,\ 1$,  then $\omega$ cannot be extremal. 
If $\mu_\infty=0$, due to Corollary 1, the only possibility for 
$\omega$ to be extremal is when it is of the form 
$\frac{\mathcal{M}(S)}{\sharp S}$ for an element  $S\in\Conf_0$. 
In fact, using the uniqueness of the  expression
(Lemma 3.), $\frac{\mathcal{M}(S)}{\sharp S}$ 
can be shown to be extremal.

Suppose $\mu_\infty=1$.
For any fixed $S\in\Conf_0$ and  real $\varepsilon > 0$,
let $T_n^+ (S,\varepsilon)$ (resp. $T_n^- (S,\varepsilon)$) be the
subgraph of $T_n$ consisting of the components $T$ such that
$A(S,T)/\sharp T \ge a_S + \varepsilon$ (resp. $\le a_S - \varepsilon$).
Let us show that $\overline{\underset{n}{\lim}}~ \sharp T_n^{\pm}
(S,\varepsilon)/\sharp T_n = 0$.
If not, then there exists a subsequence $\{\widehat{n}\}$ such that
$\underset{\widehat{n}}{\lim}~\sharp T_{\widehat{n}}^{\pm}(S,\varepsilon)/
\sharp T_{\widehat{n}} = \lambda >0$.
Due to the compactness of 
$\overline{\log(\text{EDP})}$ ((\ref{subsec:10.3}) Assertion 1.),
we can choose a subsequence such that 
$\mathcal{M}(T_{\widehat{n}}^{\pm}(S,\varepsilon))/\sharp T_{\widehat{n}}^{\pm}
(S,\varepsilon)$ and $\mathcal{M}(T_{\widehat{n}}\backslash T_{\widehat{n}}^{\pm}
(S,\varepsilon))/\sharp (T_{\widehat{n}}\backslash T_{\widehat{n}}^{\pm}
(S,\varepsilon))$ converges to some 
$\textstyle{\sum}_{T\in\Conf_0}\varphi (T)\cdot b_T$ and
$\textstyle{\sum}_{T\in\Conf_0}\varphi (T)\cdot c_T$,
respectively, so that 
\vspace{-0.3cm}
\begin{equation*}
\omega = \lambda\cdot\textstyle{\sum}_{T\in\Conf_0}\varphi (T)\cdot b_T
+ (1-\lambda)\cdot\textstyle{\sum}_{T\in\Conf_0}\varphi (T)\cdot c_T.
\vspace{-0.2cm}
\end{equation*}
In particular, the coefficient of $\varphi (S)$ has the relation 
$a_S = \lambda \cdot b_S + (1-\lambda )\cdot c_S$.
Since $|b_S-a_S|\ge\varepsilon$,
$\lambda$ cannot be 1.
This contradicts the extremity of $\omega$.

For any finite subset $C$ of $\Conf_0$, put 
$T_n^* (C,\varepsilon):= T_n \backslash\bigcup_{S\in C}
\bigl( T_n^+ (S,\varepsilon)\cup T_n^- (S,\varepsilon)\bigr)$.
Then $T_n^* (C,\varepsilon)\neq \emptyset$ for sufficiently large $n$,
since $\underset{n\to\infty}{\lim}~ \sharp T_n^* (C,\varepsilon)/\sharp T_n = 1$
due to the above fact.
Let $\{C_m\}_{m\in\mathbb{Z}_{\ge0}}$ be an increasing sequence of finite
subsets of $\Conf_0$ such that $\bigcup_{m\in\mathbb{Z}_{\ge0}}
C_m =\Conf_0$ and let$\{\varepsilon_m\}_{m\in\mathbb{Z}_{\ge0}}$ be a sequence of 
real numbers with $\varepsilon_m\downarrow 0$.
For each $m\in\mathbb{Z}_{\ge0}$, choose any connected component of
$T_n^* (C_{m},\varepsilon_{m})$, say $T_m^*$, for large $n$,
and put $\omega_m :=\mathcal{M}(T_m^* )/\sharp T_m^* =
\sum_{S\in\Conf_0}\varphi (S)\cdot a_S^{(m)}$.
By definition $|a_S - a_S^{(m)}| < \varepsilon_{m}$ for 
$S\in C_m$, which implies $\omega = \underset{m\to\infty}{\ \ \lim^{cl}}~ \omega_m$.
There are two cases to consider:
i) Suppose $\exists$ a subsequence $\{\widehat{m}\}$ such that
$\sharp T_{\widehat{m}}^*$ is bounded.
Since $\sharp \{T\in\Conf_0 | \sharp T\le c\}$ for any constant
$c$ is finite, there exists $T\in\Conf_0$ which appears in
$\{T_m^*\}_m$ infinitely often.
So $\omega = \mathcal{M}(T)/\sharp T$ and $\overline{K}(\omega ) = 
\mathcal{M}(T)/\sharp T \neq 0$.
ii) Suppose $\sharp T_m^* \to \infty$.
Then the formula \eqref{eq:10.4.4} and \eqref{eq:10.4.7} imply
$\overline{K}(\omega ) = 0$.
\end{proof}

\subsection{Function value representation of elements of 
$\Omega_\infty=\overline{\log(\text{EDP})}_\infty$}
\label{subsec:10.6}

The coefficients $a_S$ 
at $S\in\Conf_0$ of the sequential limit 
$\omega=\underset{n\to\infty}{\ \ \lim^{cl}}\mathcal{M}(T_n)/\sharp T_n$ 
\eqref{eq:10.3.2} 
are usually hard to calculate. 
However, in certain good cases, we represent the coefficient as a special 
value of a function in one variable $t$.

Given an expression of the form \eqref{eq:10.3.2} of 
$\omega\in\overline{\log(\text{EDP})}_\infty$ 
and an increasing sequence  of integers $\{n_m\}_{n=0}^{\infty}$, 
we consider the following two formal power series in $t$.
 \vspace{-0.2cm}
\begin{align}\label{eq:10.6.1}
\vspace{-0.2cm}
P(t) &\ :=\ \sum_{m=0}^{\infty}\sharp T_m \cdot t^{n_m}
\qquad \in\ \mathbb{Z}\db[t\db],\\
\vspace{-0.2cm}
P\mathcal{M}(t) &\ :=\ \sum_{m=0}^{\infty}
\mathcal{M}(T_m) \cdot t^{n_m}\  \in\  \mathcal{L}_{\mathbb{Q}}\db[t\db]
=\mathcal{L}_{\mathbb{Q}\db[t\db]},
\vspace{-0.1cm}
\label{eq:10.6.2}
\end{align}
where, using the basis expansion \eqref{eq:8.2.3}, 
the series $P\mathcal{M}(t)$ can be expanded as 
\centerline{
$
P\mathcal{M}(t)\ =\ \sum_{S\in\Conf_0}\varphi(S) PM(S,t),
$}

\noindent
whose coefficients at $S\in \Conf_0$ are given by 
\vspace{-0.2cm}
\begin{equation}\label{eq:10.6.3}
PM(S,t)\ :=\ \partial_S P\mathcal{M}(t)\ =\ 
\sum_{m=0}^{\infty} A(S,T_{m})\cdot t^{n_m}
\in \mathbb{Q}\db[t\db] \ .
\vspace{-0.2cm}
\end{equation}
Since $T_n\in \Conf_+$, 
one has $P(t)\not=0$ and its radius of 
convergence is at most 1.  
\begin{lemma}
Suppose that the series $P(t)$ has a positive radius of convergence $r$. 
Then, 
for any $S\in \Conf_0$ (c.f.\ Remark), we have

{\rm i)} The series $PM(S,t)$
converges at least in the radius $r$ for $P(t)$. 
The radius 
of convergence of $PM(S,t)$ coincides with $r$, 
if $a_S:=\underset{m\to \infty}{\lim}\frac{M(S,T_m)}{\#T_m}\not=0$.

{\rm ii)} The following two limits in LHS and RHS give the same value:
\begin{equation}\label{eq:10.6.4}
\lim\limits_{t \uparrow r}\frac{ PM(S,t)}{P(t)}\ 
=\ \underset{n\to \infty}{\lim}\frac{M(S,T_n)}{\#T_n}.
\vspace{-0.1cm}
\end{equation}
Here by the notation $t\! \uparrow\! r$ we mean that the real variable $t$ tends to $r$ 
from below.

{\rm iii)} The proportion $P\mathcal{M}(t)/P(t)$ for 
$t\uparrow r$ 
converges to $\omega$ \eqref{eq:10.3.2}:
\begin{equation}\label{eq:10.6.5}
\omega\ =\ {\lim\limits_{t \uparrow r}}^{cl}\frac{P\mathcal{M}(t)}{P(t)} 
\ =\ \sum_{S\in\Conf_0}\varphi(S)\ \lim\limits_{t \uparrow r}\frac{ PM(S,t)}{P(t)}.
\end{equation}
\end{lemma}
\vspace{-0.2cm}
\begin{proof}
Before proceeding to the proof, we recall two general properties of 
power series:

\noindent
A) The radius of convergence of P(t) is  
$r\!: =\! 1/\limsup\limits_{m\to\infty}\sqrt[\leftroot{3}n_m]{\sharp T_m}$ (Hadamard).
%

\noindent
B) Since the coefficients $\sharp T_m$ of $P(t)$ are non-negative real numbers, 
{\it $P(t)$ is an increasing positive real function
on the interval (0,$r$) and 
$\lim\limits_{t\uparrow r}P(t)=+\infty$}.

%
We now turn to the proof. Due to the linear relations among $M(S,T_m)$ for $S\!\in\!\Conf$ \eqref{eq:8.3.2}, 
it is sufficient to show the lemma only for the cases $S\! \in\! \Conf_0$.

i) Let us show that $PM(S,t)$ for $S \in \Conf_0$ has the radius $r$
of convergence. 
Since we have 
$M(S,T_m)=A(S,T_m)$ (\ref{subsec:6.1} Remark 1),
using \eqref{eq:5.2.1}, we have

$
\limsup\limits_{m\to\infty}\sqrt[\leftroot{3}n_m]{M(S,T_m)}\ 
\le\ \limsup\limits_{m\to\infty}\sqrt[\leftroot{3}n_m]{\#T_m}\sqrt[\leftroot{3}n_m]{q^{\#S-1}/\#\Aut(S)}\ =\ 1/r.
$

\noindent
This proves the first half of i). The latter half is shown in the next ii).

\medskip
ii) We show that
the convergence of the sequence 
$  
A(S,T_{m})/\sharp T_m$ to some $a_S \in \mathbb{R}$
implies the convergence of the values of the function
$  
PM(S,t)/P(t)$ to $a_S$ as $t\uparrow r$.
The assumption implies that
for any $\varepsilon>0$,
there exists $N > 0$ such that
$|A(S,T_{m})/\sharp T_m - a_S|\le\varepsilon$ for all $m \ge N$. 
Therefore, 
\vspace {-0.2cm}
\begin{align*}
\begin{array}{rl}
\Big| \ \frac{PM(S,t)}{P(t)} -a_S\ \Big|
\ =\  &\ \frac{ |\ Q_N(t) + \textstyle\sum_{m=N}^{\infty}
\Bigl( A(S,T_{m})-a_S \cdot \sharp T_{m} \Bigr) t^{n_m}\ |}{P(t)}\\
\le \  &\ \frac{|\ Q_N (t)-\varepsilon \textstyle\sum_{m=0}^{N-1}\#T_mt^{n_m}\ |}
{P(t)} + \varepsilon 
\end{array}
\end{align*}
where $Q_N (t):=\sum_{m<N} \Bigl(A(S,T_m)-a_S
\cdot \sharp T_m \Bigr) t^{n_m}$ is a polynomial in $t$.
Due to statement B) above,
the first term of the last line tends to 0 as $t\uparrow r$.
Hence, $|PM(S,t)/P(t)-a_S|\le2\varepsilon$
for $t$ sufficiently close to $r$.
This proves \eqref{eq:10.6.4}.

If $a_S\not=0$, then 
$\lim\limits_{t\uparrow r}PM(S,t)=\infty$ 
since $\lim\limits_{t\uparrow r}P(t)=\infty$. 
Thus, the radius of convergence of $PM(S,t)$ is 
less or equal than $r$. This proves the latter
half of i). 

iii) We have only to recall that the classical topology on 
$\mathcal{L}_\mathbb{R}$ is the same as  coefficient-wise convergence 
with respect to the basis $\{\varphi(S)\}_{S\in\Conf_0}$.
\end{proof}

\begin{cor}
If $P(t)$ and $PM(S,t)$ ($S\!\in\!\Conf_0$) extend to meromorphic functions at
t=r, then  $PM(S,t)/P(t)$ is regular at $t\!=\!r$ and 
one has 
%
\begin{equation}\label{eq:10.6.6}
\omega\ =\ \sum_{S\in\Conf_0}\varphi(S)\ \frac{PM(S,t)}{P(t)}\Big|_{t=r}
\end{equation}
\end{cor}
\begin{proof} 
We have to show that $PM(S,t)/P(t)$ 
becomes holomorphic at 
$t=r$ under the assumption.
If it were not holomorphic, it would have a pole at $t=r$ and hence
$\lim\limits_{t\uparrow r}PM(S,t)/P(t)$ diverges.
On the other hand, 
in view of \eqref{eq:5.2.2}, one has the inequality
$0 \le PM(S,t) \le P(t)\cdot q^{\sharp S-1}/\Aut(S)$ for
$t \in (0,r)$.
Then the positivity of $P(t)$ implies the boundedness
$0 \le PM(S,t)/P(t) \le q^{\sharp S-1}/\Aut(S)$ for
$t \in (0,r)$.
This is a contradiction.
\end{proof}
We sometimes call \eqref{eq:10.6.6} a {\it residual expression of}
 $\omega$, since 
the coefficients are given by the proportions of residues of 
meromorphic functions.

\begin{remark} {\bf 1.} The equality \eqref{eq:10.6.4} gives the following 
important replacement.
Namely, the RHS, which  is a sequential limit of rational numbers and 
is hard to determine in general, is replaced by  
the LHS, which is the limit of value of a function in a variable $t$
at the special point $t=r$ where $r$ is often a real algebraic number whose 
defining equation is easily calculable.

{\bf 2.} The convergence of the sequence 
${\lim\limits_{n\to\infty}}^{\!\!\! cl}\frac{M(S,T_n)}{\#T_n}$ does not imply 
the convergence of the 
series $PM(S,t)$ and $P(t)$ in a positive radius. 
Conversely, the convergence of the series $PM(S,t)$ and $P(t)$ in a positive 
radius does not 
imply the convergence of the sequence ${\lim\limits_{n\to\infty}}^{\!\!\! cl}\frac{M(S,T_n)}{\#T_n}$.

\end{remark}


\section{Limit space $\Omega(\Gamma,G)$ for a finitely generated monoid.}

We apply the space $\LL_{\R,\infty}$ to the study of finitely generated monoids.

For a pair $(\Gamma,G)$ of a monoid $\Gamma$ and a 
finite generating system $G$ with Assumptions 1, 2, 
we introduce 1) the limit space $\Omega(\Gamma,G)$ as a subset of 
$\LL_{\R,\infty}$, 2) another limit space $\Omega(P_{\Gamma,G})$ associated with Poincare series $P_{\Gamma,G}(t)$ of $(\Gamma,G)$, and  
3) a proper surjective map $\pi_\Omega:\Omega(\Gamma,G)\to \Omega(P_{\Gamma,G})$ 
(see 11.2 Theorem).

The main result of the present paper is given in 11.5 Theorem, where the sum of elements of a fiber of $\pi_\Omega$ is expressed by a linear combination of the proportions  of 
residues of the
Poincare series $P_{\Gamma,G}(t)$ and $P_{\Gamma,G}\mathcal{M}(t)$ 
at the poles on the circle of the radius of convergence of the Poincare series.

\vspace{-0.1cm}

\subsection{The limit space $\Omega(\Gamma,G)$ for 
a finitely generated monoid}
\label{subsec:11.1}

Let $\Gamma$ be a monoid with left and right cancellation conditions and let $G$ be its finite generatig system with $e\! \notin\! G$.
We denote by $(\Gamma,G)$ the associated colored oriented 
Cayley graph (\ref{subsec:2.1} Example 1). In this and the next section, we use $G\cup G^{-1}$ as the color set and $q:=\#(G\cup G^{-1})$ for the definition of $\Conf$ in \eqref{eq:2.2.1}.
The set of all isomorphism classes of finite subgraphs of $(\Gamma,G)$ 
is denoted by $\langle\Gamma, G\rangle$. Put  
$\langle\Gamma, G\rangle_0:= \langle\Gamma, G\rangle\cap \Conf_0$.

The length of $\gamma \in \Gamma$ with respect to $G$ is defined  by 
\begin{equation}
\label{eq:11.1.1}
\begin{array}{lll}
\ell_G (\gamma) :
& =\  \inf\ \{n\in\Z_{\ge0}\mid \gamma = g_1 \cdots g_n ~\text{for some}~ 
g_i \in G \ (i=1,\ldots ,n)\}\\
\end{array}\!\!\!\!\!\!
\end{equation}
We remark that in the above definition \eqref{eq:11.1.1}, we admit the expressions of $\gamma$ only in positive powers of elements of $G$ (except when $G$ itself already contains the inverse). This means that we allow only edges whose ``orientation'' fits with the orientation of the path. In particular, $l_G(\gamma)$ may not coincide with the distance of $\gamma$ from $e$ in the Cayley garaph.\footnote{
  The length $l_G$ coincides with the distance from $e$ for  the case $G\!=\!G^{-1}$ when $\Gamma$ is a group. Besides this case, there is an important class of monoids, where both concepts coincides, namely, when the monoid is defined by positive homogeneous relations [S-I].
}
 For $n\in\Z_{\ge0}$, let us consider the ``balls'' of radius $n$ of $(\Gamma,G)$ 
defined by
\begin{equation}
\label{eq:11.1.2}
\Gamma_n\ :\ =\ \  \{\  \gamma\in\Gamma \mid  \ell_G (\gamma) \le n\ \}.
\end{equation}
We shall denote $\dot\Gamma_n:=\Gamma_n\!\setminus\! \Gamma_{n-1}$ for $n\in\Z_{\ge0}$.
So far there is no confusions, we shall denote by $\Gamma_n$ its isomorphism 
class  $[\Gamma_n]\in \Conf_0$ also.

\begin{definition}
The {\it set of limit elements for} $(\Gamma,G)$ is defined by
\begin{equation}\label{eq:11.1.3}
\Omega(\Gamma,G)\ :\ =\  \LL_{\R,\infty}\cap 
\overline{\left\{\frac{\mathcal{M}(\Gamma_n)}{\sharp \Gamma_n} \mid n\in\Z_{\ge0}\right\}}\ ,
\vspace{-0.2cm}
\end{equation}
where  $\overline{A}$ is the closure of 
a subset $A\subset \LL_{\R}$ with respect to the classical topology.
\end{definition}

\begin{fact}
{\it The limit space $\Omega(\Gamma,G)$ is non-empty if and only if $\Gamma$ is infinite.}
\end{fact}

\noindent
{\it Proof.}
Since 
$\big\{\frac{\mathcal{M}(\Gamma_n)}{\sharp \Gamma_n}\mid n\! \in\!\Z_{\ge0}\big\}
\!\subset\! \log(EDP)$ and $\overline{\log(EDP)}$ 
is compact (\ref{subsec:10.3}), 
the sequence $\big\{\frac{\mathcal{M}(\Gamma_n)}{\sharp \Gamma_n}\mid n\!\in\!\Z_{\ge0}\big\}$
always has accumulation 
points. 
Due to  \eqref{eq:10.4.4} and  \eqref{eq:10.4.7}, an accumulation point 
$\omega$ belongs to $\LL_{\R,\infty}$, i.e.\ 
it satisfies the kabi-condition $\overline{K}(\omega)=0$, 
if and only if $\#\Gamma_n\!\to\! \infty$.\qquad $\Box$

\medskip
Since $\overline{\log(EDP)}$ is metrizable, any element $\omega$ in 
$\Omega(\Gamma,G)$ can be expressed as a sequential limit.
That is, there exists a subsequence $n_m\!\!\uparrow\!\!\infty$ 
of $n\!\!\uparrow\!\!\infty$ such that 
\vspace{-0.4cm}
\begin{equation}
\label{eq:11.1.4}
\omega= \underset{n_m\to\infty}{{\lim}^{cl}}
\frac{\mathcal{M}(\Gamma_{n_m})}{\sharp \Gamma_{n_m}}
\ = \sum_{S\ \in\ \langle\Gamma,G\rangle_0}\!\!\varphi(S)\
\underset{n_m\to\infty}{\lim}
\frac{A(S,\Gamma_{n_m})}{\sharp \Gamma_{n_m}}
\vspace{-0.2cm}
\end{equation}
where the coefficient of $\varphi(S)$ is convergent for all $S$. 

\smallskip
\noindent
{\bf Definition.}
We call a finitely generated monoid $(\Gamma,G)$  {\it simple} 
(resp. {\it finite}) {\it accumulating} 
if $\Omega(\Gamma,G)$ consists of a single (resp. finite number of) 
element(s).

\bigskip
\noindent
{\bf Assumption 1.}  From now on until the end of the present paper, we assume that the monoid $\Gamma$ is embeddable into a group. That is, there exists an injective homomorphism from $\Gamma$ into a group. This is obviously satisfied if $\Gamma$ is a group.

\bigskip
In the following Examples 1. and 2., we show that any 
polynomial growth group and any free group 
is simple accumulating. 
We first state some general properties of 
the set $\Gamma_n$, which are immediate consequences of the definition.

\begin{fact}
 1. For $m,n\in\Z_{\ge0}$, one has a natural surjection:
\begin{eqnarray}
\label{eq:11.1.5}
\Gamma_m  \ \times\ \Gamma_n \longrightarrow \Gamma_{m+n},& \ 
\gamma\times\delta  \mapsto \  \gamma\delta 
\end{eqnarray}
 2. For any $S\!\in\! \Conf_0$ with $S\le \Gamma_k$ ($k\in\Z_{\ge0}$)
and for any $n\!\in\! \Z_{\ge0}$, 
one has:
\begin{equation}\label{eq:11.1.6}
\#\Gamma_{n-k}\ \le\ \#(\Aut(S))\cdot A(S,\Gamma_{n})\ \le\ \#\Gamma_{n}.
\end{equation}
\end{fact}
\begin{proof} 1. Obvious by definition.

2. By the assumption on $S$, there exists a subgraph 
$\mathbb{S}\!\subset\!\Gamma_k$ such that $S\!=\![\mathbb{S}]$. Note that 
$Aut(S)\!\simeq\! Aut(\mathbb{S})\!= \!\{g\in \hat \Gamma\mid g\mathbb{S}\!=\!\mathbb{S}\}$ {\it is finite and its action is fixed point free.}
Consider a map $p$ from $\Gamma$ to the set of subgraphs of $(\Gamma,G)$
defined by $p(g):= g\mathbb{S}$, and define an equivalence relation 
$\sim$ on $\Gamma$ by ``$g\sim h 
\Leftrightarrow  g\mathbb{S}= h\mathbb{S} 
\Leftrightarrow g^{-1}h\in \Aut(\mathbb{S})$''. Then, one has
$A(S,\Gamma_n)\ge \#({\rm Image}(p|_{\Gamma_{n-k}}))=\#(\Gamma_{n-k}/\sim)
\ge \#(\Gamma_{n-k})/\#(\Aut(\mathbb{S}))$. 
This implies the first inequality.

Choose a point $x\in \mathbb{S}$. 
Consider a set 
 $ P\!:=\!\{g\!\in\!\Gamma_{n}\mid gx^{-1}\mathbb{S}\!\subset\! \Gamma_n\}$.
Then, the map 
$p|_P\circ x^{-1}: P\!\to\! \mathbb{A}(S,\Gamma_n)$ is surjective and 
$P$ is closed under the right multiplication of $x^{-1}\Aut(\mathbb{S})x$.
Then, one has $A(S,\Gamma_n)\!=\! \#(P)/\#(\Aut(\mathbb{S}))\!\le\!  \#(\Gamma_{n})/\#(\Aut(\mathbb{S}))$.
This implies the second inequality.
\end{proof}

 Let $(\Gamma,G)$ be a monoid such that $\underset{n\to\infty}\lim (\#\Gamma_{n-k}/\#\Gamma_n)=1$ for any $k\in \Z_{\ge0}$. Then, as a consequence of (11.1.6),  one has
\begin{equation}\label{eq:11.1.7}
\vspace{-0.1cm}
\underset{n\to\infty}\lim \frac{A(S,\Gamma_n)}{\#\Gamma_n} \ =\ \frac{1}{ \#(\Aut(S))}.
\end{equation}

\begin{example}
{\bf 1.}
 If $\Gamma$ is a group of polynomial growth,
 then it is simple accumulating for any generating system $G$ and the limit element is given by  
 \begin{equation}\label{eq:11.1.8}
 \omega_{\Gamma,G}\ :\ = \sum_{S \in \langle\Gamma,G\rangle_0}\frac{1}{ \#(\Aut(S))}\varphi(S).
 \end{equation}
\begin{proof}
For a group $(\Gamma,G)$ of polynomial growth (i.e.\ $\Gamma$ contains a 
finitely generated nilpotent group of finite index,  
Wolf and Gromov \cite{Gr1}), there exist constants
$c,d\in \Z_{>0}$ such that $\#\Gamma_n=cn^d +o(n^d)$ (Pansu \cite{P}).
\end{proof}

{\bf 2.}  Let $F_f$ be a free group with the generating system
$G\!=\!\{g_1^{\pm1},\!\cdots\!,g_f^{\pm1}\}$ for $f\!\in\!\Z_{\ge2}$. 
Then $(F_f,G)$ is simple accumulating. The limit element is given by
{\small
\begin{equation}\label{eq:11.1.9}
\omega_{F_f,G} :=
 \sum_{k=0}^{\infty}(2f-1)^{-k}
\left(\sum_{\substack{S\in \langle \Gamma,G\rangle_0\\d(S)=2k}}\varphi(S)+f^{-1}\sum_{\substack{S\in \langle \Gamma,G\rangle_0\\d(S)=2k+1}}\varphi(S)\right),
\end{equation}
}
where $d(S):=\max\{d(x,y)\mid x,y\in S\}$ is the {\it diameter} of $S$ for 
$S\!\in\! \langle F_f,G\rangle_0$.
\end{example}

\vspace{-0.3cm}
\begin{proof}
The induction relation: $\#\Gamma_{n+1}\!-\!(2f\!-1\!)\#\Gamma_n\!=\!2$ with 
the initial condition $\#\Gamma_0\!=\!1$ implies
$\#\Gamma_n\!=\!\frac{f(2f-1)^n-1}{f-1}$ for $n\!\in\!\Z_{\ge0}$ so that 
$P_{F_f,G}(t)=\frac{1+t}{(1-t)(1-(2f-1)t)}$.
On the other hand, for $S\in \langle F_f,G\rangle_0$ and  
for $n\!\ge\![d(S)/2]$, 
{\small
\begin{eqnarray}
\label{eq:11.1.10}
A(S,\Gamma_n)&\ =\ \begin{cases}\frac{f(2f-1)^{n-[d(S)/2]}-1}{f-1}& \text{if $d(S)$ is even,}\\
\frac{(2f-1)^{n-[d(S)/2]}-1}{f-1}& \text{if $d(S)$ is odd.}
\end{cases} \\
\label{eq:11.1.11}
\underset{n\to\infty}\lim\frac{A(S,\Gamma_n)}{\#\Gamma_n}& \ \ \ =\ 
\begin{cases}
(2f-1)^{-[d(S)/2]} &\text{ if $d(S)$ is even,}\\
f^{-1}(2f-1)^{-[d(S)/2]} &\text{ if $d(S)$ is odd.}
\end{cases}
\end{eqnarray}
}
We have only to prove the first formula. 
Depending on whether $d(S)$ is even or odd, 
$S$ has either one or two central points. 
 Then it is easy to see the following one to one correspondence:	
{\it an embedding of $S$ in $\Gamma_n$ $\Leftrightarrow$
an embedding of the central point(s) of $S$ in $\Gamma_n$ 
such that the distance from the point to the boundary of $\Gamma_n$ is at least
half of the diameter $[d(S)/2]$.}
Taking this into account, we can calculate directly the formula.
\end{proof}

\subsection{The space $\Omega(P_{\Gamma,G})$ of opposite sequences}
\label{subsec:11.2}

We introduce another accumulation set $\Omega(P)$, 
called  {\it the space of opposite sequences},
associated to certain real power series $P(t)$.
Under a suitable assumption on $(\Gamma,G)$, we have a fibration   
$\pi_\Omega: \Omega(\Gamma,G)\to\Omega(P_{\Gamma,G})$ for the Poincare series 
$P_{\Gamma,G}$ of $(\Gamma,G)$.
We construct semigroup actions on $ \Omega(\Gamma,G)$ and 
$\Omega(P_{\Gamma,G})$ generated by $\tilde \tau_\Omega$ and $\tau_\Omega$, 
respectively, which are equivariant with $\pi_\Omega$.

We start with a general definition. Consider a power series in $t$
\begin{equation}
\label{eq:11.2.1}
\begin{array}{l}
P(t) \ =\ \sum_{n=0}^\infty \gamma_nt^n
\end{array}
\vspace{-0.1cm}
\end{equation}
whose coefficients are real numbers. We assume that there exist  
positive real numbers $u,v$ (depending on $P$) such that 
$u\!\le\! \gamma_{n-1}/\gamma_n\! \le\! v$ for all $n\!\in\!\Z_{\ge1}$.
This, in particular, implies that $P$ is convergent of radius 
$r$ with $u\!\le\! r\! \le\! v$.

\medskip
{\bf Example.\ }If the sequence 
$\{\gamma_n\}_{n\in\Z_{\ge0}}$ is increasing and semi-multiplicative
$\gamma_{m+n}\!\le\!\gamma_m\gamma_n$, 
we may choose $u=1/\gamma_1$ and $v=1$. For example,
let $\gamma_n\!:=\!\#\Gamma_n$ ($n\!\in\!\Z_{\ge0}$) in the setting of 11.1, 
then \eqref{eq:11.1.5}
implies semi-multiplicativity.

Associated to $P$, consider 
a sequence $\{X_n(P)\}_{n\in\Z_{\ge0}}$ of polynomials:
\begin{equation}
\label{eq:11.2.2}
\begin{array}{ll}
\qquad \qquad 
X_n(P)\ := \ \sum_{k=0}^n\frac{\gamma_{n-k}}{\gamma_n}\ s^k, \ \ 
\ n=0,1,2,\cdots ,
\end{array}
\vspace{-0.1cm}
\end{equation}
in the space $\R\db[s\db]$
of formal power series, where $\R\db[s\db]$ is equipped with the formal 
classical topology, i.e.\ the product topology of  convergence of every coefficient in classical topology.
Since each coefficients of $X_n(P)$ are bounded, i.e.\ $u^k\!\le\! \frac{\gamma_{n-k}}{\gamma_n}\!\le\! v^k$, 
the sequence accumulates to a non-empty compact set:
\begin{eqnarray}
\label{eq:11.2.3}
&\qquad \Omega(P)  := \text{the set of accumulation points of 
the sequence \eqref{eq:11.2.2}}.
\vspace{-0.1cm}
\end{eqnarray} 
An element $a(s)\!=\!\Sigma_{k=0}^\infty a_ks^k$ of $\Omega(P)$ 
is called an {\it opposite series}.
The coefficients $\{a_k\}_{k=0}^\infty$
satisfies $u^k\le a_k\!\le\!v^k$. We call $a_1$ 
the 
{\it initial} of the opposite series $a$, denoted by $\iota(a)$.
Let us introduce the space of initials:
\begin{equation}
\label{eq:11.2.4}
\Omega_1(P)  := \text{the set of accumulation points of 
the sequence $\Big\{\!\!\frac{\gamma_{n-1}}{\gamma_n}\!\!\Big\}_{\!n\in\Z_{\ge0}}$\!\!}\!,
\!\!\!\!\!\!\!
\vspace{-0.1cm}
\end{equation}
which is a compact subset of the positive interval $[u,v]$. 
The projection map 
$a\!\in\!  \Omega(P)\mapsto \iota(a)\!\in\!\Omega_1(P)$ 
is a continuous surjective map.  

\begin{assertion} 
{\bf 1.} If a sequence $\{X_{n_m}(P)\}_{m\in\Z_{\ge0}}$ converges 
to an opposite sequence $a$, then the sequence $\{X_{n_m-1}(P)\}_{m\in\Z_{\ge0}}$
 converges also to an opposite sequence, denoted by $\tau_{\Omega}(a)$. 
We have 
\vspace{-0.2cm}
\begin{equation}
\label{eq:11.2.5}
 \tau_\Omega(a)\ =\ (a-1)/\iota(a) s.
\end{equation}

\vspace{-0.2cm}
{\bf 2.}  Consider a map 

\noindent
{\rm (11.2.5)*} \quad  \qquad  $\tau:\Omega(P)\longrightarrow \overline{\R\Omega}(P), 
\quad  a \ \mapsto \ \iota(a) \tau_\Omega(a)$ 

\medskip
\noindent
where $\overline{\R\Omega}(P)$ 
is a closed $\R$-linear subspace of $\R\db[s\db]$ generated by 
$\Omega(P)$. Then, the map $\tau$ naturally extends to an endomorphism of
$\overline{\R\Omega}(P)$.
\vspace{-0.3cm}
\end{assertion}
\begin{proof} 1.\ By definition, the sequence $\{\gamma_{n_m\!-\!1}\!/\gamma_{n_m}\!\}_m$
converges to the non-zero initial $\iota(a)\!\not=\!0$. 
Then, for any fixed $k\!>\!0$, the $(k-1)$th coefficient of $\tau_{\Omega}(a)$ is given by the 
limit of sequence
$\{\gamma_{n_m\!-\!k}/\gamma_{n_m\!-\!1}\}_m$ converging to 
$a_k/a_1$.

2. Let $\sum_{i\in I} c_i a_i(s)\!=\!0$ be a linear 
relation among opposite sequences $a_i(s)$ ($i\!\in\! I$) 
with $\#I\!<\!\infty$, then we also have  a  
linear relation $\sum_{i\in I} c_i a_{i,1}\tau_\Omega(a_i(s))\!=\!0$, 
since, using the expression \eqref{eq:11.2.5}, 
this follows from the original relation $\sum_{i=1}^\infty c_i a_i(s)\!=\!0$
and another one $\sum_{i=1}^\infty c_i \!=\!0$, which is obtained by
substituting $s\!=\!0$ in the first relation. This implies that $\tau$ 
is extended to a linear map: $\R\Omega(P)\to \overline{\R\Omega}(P)$. 
On the other hand, $a(s)\in \R\db[s\db]\mapsto (a(s)-a(0))/s\in \R\db[s\db]$ 
is a well-defined continuous map, so that it induces a map 
$\mathrm{End}_\R(\overline{\R\Omega}(P))$.
\end{proof}

\vspace{-0.2cm}
We return to the setting in \ref{subsec:11.1} and consider 
a Cayley graph $(\Gamma,G)$.
For the sequence $\{\Gamma_n\}_{n\in\Z_{\ge0}}$ \eqref{eq:11.1.2},  
we consider two series \eqref{eq:10.6.1} and \eqref{eq:10.6.2}:
\begin{equation}
\label{eq:11.2.6}
 \vspace{-0.4cm}
\begin{array}{llll}
P_{\Gamma,G}(t)\ &:=\ &\sum_{n=0}^{\infty}\sharp\Gamma_n\cdot t^n ,
\end{array}
\end{equation}
\begin{equation}
\label{eq:11.2.7}
\begin{array}{llll}
P_{\Gamma,G}\mathcal{M}(t)\ &:=\ &\sum_{n=0}^{\infty}
\mathcal{M}(\Gamma_n )\cdot t^n .
\end{array}
\end{equation}

\vspace{-0.1cm}
Here \eqref{eq:11.2.6} is well known \cite{M} 
as the growth (or Poincare) series
for  $(\Gamma,G)$, and \eqref{eq:11.2.7} 
 is the series which we study in the present paper.
Due to  \eqref{eq:11.1.5}, 
it is well known that the growth series converges 
with positive radius:
\begin{equation}
\label{eq:11.2.8}
r_{\Gamma,G}\ :=\ 1\ /\ 
\underset{n\to\infty}\lim \sqrt[\leftroot{3}n]{\#\Gamma_n}\ \ge\ 
1\ /\ \#\Gamma_1.
 \vspace{-0.1cm}
\end{equation}
Due to \ref{subsec:10.6} {\it Lemma} i), the series 
$P_{\Gamma,G}\mathcal{M}(t)$ converges in the same radius as 
$P_{\Gamma,G}(t)$. 
This fact can be directly confirmed by  using  
\eqref{eq:11.1.6} for $S\le [\Gamma_k]$ as 

\centerline{
$
\lim\limits_{n\to\infty}(\sqrt[\leftroot{3}n-k]{\#\Gamma_{n-k}})^{\frac{n-k}{n}}
\le 
\lim\limits_{n\to\infty}\sqrt[\leftroot{3}n]{\#(\Aut(S))A(S,\Gamma_{n})}
\le
\lim\limits_{n\to\infty}\sqrt[\leftroot{3}n]{\#\Gamma_{n}}.
$
}

\medskip
Let us consider the continuous linear projection map:
\begin{equation}
\label{eq:11.2.9}
\begin{array}{l}
\pi\ :\ \mathcal{L}_\R\langle\Gamma,G\rangle \longrightarrow\ \R\db[s\db],\quad 
\sum_{S\in \Conf_0}\varphi(S)\cdot a_S  \mapsto \sum_{k=0}^\infty a_{\Gamma_k}s^k .\!\!\!
\end{array}
\end{equation}
In order that the map $\pi$ induces the map 
$\pi_\Omega$ \eqref{eq:11.2.12},  
we consider the next two conditions {\bf S} and {\bf I} 
on the graph $(\Gamma,G)$.

First, let us reformulate the concept of {\it dead} element (c.f.\ Bogopolski \cite{Bo}, \cite{E2}) to a monoid: an element $g\in \Gamma$ 
is called {\it dead} with respect to $G$ if $\ell_G(gx)\le \ell_G(g)$ $\forall x\in G$.
\footnote
{The author is grateful to Takefumi Kondo for the information on some works on the subject.}
We denote by $D(\Gamma,G)$ the set of dead elements in $\Gamma$.

\begin{itemize}
\item 
{\bf S:}  \ The portion $\frac{\#(\Gamma_n\cap D(\Gamma,G))}{\#(\Gamma_n)}$ 
tends to 0 as $n \to \infty$.

 \vspace{-0.2cm}

\item
{\bf I:} \ For any connected subgraph $\mathbb{S}$ of $(\Gamma,G)$ and 
any element $g\in \hat \Gamma$, the 

 \vspace{-0.2cm}

\quad\   equality $\mathbb{S}\Gamma_1=g\mathbb{S}\Gamma_1$
implies $\mathbb{S}=g\mathbb{S}$, where 
$\mathbb{S}\Gamma_1\!:=\!\cup_{\alpha\in\mathbb{S}}\alpha \Gamma_1$.
\end{itemize}	

\noindent
{\bf Assumption 2.} From now on until the end of the present paper, we assume the conditions {\bf S} and {\bf I} hold for  $(\Gamma,G)$.

\begin{rem}
{\bf 1.} Bogopolski (\cite{Bo} Question(2)) asked whether 
{\bf S} holds for arbitrary finite generating system $G$ of a group $\Gamma$.  We ask the same question for a monoid $\Gamma$ satisfying {\bf Assumption 1.} and any finite generating system $G$.

{\bf 2.} Since $\Aut(S)$ is a finite subgroup of $\hat\Gamma$ for  
$S\!\in\! \langle\Gamma,G\rangle_0$, it is trivial
if $\hat \Gamma$ is  torsion free. Then,  {\bf I}
holds automatically for arbitrary finite generating system.

{\bf 3.} If $\Gamma$ has a torsion element $g$ of order 
$d>1$, define a new generating system 
$G':=\cup_{i,j=0,\cdots,d-1}(g^i\Gamma_1g^j)\setminus \{e\}$ for a given $G$. Then, 
the new unit ball $\Gamma_1'\!:=\!G'\cup\{e\}$ satisfies  $\Gamma_1'\!=\! g\Gamma_1'$.
That is, the condition {\bf I} fails for $\mathbb{S}:=\{e\}$. 
This suggests that in order to satisfy {\bf I}, $G$ should be small relative to torsion elements. 
It is an open question whether, for any finitely generated infinite group $\Gamma$, there always exists a generating system $G$
satisfying {\bf I}.
\end{rem}

\noindent
{\it Notation.} We define the {\it fattening} $S \Gamma_1$ for $S\!\in\! \langle\Gamma,G\rangle_0$ by the isomorphism class 
$[\mathbb{S}\Gamma_1]$ 
for any representative $\mathbb{S}$ of $S$ (the isomorphism class $[\mathbb{S}\Gamma_1]$ does not depend on a choice of $\mathbb{S}$ due to the embeddability of $\Gamma$ into a group).

We regard $ \mathcal{L}_\R\langle\Gamma,G\rangle$ as an 
$\R\db[s\db]$-module by letting $s$ 
 act on the basis by $\varphi(S)\!\mapsto\!\varphi(S\Gamma_1\!)$ and extending the action 
 formally to $\R\db[s\db]$. 
However, 
 {\it the map $\pi$ \eqref{eq:11.2.9} is not an $\R\db[s\db]$-homomorphism}
($S\Gamma_1\!=\!\Gamma_{k+1}$ does not imply  $S\!=\!\Gamma_k$).

Let us state some important consequences of the assumptions {\bf S} and {\bf I}. 
Recall the notation \eqref{eq:5.1.1} and \eqref{eq:5.1.2}). 

\begin{assertion}
For any $S\!\in\! \langle\Gamma,G\rangle_0$, one has 
the inequalities:
\begin{equation}
\begin{array}{l}
\label{eq:11.2.10}
0\le A(S\Gamma_1,\Gamma_{n})-A(S,\Gamma_{n-1})
\le \#S\cdot\#(\dot\Gamma_n\cap D(\Gamma,G)).
\end{array}
\end{equation}
\end{assertion}
\begin{proof}
 Consider a map 
$\mathbb{S}\!\in\! \mathbb{A}(S,\Gamma_{n-1})\! \mapsto\! \mathbb{S}\Gamma_1
\!\in\! \mathbb{A}(S\Gamma_1,\Gamma_{n})$. Then the condition 
{\bf I} implies the injectivity of the map.
This implies the first inequality.
Any element of  
$\mathbb{A}(S\Gamma_1,\Gamma_{n})$ is expressed as $\mathbb{S}\Gamma_1$ for a unique $\mathbb{S}\!\subset\!\Gamma_n$
with $[\mathbb{S}]\!=\!S$. 
If $\mathbb{S}\Gamma_1$ is not in the image of the above map (i.e.\ $\mathbb{S}\not\subset \Gamma_{n-1}$), then $\mathbb{S}\cap \dot\Gamma_n\!\not=\!\emptyset$ 
is a subset of $D(\Gamma,G)$. Thus, 
such $\mathbb{S}$ is of the form $ds^{-1}\mathbb{S}_0$ for some
$d\in \dot\Gamma_n\cap D(\Gamma,G)$ and some $s\in \mathbb{S}_0$ for a fixed 
$\mathbb{S}_0$ with $[\mathbb{S}_0]\!=\!S$. Thus 
the number of such  $\mathbb{\mathbb{S}}\Gamma_1$
is at most $\#(S)\cdot \#(\dot\Gamma_n\cap D(\Gamma,G))$. 
This implies the second inequality.
\end{proof}
\begin{cor}  For $n,k\in \Z_{\ge0}$ with $n-k\ge0$, one has the 
inequalities:
\vspace{-0.1cm}
\begin{equation}
\vspace{-0.1cm}
\label{eq:11.2.11}
0\le A(\Gamma_k,\Gamma_{n})-\#(\Gamma_{n-k}) \le  \#(\Gamma_{k-1})\#(\Gamma_n\cap D(\Gamma,G)).
\end{equation}
\end{cor}
\noindent
{\it Proof.} 
We show by induction on k, where $k\!=\!0$ is trivial (put 
$\#\Gamma_{-1}\!:=\!0$).
Assume for $k-1$. Let $n$ be an integer with $n\ge k$. 
Applying \eqref{eq:11.2.10} for $S=\Gamma_{k-1}$, one has
$0\le A(\Gamma_{k},\Gamma_{n})\!-\! \A(\Gamma_{k-1},\Gamma_{n-1})\!\le\! \#\Gamma_{k-1}\cdot\#(\dot\Gamma_n\cap D(\Gamma,G))$. 
This together with the induction hypothesis implies  \eqref{eq:11.2.11}.
\quad  $\Box$

\bigskip
Under {\bf Asumpptions 1} and {\bf 2}, we show the main result of the present section: the map $\pi$ (11.2.9) 
induces the fibration map $\pi_\Omega:\Omega(\Gamma,G)\!\to\! \Omega(P_{\Gamma,G})$. 
The fibration is the key structure of the whole present  paper.

\begin{thm} 
{\bf 1.} If $\underset{n_m\to\infty}{{\lim}^{cl}}
\frac{\mathcal{M}(\Gamma_{n_m})}{\sharp \Gamma_{n_m}}$ converges 
to an element $\omega\in\Omega(\Gamma,G)$, then 
$\underset{n_n\to\infty}{{\lim}^{cl}}
X_{n_m}(P_{\Gamma,G})$
converges to $\pi(\omega)\in \R \db[s\db]$.
We denote by 
\begin{equation}
\label{eq:11.2.12}
\pi_\Omega\ :\ \Omega(\Gamma,G) \ \longrightarrow\ \Omega(P_{\Gamma,G})
\end{equation}
the induced map. The $\pi_\Omega:=\pi|_{\Omega(\Gamma,G)}$ is 
a surjective and continuous map. 

{\bf 2.} If a sequence 
$\left\{\!\frac{\mathcal{M}(\Gamma_{n_m})}{\#\Gamma_{n_m}}\!\right\}_{m\in\Z_{\ge0}}$ 
\!\!\!converges to an element $\omega\in\Omega(\Gamma,G)$, 
then the sequence 
$\left\{\!\frac{\mathcal{M}(\Gamma_{n_m-1})}{\#\Gamma_{n_m-1}}\!\right\}_{m\in\Z_{\ge1}}$\!\!\! 
converges also to an element, depending only on $\omega$, denoted by $\tilde\tau_\Omega(\omega)$.
For 
$\omega=\sum_{S\in\langle\Gamma,G\rangle_0}a_S\varphi(S)\in \Omega(\Gamma,G)$, 
one has
\vspace{-0.2cm}
\begin{equation}
\label{eq:11.2.13}
\begin{array}{l}
\tilde \tau_{\Omega}\left(\omega\right)\ =\ \frac{1}{\iota(\pi_\Omega(\omega))}\sum_{S\in \langle\Gamma,G\rangle_0}. a_{S\Gamma_1}\varphi(S).
\end{array}
\end{equation}
Using the notation $\partial_S$\! and $\partial_{S\Gamma_1}$\!
for $S\!\in\! \langle\Gamma,G\rangle_0$ in\! \S{\rm 8.1}, {\rm (11.2.13)} is equivalent to\!\! 
\[\begin{array}{ll}
\partial_S(\tilde{\tau}_\Omega\omega)\ =\ \frac{1}{\iota(\pi_\Omega(\omega))}\partial_{S\Gamma_1}(\omega) .
\end{array}
\leqno{(11.2.13)^*}
\]
Then, $\pi_\Omega$ \eqref{eq:11.2.12} 
is equivariant with respect to the actions 
$\tilde \tau_\Omega$ and $\tau_\Omega$.

{\bf 3.} Let us denote by $\overline{\R\Omega}(\Gamma,G)$ the closed $\R$-linear subspace 
of $\mathcal{L}_{\R,\infty}$ generated by $\Omega(\Gamma,G)$. Define 
a map $\tilde \tau$ from $\Omega(\Gamma,G)$ to  $\overline{\R\Omega}(\Gamma,G)$ by 
\begin{equation}
\label{eq:11.2.14}
\begin{array}{l}
\tilde \tau(\omega) \ :=\ \iota(\pi_\Omega(\omega))\ \tilde \tau_{\Omega}(\omega).
\end{array}
\end{equation}
Then, $\tilde \tau$ naturally extends to an $\R$-linear endomorphism of 
$\overline{\R\Omega}(\Gamma,G)$.

{\bf 4.} The restriction of $\pi$ \eqref{eq:11.2.9} ($=$ the $\R$-linear extension of $\pi_\Omega$):
\begin{equation}
\label{eq:11.2.15}
\pi\ :\ \overline{\R \Omega}(\Gamma,G) \longrightarrow \overline{\R \Omega}(P_{\Gamma,G}).
\end{equation}
is equivariant with respect to the  endomorphisms $\tilde \tau$ and $\tau$, \ 
i.e.\ $\tau\circ\pi=\pi\circ\tilde\tau$.
\end{thm}
\vspace{-0.3cm}
\begin{proof} 
1. 
Using \eqref{eq:8.2.3}, \eqref{eq:11.2.2} and \eqref{eq:11.2.6}, 
we see that the difference
$\pi\left( \frac{\mathcal{M}(\Gamma_n)}{\#\Gamma_n}\right)\!- \! X_n(P_{\Gamma,G})$ is a polynomial in $s$ of degree $\le n$ whose 
$k$th coefficient is
$(A(\Gamma_k,\Gamma_n)-\#\Gamma_{n-k})/\#\Gamma_n$.
Put $n=n_m$ and take the limit by letting $m\to \infty$. Applying 
\eqref{eq:11.2.11} and the assumption {\bf S}, we see 
that this tends to 0. That is, $k$th coefficient of $X_{n_m}(P_{\Gamma,G})$ 
tends to the coefficient $a_{\Gamma_k}$ at $\Gamma_k$ of $\omega$.
That is, $X_{n_m}(P_{\Gamma,G})$  converges to the 
$\pi$-image of $\omega$. Thus the map $\pi_\Omega$ \eqref{eq:11.2.12} is defined.
To show the surjectivity of $\pi_\Omega$, for any subsequence $\{X_{n_m}(P_{\Gamma,G})\}_{m}$ converging to an opposite sequence, we  choose a convergent sub-subsequence $\{\frac{\mathcal{M}(\Gamma_{n_{m_l}})}{\sharp \Gamma_{n_{m_l}}}\}_{l}$ due to the compactness of $\overline{\log(EDP)}$ (10.3).

\medskip
2. For 
$S\!\in\! \langle \Gamma,G\rangle_0$ and $n\!\in\!\Z_{\ge 1}$,
one has  
\[
\begin{array}{ll}
\frac{A(S,\Gamma_{n-1})}{\#\Gamma_{n-1}}\ =\ 
\left(\frac{A(S\Gamma_1,\Gamma_{n})}{\#\Gamma_{n}}-
\frac{A(S\Gamma_1,\Gamma_{n})-A(S,\Gamma_{n-1})}{\#\Gamma_{n}}\right)
\Big/\frac{\#\Gamma_{n-1}}{\#\Gamma_{n}}
.
\end{array}
\leqno{*)}
\]
Let the sequence $\frac{\mathcal{M}(\Gamma_{n_m})}{\#\Gamma_{n_m}}$ 
associated to a subsequence $\{n_m\}_{m\in\Z_{\ge0}}$ of $\Z_{\ge0}$
converges to an element 
$\omega=\sum_{S\in\langle\Gamma,G\rangle_0}a_S\varphi(S)\in \Omega(\Gamma,G)$. 
Put $n\!=\!n_m$ in $*)$ and let the index $m$ run to $\infty$.
The first (resp.\ second) term in the bracket  in the RHS of $*)$
converges to $a_{S\Gamma_1}$ (resp.\ 0  due to 
\eqref{eq:11.2.10} and the assumption {\bf S}). 
The denominator of the RHS of $*)$ converges to the initial $\iota(\pi(\omega))$ (11.2.4).  Consequently,
the RHS of $*)$ converges to $\frac{1}{\iota(\pi(\omega))}a_{S\Gamma_1}$ 
for all $S$.
This implies the convergence of ${\lim\limits_{m\to\infty}}^{\!\!\!cl} \frac{\mathcal{M}(\Gamma_{n_m-1})}{\#\Gamma_{n_m-1}}$ 
and the formula \eqref{eq:11.2.13} and (11.2.13)$^*$.

\noindent
\ \ \ Let  $a\!=\!\pi_\Omega(\omega) (:=\!\!\sum_{k=0}^\infty \!a_{\Gamma_k}s^k)$.
Comparing the formulae \eqref{eq:11.2.5} and \eqref{eq:11.2.13}, 
one calculates: 
%
  $ \pi_\Omega(\tilde\tau_{\Omega}(\omega)) 
=  \frac{1}{\iota(\pi(\omega))}\sum_{k=0}^\infty a_{\Gamma_k\Gamma_1}s^k 
=  \frac{1}{\iota(\pi(\omega))}\sum_{k=0}^\infty a_{\Gamma_{k+1}}s^k =$ 

\noindent
$ \frac{1}{\iota(\pi(\omega))}\sum_{l=1}^\infty a_{\Gamma_{l}}s^{l-1}
=\tau_\Omega(a)=\tau_\Omega(\pi_\Omega(\omega))$.
This implies that the map $\pi_\Omega$ is equivariant with the $(\tilde\tau_{\Omega},\tau_{\Omega})$-action.

\medskip
3. Let 
(r): $\sum_{i\in I} c_i\omega_i\!=\!0$ be a linear  relation 
for $\omega_i\!\in\! \Omega(\Gamma,G)$ 
and $c_i\!\in\! \R$ ($i\!\in\! I$) with $\#(I)\!<\!\infty$. 
Let us show the linear  relation 
(s): $\sum_{i\in I} c_i\tilde \tau(\omega_i)\!=\!0$.
Let us expand $\omega_i\!=\!\sum_S a_{S,i}\varphi(S)$. 
Then, the relation (r) is expressed as the relations 
$\sum_{i\in I} c_i a_{S,i}\!=\!0$ of coefficients for all $S\!\in\! \langle\Gamma,G\rangle_0$.
Then the relation (s) is expressed as $\sum_{i\in I}c_i a_{S\Gamma_1,i}=0$ 
for all $S\!\in\! \langle\Gamma,G\rangle_0$, which is a part of the former relations of the 
coefficients and is automatically satisfied.

This implies that $\tilde \tau$ extends to a linear map 
${\R\Omega}(\Gamma,G)\to \overline{\R\Omega}(\Gamma,G)$.
On the other hand, the correspondence $\sum_{S\in \langle\Gamma,G\rangle}a_S\varphi(S) \mapsto \sum_{S\in \langle\Gamma,G\rangle}a_{S\Gamma_1}\varphi(S)$
defines a redefined continuous linear map from a closed subspace of 
$\mathcal{L}_{\R}$ to itself, which induces the endomorphism 
$\tilde\tau\in \mathrm{End}_\R(\overline{\R\Omega}(\Gamma,G))$.

\medskip
4. Let the notation be as in 1.
Comparing (11.2.5)* and \eqref{eq:11.2.14}, one calculates: \ 
 $ \pi(\tilde\tau(\omega))= \pi(\iota(\pi(\omega)) \tilde\tau_{\Omega}(\omega))
= \iota(\pi(\omega))\pi_\Omega (\tilde\tau_{\Omega}(\omega))
= \iota(\pi(\omega)) \tau_\Omega(\pi_\Omega(\omega))
= \iota(\pi(\omega)) \tau_\Omega(a)
= \ \tau(a)=\tau(\pi(\omega))$.
This implies the equivariance of $\pi$.
\end{proof}

The map $\pi_\Omega$ \eqref{eq:11.2.12} is conjecturally a finite map. 
In that case, the sum of elements in a fiber is called a {\it trace}, and we, in 11.5, represent the traces  by
suitable ``residue values'' of the functions  \eqref{eq:11.2.6} and 
\eqref{eq:11.2.7}.
The key to understand this formula is the 
{\it ``duality'' between the limit space $\Omega(P_{\Gamma,G})$ 
and the space of singularities $Sing(P_{\Gamma,G})$ of the series 
$P_{\Gamma,G}(t)$
on the circle of the convergent radius $r_{\Gamma,G}$.
} 
In next sections 11.3 and 11.4, we study the ``duality''  
in case  $\Omega(P_{\Gamma,G})$ is finite
(see 11.4 {\bf Theorem} and \eqref{eq:11.4.3} and \eqref{eq:11.4.4}). 

\medskip
In the following, we give an example of $(\Gamma,G)$, 
where $\Omega(P_{\Gamma,G})$ consists of two elements $a^{[0]}$ 
and $a^{[1]}$, and $\tau_\Omega$ 
acts on $\Omega(P_{\Gamma,G})$ as their transposition. However, 
we note that  $\tau^2\not=\id$ and $\det(t\cdot\id-\tau)=t^2-1/2$.
\vspace{-0.1cm}
\begin{example}\!\!\!\!\!\!(Mach\`i)\ \ \ 
Let $\Gamma:=\Z/2\Z *\Z/3\Z$ and 
$G:=\{a,b^{\pm1}\}$ where $a,b$ are the generators of $\Z/2\Z$ and $\Z/3\Z$, 
respectively.
Then, Mach\`i has shown
\vspace{-0.1cm}
\[\begin{array}{cccc}
P_{\Gamma,G}(t)&\ :=\ & \sum_{k=0}^\infty \#\Gamma_{k}t^{k} & =\ \frac{(1+t)(1+2t)}{(1-2t^2)(1-t)},
\end{array}
 \vspace{-0.1cm}
\] 
so that
$\#\Gamma_{2k}\!=\! 7\cdot2^k\!-\!6$ and 
$\#\Gamma_{2k+1}\! =\! 10\cdot2^k\!-\!6$ for $k\!\in\!\Z_{\ge0}$.
Then, one has   

\vspace{0.1cm}

\centerline{
$\Omega_1(P_{\Gamma,G})\! =\! \Big\{
\iota(a^{[0]})\!:=\!\underset{n\to \infty}{\lim}\!\frac{\#\Gamma_{2n-1}}{\#\Gamma_{2n}}\!=\!\frac{5}{7} \ \&\  
\iota(a^{[1]})\!:=\!\underset{n\to \infty}{\lim}\!\frac{\#\Gamma_{2n}}{\#\Gamma_{2n+1}}\!=\!\frac{7}{10}\Big\}$, 
}

\vspace{0.1cm}

\noindent
and, hence $h_{\Gamma,G}\!=\!2$.
In fact, $\Omega(P_{\Gamma,G})$ consists of two 
opposite sequences: 

\vspace{-0.4cm}

 {\footnotesize
\[
\qquad \quad\quad
\begin{array}{lll}
a^{[0]}(s)&\!: =\!\sum_{k=0}^\infty 2^{-k}s^{2k}\!+\!\frac{5}{7}s\sum_{k=0}^\infty 2^{-k}s^{2k} &\! =\! 
{(1\!+\!\frac{5}{7}s)}/{(1\!-\!\frac{s^2}{2})}
 \! =\! \frac{\frac{7\!+\!5\sqrt{2}}{14}}{1\!-\!\frac{s}{\sqrt{2}}}\!+\!\frac{\frac{7\!-\!5\sqrt{2}}{14}}{1\!+\!\frac{s}{\sqrt{2}}}.
 \\
a^{[1]}(s)&\!: =\!\sum_{k=0}^\infty 2^{-k}s^{2k}\!+\!\!\frac{7}{10}s\sum_{k=0}^\infty 2^{-k}s^{2k} &\! =\! 
{(1\!+\!\!\frac{7}{10}s)}/{(1\!-\!\frac{s^2}{2})}
\! =\! \frac{\frac{10\!+\!7\sqrt{2}}{20}}{1\!-\!\frac{s}{\sqrt{2}}}\!+\!\frac{\frac{10\!-\!7\sqrt{2}}{20}}{1\!+\!\frac{s}{\sqrt{2}}}.
\end{array}
\vspace{-0.1cm}
\]
}
\end{example}

\subsection{Finite rational accumulation}
\label{subsec:11.3}

We introduce the concept of a {\it finite rational accumulation}, 
and study the series $P(t)$ \eqref{eq:11.2.1} 
from that viewpoint.
First, we start with preliminary definitions.

\begin{definition}
{\bf 1.} A subset $U$ of $\Z_{\ge0}$ is called a {\it rational subset} 
if the sum 
$U(t):=\sum_{n\in U}t^n$ is the Taylor expansion at $0$ of 
a rational function in $t$. 

{\bf 2.}\ A {\it finite rational partition} of $\Z_{\ge0}$
is a finite collection $\{U_a\}_{a\in \Omega}$ of rational 
subsets $U_a\!\subset\! \Z_{\ge0}$ 
indexed by a finite set $\Omega$ 
such that 
 there is a finite subset $D$ of $\Z_{\ge0}$ so that
one has the disjoint decomposition 
$\Z_{\ge0}\setminus D=\sqcup_{a\in \Omega}(U_a\setminus D)$. 
\end{definition}

\begin{assertion}
\!\!\!For any rational subset $U$ of $\Z_{\ge0}$, there exist a positive 
integer $h$, a subset $u\!\subset\!\Z/h\Z$ and a finite subset 
$D\!\subset\! \Z_{\ge0}$ 
such that $U\!\setminus\! D\! =\! \cup_{[e]\in u}U^{[e]}\!\setminus\! D$, 
where 
$[e]$ denotes the element of $\Z/h\Z$ corresponding to $e\!\in\!\Z$
and $U^{[e]}\!:=\!\{n\!\in\!\Z_{\ge0}\mid n\! \equiv\! e \bmod h \}$.\
We call $\cup_{[e]\!\in\! u}U^{[e]}$ the standard expression of $U$.\!
\end{assertion}
\begin{proof} 
The fact that $U(t)$ is rational implies that the function $\chi:\Z_{\ge0}\to\{0,1\}\ (\chi(n)\!=\!1\leftrightarrow n\!\in\!U$) is recursive, i.e.\ there exist $N\!\in\!\Z_{\ge1}$ 
and numbers $\al_1,\!\cdots\!,\al_N$ such that
one has the recursive relation 
$\chi(n)\!+\!\chi(n\!-\!1)\al_1\!+\!\cdots\!+\!\chi(n\!-\!N)\al_N\!=\!0$
for sufficiently large $n\gg0$. Since the range of 
$\chi$ is finite, there exist two large numbers $n\!>\!m$ such that
$\chi(n\!-\!i)\!=\!\chi(m\!-\!i)$ for $i\!=\!0,\!\cdots\!,N$. Due to the recursive relation,
this means $\chi$ 
is $h\!:=\!n-m$-periodic after $m$.
\end{proof}

\begin{cor} Any finite rational partition of  $\Z_{\ge0}$ has a 
subdivision of the form $\mathcal{U}_h:=\{U^{[e]}\}_{[e]\in \Z/h\Z}$ 
for some $h\in\Z_{>0}$, called a {\it period} of the partition. If $h$ is the 
minimal period,
$\mathcal{U}_h$ is called the {\it standard subdivision} of the partition.
\end{cor}

\begin{definition}
A sequence $\{X_n\}_{n\in\Z_{\ge0}}$\! in a Hausdorff space 
is {\it finite rationally accumulating} if the sequence
accumulates to a finite set, say $\Omega$, such that for a system
 of open neighborhoods $\mathcal{V}_a$ 
for $a\!\in\!\Omega$ 
with $\mathcal{V}_a\!\cap\!\mathcal{V}_b\!=\!\emptyset$ if $a\!\not=\!b$,
the system $\{U_a \}_{a\in\Omega}$ for 
$U_a\!:=\!\{n\!\in\!\Z_{\ge0}\mid X_n\!\in\!\mathcal{V}_a\}$ 
is a finite rational partition of $\Z_{\ge0}$. 
We say also that {\it $\Omega$ is 
a finite rationally accumulation set of period $h$}.
 \end{definition}

The next and 11.5 Lemmas are key facts, 
which justify the introduction of the concept
``rational accumulation''. They are also the 
starting point of the concept of {\it periodicity} which is the thorough bass 
of the whole study in sequel.

\begin{lemma}
Let $P(t)$ be a power series in $t$ as given in \eqref{eq:11.2.1}.\ 
If $\Omega(P)$ is finite, then it is a finite rationally 
accumulation set with respect to the standard partition 
$\mathcal{U}_h$\! of $\Z_{\ge0}$ for some $h\!\!>\!\!0$,\!
and $\tau_\Omega$\! acts transitively on $\Omega(P)$\! of period\! $h$.
\end{lemma} 
\begin{proof}
Recall the $\tau_\Omega$-action on the set $\Omega(P)$ in 11.2. 
Since $\Omega(P)$ is finite, there exists a non-empty $\tau_\Omega$-invariant
subset of $\Omega(P)$. More explicitly, there exists an element 
$a\in\Omega(P)$ and a positive integer $h\in\Z_{>0}$ such that 
$(\tau_\Omega)^ha\!=\!a\!\not=\!(\tau_\Omega)^{h'}a$ for $0\!<\!h'\!<\!h$.
Put 
$U_a\!:=\!\{n\in\Z_{\ge0}\!\mid\! X_n(P)\!\in\! \mathcal{V}_a\}$ where
$\{\mathcal{V}_a\}_{a\in \Omega(P)}$ is a system of open neighborhoods 
of points of $\Omega(P)$ such that 
$\mathcal{V}_a\!\cap\! \mathcal{V}_b\!=\!\emptyset$ for 
%
any $a\not=b \!\in\!\Omega(P)$. 
By the definition of $\tau_\Omega$, the relation 
$(\tau_\Omega)^ha\!=\!a$ implies that 
the sequence $\{X_{n-h}(P)\}_{n\!\in\! U_a}$ converges to $a$. 
That is, there 
exists a positive number $N$ such that for 
any $n\!\in\! U_a$ with $n\!>\!N$, 
$n\!-\!h$ is contained in $U_a$. Consider the set $A\!:=\!\{[e]\!\in\! \Z/h\Z\mid $
there are infinitely many elements of $U_a$ which are congruent to 
$[e]$ modulo $h$ $\}$. Then, $U_a$ is, up to a finite number 
of elements, equal to the rational set $\cup_{[e]\!\in\! A} U^{[e]}$. This implies 
$A\!\not=\!\emptyset$. Further more, $U_{(\tau_\Omega)^ia}$ is 
also, up to a finite number of elements, equal to the rational set 
$\cup_{[e]\in A} U^{[e-i]}$. Then, the union 
$\cup_{i=0}^{h-1} U_{(\tau_\Omega)^ia}$ already covers $\Z_{\ge0}$ 
up to finite elements. Since there should not be an overlapping, 
$\#A\!=\!1$, say $A\!=\!\{[0]\}$. If a subsequence $\{X_{n_m}(P)\}$ converges to an element
in $\Omega(P)$, then there is at least one $[e]\in\Z/h\Z$ such that 
$\#(\{n_m\}_{m=0}^\infty\! \cap\! U^{[e]})\!=\!\infty$ so that it converges 
to $(\tau_\Omega)^{h-e}a$.
That is, 
$\Omega(P)$ is equal to the set 
$\{a,\tau_\Omega a,\cdots,(\tau_\Omega)^{h-1}a\}$, 
%
which  
is a finite rationally accumulating set with the $h$-periodic 
action of $\tau_\Omega$.
\end{proof}

In the sequel, we analyze the finite accumulation set $\Omega(P)$ in detail.
\begin{assertion} 
Let $P(t)$ be a power series in $t$  as given in \eqref{eq:11.2.1}. 

{\bf 1.} $\Omega(P)$ is a finite rationally accumulation set of period 
$h\!\in\!\Z_{\ge1}$ if and 
only if $\Omega_1(P)$ is.\
We say $P$ is finite rationally accumulating of period $h$.

{\bf 2.} Let  $P$ be finite rationally accumulating of period  
$h\in\Z_{\ge1}$. 
Then the opposite series $a^{[e]}=\sum_{k=0}^\infty a_k^{[e]} s^k$ 
in $\Omega(P)$ associated to 
the rational subset $U^{[e]}$
for $[e]\in \Z/h\Z$ of 
\vspace{-0.3cm}
the $h$-partition of $\Z_{\ge0}$ converges to a rational function   
\begin{equation}
\label{eq:11.3.1}
a^{[e]}(s)\ =\ \frac{A^{[e]}(s)}
{1-r^hs^h},
\vspace{-0.3cm}
\end{equation}
where the numerator $A^{[e]}(s)$ is a polynomial in $s$ of degree 
$h\!-\!1$ given by
\vspace{-0.2cm}
\begin{equation}
\label{eq:11.3.2}
\vspace{-0.2cm}
 \begin{array}{rll}
\qquad A^{[e]}(s)\ := \
\sum_{j=0}^{h-1}\left(\prod_{i=1}^{j}a_1^{[e-i+1]}\right)s^{j} \quad\&
\end{array}
\vspace{-0.3cm}
\end{equation}
\begin{equation}
\begin{array}{rll}
 \label{eq:11.3.3}
r^h \ :=\ \ \prod_{i=0}^{h-1}a_1^{[i]}.
\end{array}
\vspace{-0.1cm}
\end{equation}
The $h$th positive root $r\!>\!0$ of \eqref{eq:11.3.3} 
is the radius of convergence of $P(t)$.

{\bf 3.} If the period $h$ is minimal, then 
the  opposite sequences $a^{[e]}(s)$ for $[e]\in \Z/h\Z$ are mutually distinct.
That is, $\Omega(P)\simeq \Z/h\Z$, $a^{[e]}(s) \leftrightarrow [e]$
and the standard 
partition $\mathcal{U}_h$ 
is the exact partition of $\Z_{\ge0}$ 
for the opposite series $\Omega(P)$.
\end{assertion}
\vspace{-0.2cm}
\begin{proof} 
1. The necessity is obvious. 
To show sufficiency, assume that
$\{\gamma_{n-1}/\gamma_n\}_{n\in\Z_{\ge0}}$ 
accumulate finite rationally of period $h$. 
Let 
the subsequence 
$\{\gamma_{n-1}/\gamma_n\}_{n\in U_{[e]}}$ 
for $[e]\in\Z/h\Z$ accumulate to a unique value $a_1^{[e]}$. 

For any $k\in \Z_{\ge0}$, one has the obvious relation:
\vspace{-0.2cm}
\[
\frac{\gamma_{n-k}}{\gamma_n}\ =\
\frac{\gamma_{n-1}}{\gamma_n}\frac{\gamma_{n-2}}{\gamma_{n-1}}\cdots
\frac{\gamma_{n-k}}{\gamma_{n-k+1}}.
\vspace{-0.2cm}
\]
For $n\!\in\!  U_{[e]}\!=\!\{n\!\in\!\Z_{\ge0}\mid n\!\equiv\! e \bmod h\}$
for  $[e]\!\in\! \Z/h\Z$, we see that the RHS converges to
$a_1^{[e]}a_1^{[e-1]}\ldots a_1^{[e-k+1]}$. 
Then, for $[e]\!\in\! \Z/h\Z$ and $k\in\! \Z\!_{\ge0}$, by putting
\begin{equation}
\label{eq:11.3.4}
a_k^{[e]}\ :=\ a_1^{[e]}a_1^{[e-1]}\ldots a_1^{[e-k+1]},
\end{equation}
the sequence
$\{X_n(P)\}_{n\in  U_{[e]}}$ 
converges to 
$a^{[e]}\!:=\!\sum_{k=0}^\infty a_k^{[e]}s^k$ 
with $a_1^{[e]}=\iota(a^{[e]})$.

2. Define $r^h$ by the relation \eqref{eq:11.3.3}. 
Then, the formula \eqref{eq:11.3.4} implies the ``{\it periodicity}'' 
$a_{mh+k}^{[e]}\!=\!r^{mh}a_k^{[e]}$ for $m\!\in\!\Z_{\ge0}$. This implies 
\eqref{eq:11.3.1}. 

To show that $r$ is the radius of convergence of $P(t)$, it is sufficient to
show:

\medskip
\noindent
{\bf Fact.} {\it Let $P(t)$ be finite rationally accumulating of period $h$. 
Define $r\ge0$ by the relation \eqref{eq:11.3.3}.
There exist 
positive real constants $c_1$ and $c_2$ such that for any $k\in\Z_{\ge0}$
there exists $n(k)\in\Z_{\ge0}$ and for any integer $n\ge n(k)$, one has
$
c_1 r^k\le \frac{\gamma_{n-k}}{\gamma_n}\le c_2r^k
$.
%
}
 
{\it Proof.}
Choose $c_1,c_2\in\R_{>0}$ satisfying 
$c_1\!\!<\!\!\min\{\frac{a_i^{[e]}}{r^i}\!\mid\! [e]\!\in\!\Z/h\Z, i\!\in\!\Z\cap[0,h\!-\!1]\}$
and
$c_2\!>\!\max\{\frac{a_i^{[e]}}{r^i}\mid [e]\in\Z/h\Z, i\!\in\!\Z\cap[0,h\!-\!1]\}$.
\qquad $\Box$
        
3. Suppose $a^{[e]}(s)=a^{[f]}(s)$ for some $[e],[f]\in\Z/h\Z$. 
Then, by comparing the coefficients of 
$A^{[e]}(s)$ and $A^{[f]}(s)$, we get 
$a_1^{[e-i]}=a_1^{[f-i]}$ for $i\!=\!0,\!\cdots\!,h\!-\!1$. 
This means $e\!-\!f$ is a period. The minimality of $h$ implies $[e\!-\!f]=0$.
\end{proof}

Even if, as in the Assertion, the opposite series $a^{[e]}(s)$ 
for $[e]\!\in\! \Z/h\Z$ are 
mutually distinct for the minimal period $h$ of $P(t)$, 
they may be linearly dependent. 
This phenomenon occurs at the zero-loci of the determinant
\begin{equation}
\begin{array}{l}
\label{eq:11.3.5}
D_h(a_1^{[0]},\cdots,a_1^{[h-1]})\ :=\ \det\left((\prod_{i=1}^fa_1^{[e-i+1]})_{e,f\in\{0,1,\cdots,h-1\}}\right).
\end{array}
\end{equation}

\vspace{-0.2cm}
Regarding $a_1^{[0]},\cdots,a_1^{[h-1]}$ as indeterminates,
$D_h$ is an irreducible homogeneous polynomial  
of degree $h(h-1)/2$, which is neither symmetric nor anti-symmetric, 
but anti-invariant under a cyclic permutation 
(depending on the parity of $h$). 
Let us formulate more precise statements for an arbitrary field $K$.

\begin{assertion}
Let $h\!\in\!\Z_{>\!0}$. For an $h$-tuple 
$\bar{a}\!=\!(a_1^{[0]},\cdots\!,a_1^{[h-1]})\!\in\! (K^\times)^h$, 
define polynomials $A^{\![e]}(s)$ $([e]\!\in\!\!\Z\!/h\Z)$ 
and $r^h\!\in\!\! K^{\!\times}$ by \eqref{eq:11.3.2} and \eqref{eq:11.3.3}. 

{\rm i)} In the ring $K[s]$, the greatest common divisors   
 $\gcd(A^{[e]}(s), 1\!-\!r^hs^h)$  and $\gcd(A^{[e]}(s),A^{\![e+1]}(s))$
for all $[e]\!\in\!\Z/h\Z$ are the same up to factors in $K^{\!\times}$.
Let $\delta_{\bar{a}}(s)$ be the common divisor 
whose constant term is normalized to 1.
Put
\vspace{-0.1cm}
\begin{equation}
\label{eq:11.3.6}
\Delta_{\bar{a}}^{op}(s)\ :=\ (1-r^hs^h)/\delta_{\bar{a}}(s). 
\vspace{-0.1cm}
\end{equation}

{\rm ii)}
For $[e]\!\in\! \Z/h\Z$, let $a^{[e]}(s)=b^{[e]}(s)/\Delta_{\bar{a}}^{op}(s)$ 
be the reduced expression (i.e.\ $b^{[e]}(s) $ is a polynomial of degree 
$<\!\deg(\Delta_{\bar{a}}^{op})$ and 
$\gcd(b^{[e]}(s),\Delta_{\bar{a}}^{op}(s))=1$).
Then, the polynomials $b^{[e]}(s)$ for $[e]\!\in\! \Z/h\Z$ 
span the space $K[s]_{<\deg(\Delta_{\bar a}^{op})}$ of 
polynomials of degree less than $\deg (\Delta_{\bar{a}}^{op})$.
One has the equality:
\vspace{-0.1cm}
\begin{equation}
\begin{array}{c}
\label{eq:11.3.7}
\rank\left((\prod_{i=1}^fa_1^{[e-i+1]})_{e,f\in\{0,1,\cdots,h-1\}}\right)
\ =\ \deg (\Delta_{\bar{a}}^{op}).
\end{array}
\end{equation}

\vspace{-0.1cm}

{\rm iii)} Let $K\!=\!\R$ and $\bar{a}\in(\R_{>\!0})^h$. 
Then, $\Delta_{\bar{a}}^{op}$ is divisible by  $1\!-\!rs$. Conversely, let
$\Delta^{op}$ be a factor of $1\!-\!r^hs^h$ which is divisible by $1\!-\!rs$ 
for  $r\in \R_{>0}$ 
with the constant term 1. 
Then there exists a smooth non-empty semialgebraic set $C_{\Delta^{op}}\!\subset\! (\R_{>0})^h$ 
of dimension $\deg(\Delta^{op})\!-\!1$ such that 
$\Delta^{op}\!=\!\Delta_{\bar{a}}^{op}$ for all $\bar{a}\!\in\! C_{\Delta^{op}}$.
\end{assertion}

\vspace{-0.1cm}
\begin{proof} 
i) By the definitions \eqref{eq:11.3.3} and \eqref{eq:11.3.4}, 
we have the relations:
\begin{equation}
\label{eq:11.3.8}
a_1^{[e+1]}s A^{[e]}(s)+ (1-r^hs^h)\ =\ A^{[e+1]}(s) 
\end{equation}
for  $[e]\in \Z/h\Z$. 
This implies
$\gcd(A^{[e]}(s),1-r^hs^h)\mid \gcd(A^{[e+1]}(s),1-r^hs^h)$ for 
$[e]\in \Z/h\Z$ so that one concludes that all the elements
$ \gcd(A^{[e]}(s),1-r^hs^h)$ $=\gcd(A^{[e]}(s),A^{[e+1]}(s))$ 
for $[e]\in \Z/h\Z$ are the same up to a constant factor.

ii) Let us show that 
the images in $K[s]/(\Delta_{\bar{a}}^{op})$ of the polynomials 
$A^{[e]}(s)/\delta_{\bar{a}}(s)$ for $[e]\in \Z/h\Z$ 
span the entire space over $K$.
Let $V$ be the space spanned by the images.
The relation \eqref{eq:11.3.8} implies that $V$ is closed under the 
multiplication of $s$. On the other hand,  $A^{[e]}(s)/\delta_{\bar{a}}(s)$ 
and $\Delta_{\bar{a}}^{op}$ are relatively prime so that they
generate 1 as a $K[s]$-module. That is, $V$ contains the class [1] of 1, and, 
hence, $V$ contains the whole $K[s]\cdot[1]$. 
Since $\deg(A^{[e]}(s)/\delta_{\bar{a}}(s))<\deg(\Delta_{\bar{a}}^{op})$, 
this means that the polynomials 
$A^{[e]}(s)/\delta_{\bar{a}}(s)$ for $[e]\in \Z/h\Z$ span the space of polynomials 
of degree less than $\deg(\Delta_{\bar{a}}^{op})$. 
In particular, one has $\rank_KV=\deg(\Delta_{\bar{a}}^{op})$. 

By definition,
$\rank(\left((\prod_{i=1}^fa_1^{[e-i+1]})_{e,f\in\{0,1,\cdots,h-1\}}\right))$
is equal to the rank of the space spanned by 
$A^{[e]}(s)$ for $[e]\in \Z/h\Z$, which is equal to 
the rank of the space spanned by 
$A^{[e]}(s)/\delta_{\bar{a}}(s)$ for $[e]\in \Z/h\Z$ 
and is equal to  $\deg(\Delta_{\bar{a}}^{op})$.

iii) If $(1\!-\!rs\!) \not\!|  \ \Delta_{\bar a}^{op}$, then 
$1\!-\!rs\mid \delta_{\bar{a}}\mid A^{[e]}(s)$ and 
$A^{[e]}(1/r)\!=\!0$. This is impossible 
since all coefficients of $A^{[e]}$ and $1/r$ are positive.
Conversely, let $\Delta^{op}$ be a factor of $1\!-\!r^hs^h$ which is divisible 
by $1-rh$, whose degree is $d\!>\!0$. Put 
$\R[s]_{d-1}\! :=\! \{c(s)\in \R[s]\mid \deg(c(s))\!=\!d\!-\!1, \ c(0)\!=\!1\}$. Consider the set 

\smallskip
\noindent
\quad  $\overline C_{\Delta^{op}}
\!:=\!\{ c(s)\!\in\! \R[s]_{d-1}  
\ \mid$ 
all coefficients of 
$ c'(s)\!:=\!c(s)\frac{1\!-\!r^hs^h}{\Delta^{op}} \text{ are positive}\}$.

\smallskip
\noindent
Since $\overline C_{\Delta^{op}}$ is defined by strict inequalities, it is an open subset of $\R[s]_{d-1}$. 
Further,  it is non-empty 
since it contains $\Delta^{op}/ (1\!-\!rs)$.
For any $c(s)\!\in \! \overline C_{\Delta^{op}}$, we note that 
$\deg(c'(s))\!=\!h\!-\!1$, and hence one can find 
a unique $\overline a\!\in\!\!(\R_{>0}\!)^h$ satisfying $c'(s)\!=\!A^{[0]}(s)$ 
\eqref{eq:11.3.2} and \eqref{eq:11.3.3}. 
By this correspondence $c(s)\mapsto \overline a$,
 we embed  $\overline C_{\Delta^{op}}$ 
smoothly to a smooth semialgebraic subset of $(\R_{>0})^h$ 
of dimension $d\!-\!1$. If $\overline a$ is the image of 
$c(s)\in \overline C_{\Delta^{op}}$, then 
$\delta_{\overline a}\!:=\!\gcd\{ c'(s), 1-r^hs^h\}$ is divisible 
by $(1\!-\!r^hs^h)/\Delta^{op}$. That is, $\Delta_{\overline a}^{op}:=(1-r^hs^h)/\delta_{\overline a}$ is a factor of $\Delta^{op}$. This implies that 
the $c(s)$ is a point of the embedded image $C_{\Delta_{\overline a}^{op}}\to C_{\Delta^{op}}$ (defined by the multiplication of $\frac{\Delta^{op}}{\Delta_{\overline a}^{op}}$). 
Define the semialgebraic set
$C_{\Delta^{op}}\!:=\! \overline C_{\Delta^{op}}\!\setminus\! \cup_{\Delta'} \overline C_{\Delta'}$, 
where the index $\Delta'$ runs over all factors of $\Delta^{op}$ (over $\R$)
which are not equal to $\Delta^{op}$ and are divisible by $1\!-\!rs$. 
Since  $\dim_\R(\overline C_{\Delta})=d\!-\!1>\dim_\R(\overline C_{\Delta'})$ so that 
the difference $C_\Delta$ is non-empty.
\end{proof}
\noindent
Suppose the characteristic of the field $K$ is equal to zero. 
Let $\tilde K$ be the 
splitting field of $\Delta_{\bar{a}}^{op}$ with the decomposition 
$\Delta_{\bar{a}}^{op}\!=\!\prod_{i=1}^{d}(1\!-\!x_is)$ in $\tilde K$ 
for $d:=\deg(\Delta_{\bar{a}}^{op})$. 
Then, one has the partial fraction decomposition:
\begin{equation}
\label{eq:11.3.9}
\begin{array}{l}
\frac{A^{[e]}(s)}{1-r^hs^h}\ =\ \sum_{i=1}^{d}\frac{\mu^{[e]}_{x_i}}{1-x_is}
\end{array}
\end{equation}
for $[e]\in \Z/h\Z$, where $\mu^{[e]}_{x_i}$ is a constant in $\tilde K$ 
given by the residue:
\begin{equation}
\label{eq:11.3.10}
\begin{array}{l}
\qquad
\mu^{[e]}_{x_i}\ =\ \frac{A^{[e]}(s)(1-x_is)}{1-r^hs^h} \Bigr|_{s=(x_i)^{-1}} 
\ =\ \frac{1}{h}A^{[e]}(x_i^{-1}) .
\end{array}
\end{equation}
Here, one has the equivariance 
$\sigma(\mu^{[e]}_{x_i})\!=\!\mu^{[e]}_{\sigma(x_i)}$ 
with respect to the action of $\sigma\in {\rm Gal}(\tilde K,K)$. 
The matrix $(\mu^{[e]}_{x_i})_{[e],x_i}$ 
is of maximal rank $d\!=\!\deg(\Delta_{\bar{a}}^{op})$. 

\medskip
\noindent
{\it Remark.}\
 The index $x_i$ in \eqref{eq:11.3.10} may run over all roots $x$ of 
the equation $x^h\!-\!r^h\!=\!0$. However, if $x$ is not a root of $\Delta_{\bar{a}}^{op}$
(i.e.\ $\Delta_{\bar{a}}^{op}(x^{-1})\!\not=\!0$), then $\mu^{[e]}_{x}\!=\!0$.

\medskip
We return to the series $P(t)$ \eqref{eq:11.2.1} with positive
radius $r\!>\!0$ of convergence.
If $P(t)$ is finite rationally accumulating  
of period $h$ and $a_1^{[e]}:=\iota(a^{[e]})$ for 
$[e]\!\in\! \Z/h\Z$ (recall \eqref{eq:11.3.1}), 
then $\Delta_{\bar{a}}^{op}(s)$ depends only on $P$ 
but not on the choice of the period $h$. Therefore, we shall denote it  
by $\Delta_P^{op}(s)$.
The previous Assertion ii) says that we have the $\R$-isomorphism:
\begin{equation}
\label{eq:11.3.11}
\overline{\R\Omega}(P)\ \simeq\ \R[s]/(\Delta_P^{op}(s)), 
\quad a^{[e]}\ \mapsto\ \Delta_P^{op}\cdot a^{[e]}\bmod\Delta_P^{op}.\!
\end{equation} 
Since the action of $\tau$ is invertible, we
define an endomorphism $\sigma$ on $\overline{\R\Omega}(P)$ by 
\begin{equation}
\begin{array}{lll}
\label{eq:11.3.12}
\sigma(a^{[e]})\ :=\ \tau^{-1}(a^{[e]})\ =\ \frac{1}{a_1^{[e+1]}}a^{[e+1]}.
\end{array}
\end{equation}
{\it The action of $\sigma$ on the LHS and the 
multiplication of $s$ on the RHS 
are equivariant with respect to the isomorphism \eqref{eq:11.3.11}.} 
Hence, {\it the linear dependence relations 
among the generators $a^{[e]}$ ($[e]\!\in\!\Z/h\Z$) are 
obtained by the relations}
 $\Delta_{P}^{op}(\sigma)a^{[e]}\!=\!0$  for $[e]\!\in\!\Z/h\Z$.
However, one should note that the $\sigma$-action 
on  $\overline{\R\Omega}(P)$ is not the same as the 
multiplication of $s$ 
as the subspace  of $\R\db[s\db]$.

\subsection{Duality between $\Delta_P^{op}(s)$ and $\Delta_P^{top}(t)$}
\label{subsec:11.4}

Assuming that $P(t)$ extends to a meromorphic function in a 
neighborhood of the closure of its convergent disc,  
we show a duality between the 
poles of opposite sequences of $P(t)$ 
and the poles of $P(t)$ on its convergent circle.


\medskip
\!{\bf Definition.}
For a positive real number $r$, 
let us denote by $\C\{t\}_r$ the space consisting of complex powers series 
$P(t)$ such that i) $P(t)$ converges (at least) on the open disc 
centered at 0 of radius $r$, and 
ii) $P(t)$ analytically continues to a meromorphic function 
on a disc centered at 0 of radius $>r$.
Let $\Delta_P(t)$ 
be the monic polynomial in $t$ of minimal degree such that $\Delta_P(t) P(t)$ 
is holomorphic in a neighborhood of the circle 
$|t|\!\!=\!r$. Put $\Delta_P(t)=\prod_{i=1}^N(t-x_i)^{d_i}$ where 
 $x_i$ ($i\!=\!1\!,\!\cdots\!,\!N$, $N\!\in\!\Z_{\ge 0}$) 
are mutually distinct complex numbers 
with $|x_i|\!=\!r$ and $d_i\!\in\!\Z_{>0}$ ($i\!=\!1,\!\cdots\!,N$). 
Define the equation for the set of poles of highest order:
\begin{equation}
\begin{array}{l}
\label{eq:11.4.1}
\Delta_P^{top}(t)\!:=\!\prod_{i,d_i=d_{m}}\!\!(t-x_i) \quad\text{ where}\quad  
d_{m}\!:=\!\max\{d_i\}_{ i\!=\!1}^N. \!\!\!\!\!\!\!\!
\end{array}   
\end{equation}

\begin{definition}
Define an action $T_U$ on $\C\db[t\db]$ for 
a rational set $U$ of $\Z_{\ge0}$ 
by  
\begin{equation}
\label{eq:11.4.2}
\begin{array}{l}
P=\sum_{n\in \Z_{\ge0}}\gamma_nt^n
\quad \mapsto\quad   T_UP:=\sum_{n\in U}\gamma_nt^n. 
\end{array}
\end{equation}
One may regard $T_UP$ as a product of $P$ with the  
function $U(t)$ in the sense of Hadamard \cite{H}. 
The radius of convergence of $T_UP$ is not less than that of $P$. 
\end{definition}

\medskip
\noindent
{\it Fact 1.} 
{\it The action of $T_U$ 
preserves the space $\C\{t\}_r$ for any $r\in\R_{>0}$.}

\medskip
{\it Proof.}
Let us  expand the meromorphic function $P(t)$ into partial fractions
%
\[
\begin{array}{ll}
P(t)= \sum_{i=1}^N\sum_{j=1}^{d_i}\frac{c_{i,j}}{(t-x_i)^j}\ + \ Q(t),
\end{array}
\leqno{*)}
\]
%
where   the coefficients $c_{i,j}$  of the principal part 
$\sum_{i=1}^N\sum_{j=0}^{d_i}\frac{c_{i,j}}{(t-x_i)^j}$ of $P(t)$
are constants with $c_{i,d_i}\!\not=\!0$ for all $i$, 
and $Q(t)$ is a holomorphic function on 
a disc of radius $>\!r$. Then, 
$T_UP=\sum_{i,j}T_U\frac{c_{i,j}}{(t-x_i)^{j}} + T_U Q$ where  
$T_U Q$ is a holomorphic function on a disc of radius $>r$. 
It is sufficient to show that, for any standard rational set 
$U^{[e]}\!:=\!\{n\!\in\!\Z_{\ge0}\mid n\!\equiv\! [e] \bmod h\}$ of 
period $h\!\in\!\Z_{>0}$ and
$[e]\!\in\!\Z/h\Z$, one has 
$T_{U^{[e]}}\frac{1}{(t-x_i)^{j}}\!=\!\frac{B_{i,j}(t)}{(t^h-x_i^h)^{j}}$ 
where $B_{i,j}(t)$ is a polynomial in $t$. 
We calculate this explicitly as follows. 
For the purpose, we claim a 
``semi-commutativity'' 
$T_{U^{[e]}}\!\cdot\! \frac{d}{dt}\!=\!\frac{d}{dt}\!\cdot\! T_{U^{[e+1]}}$ 
(proof is trivial and is omitted). 
Then,

\smallskip
$T_{U^{[e]}}\frac{1}{(t-x_i)^j}=T_{U^{[e]}}\frac{(-1)^{j-1}}{(j-1)!}(\frac{d}{dt})^{j-1}\frac{1}{t-x_i}=\frac{(-1)^{j-1}}{(j-1)!}(\frac{d}{dt})^{j-1}\!T_{U^{[e+j-1]}}\!\frac{1}{t-x_i}
\!$

\quad \qquad\qquad \ $=\!\frac{(-1)^{j-1}}{(j-1)!}(\frac{d}{dt})^{j-1}\!\frac{t^f}{t^h-x_i^h}$ \quad where $f\!:\!=\!e\!+\!j\!-\!1\!-\!h[(e\!+\!j\!-\!1)/h]$.

\smallskip
\noindent
This gives the required result.
\quad $\Box$.

\medskip
The following is the goal of the present subsection.
\begin{thm}\!\!\!\!\!\!\!\!{\bf 5. (Duality)} \quad 
Suppose $P(t)$ \eqref{eq:11.2.1} 
belongs in $\C\{t\}_r$ for $r=$ the radius of convergence of $P$,  
and is finite accumulating. Then, we have 
\begin{eqnarray}
\label{eq:11.4.3}
t^{\deg(\Delta_P^{op})}\Delta_P^{op}(t^{-1})
\ & =\ &\Delta_P^{top}(t) ,\\
\label{eq:11.4.4}
\rank(\overline{\R\Omega}(P))\ & =\ &\deg(\Delta_P^{op})\ =\ \deg(\Delta_P^{top}).
\end{eqnarray}
\end{thm}
\begin{proof}
We first show some special case followed by the general case.

\medskip
\noindent
{\it Fact} 2. {\it If $P(t)$, above, is simple accumulating
(i.e.\ $\#\Omega(P)\!=\!1$),
then $\Delta_P^{top}\!=\!t\!-\!r$.}

\smallskip
{\it Proof.}
That  $P(t)$ is simply accumulating means $\underset{n\to\infty}{\lim}\frac{\gamma_{n-1}}{\gamma_n}=r$ and hence, for any small $\varepsilon\!>\!0$, there exists $c\!>\!0$ such 
that $\gamma_n\!\ge\! c(r\!+\!\varepsilon)^{-n}$ for $n\!\in\!\Z_{\ge0}$. 
Let $\delta_n$ be the $n$th Taylor coefficients of $Q$ in 
the splitting $*)$. By assumption on $Q$,
there exists $r'\!>\!r$ and a constant $c'\!>\!0$ such that
$\delta_n\!\le\!c'r'^{-n}$ for $n\!\in\!\Z_{\ge0}$. 
Therefore, choosing $\varepsilon$ such that $r+\varepsilon\!<\!r'$, we have $\delta_n/\gamma_n\to 0$.
Since the $n$th Taylor coefficient of the principal part of the splitting $*)$ 
of $P$ is given by $\gamma_n\!-\!\delta_n$, the principal part, say $P'$, 
is also simply accumulating.
That is, $X_n(P')\!=\!\sum_{k=0}^n \!\frac{\sum_{i=1}^N\sum_{j=1}^{d_i}c_{i,j}x_i^{k-n-1}(n-k;j)/(j-1)!}
{\sum_{i=1}^N\sum_{j=1}^{d_i}c_{i,j}x_i^{-n-1}(n;j)/(j-1)!} s^k$ converges to 
$\frac{1}{1-rs}\!=\!\sum_{k=0}^\infty r^ks^k$. 
Under this setting, we want to show that if $c_{i,d_{m}}\!\not=\!0$ then $x_i\!=\!r$. 
For convenience in the proof, we may assume $r\!=\!1$ 
and hence $|x_i|\!=\!1$ for all $i$. 

Consider the sequence $v_n\!:=\!\sum_{i\!=\!1}^Nc_{i,d_{m}}x_i^{-n-1}$ 
in $n\!\in\!\Z_{\ge0}$. Since $|v_n|\!\le\! \sum_i|c_{i,d_m}|$ is bounded, the sequence accumulates to a compact set in $\C$.  
If the sequence has a unique accumulating value, say $v_0$ 
then the result is already true. ({\it Proof.} Consider the mean sequence:
$\{(\sum_{n\!=\!0}^{M\!-\!1}v_n)/M\}_{M\!\in\!\Z_{\!>\!0}}$. On one side, it converges to 
$v_0$ by the assumption. On the other side, 
$\sum_{i\!=\!1}^N c_{i,d_{m}}\frac{\sum_{n\!=\!0}^{M-1}x_i^{-n-1}}{M}$ 
converges to $c_{1,d_m}$, where we assume $x_1\!=\!1$.
That is, the sequence $v_n'\!:=\!\sum_{i\!=\!2}^Nc_{i,d_m}x_i^{-n-1}$ converges to 0. 
For a fixed $n_0\!\in\! \Z_{>0}$, consider the relations:
$v'_{n_0\!+\!k}\!=\!\sum_{i\!=\!2}^N (c_{i,d_m}x_i^{-n_0}) x_i^{-k+1}$ for $k\!=\!1,\cdots,N\!-\!1$.
Regarding $c_{i,d_m}x_i^{-n_0}$ ($i\!=\!2,\!\cdots\!,N$) as the unknown, we can solve
the linear equation by a use of the van del Monde determinant for the 
matrix $(x_i^{-k+1})_{i\!=\!2,\!\cdots\!,N, k\!=\!1,\!\cdots\!,N\!-\!1}$.  
So, we obtain a linear approximation: $|c_{i,d_m}x_i^{-n_0}|\le c\cdot \max\{|v'_{n_0\!+\!k}|\}_{k\!=\!1}^{N\!-\!1}$ ($i\!=\!2,\!\cdots\!,N$) 
for a constant $c>0$ which is independent of $n_0$. 
The RHS tends to zero as $n_0\!\to\! \infty$, whereas the LHS is unchanged.
This implies $|c_{i,d_m}|\!=\!0$ ($i\!=\!2,\!\cdots\!,N$).

Next, consider the case that the sequence $v_n$ has more than two 
accumulating values. Suppose the subsequence $\{v_{n_m}\}_{m\in \Z_{>0}}$ 
converges to a non-zero value, say $c$. 
Recall the assumption that the sequence $\gamma_{n-1}/\gamma_n$ converges 
to 1. So, the subsequence 
$\frac{\gamma_{n_m-1}}{\gamma_{n_m}}=\frac{v_{n_m-1}+ \text{lower terms}}
{v_{n_m}+ \text{lower terms}}$ 
should also converges to 1 as $m\to \infty$. 
In the denominator, the first term tends to $c\!\not=\!0$ 
and the second term tends to 0. Similarly, in the numerator, the second term 
tends to 0. This implies that the first term in the numerator converges 
to $c$. Repeating the same argument, we see that for any $k\in\Z_{\ge0}$,
 the subsequence $\{v_{n_m-k}\}_{m\in\Z_{>\!\!>0}}$ converges to the same $c$.
Then, for each fixed $M\in \Z_{>0}$, 
the average sequence $\{(\sum_{k=0}^{M-1}v_{n_m-k})/M\}_{m\in \Z_{>\!\!>0}}$ 
converges to $c$, whereas, for sufficiently large $M$, the values are close 
to $c_{1,d_m}$. This implies $c=c_{1,d_m}$. 
In other words, the sequences 
$\{v'_{n_m-k}\}_{m\in\Z_{>\!\!>0}}$ for any $k\ge0$ converge to 0.
Then, an argument as in the previous case implies 
$|c_{i,d_m}|\!=\!0$ ($i\!=\!2,\!\cdots\!,N$).

This is the end of the proof of Fact 2. \quad  $\Box$

\medskip
We return to the general case, where 
$P$ is finite rational accumulating of period $h$.
For the standard partition 
$\{U^{[e]}\mid  [e]\!\in\! \Z/h\Z\}$, put 
$T^{[e]}:=T_{U^{[e]}}$. 
They decompose the unity: $\sum_{[e]\in \Z/h\Z}T^{[e]}\!=\!1$. 
By the assumption, for each {\small $0\le\! f\!<\!h$}, 
the series 
{\small 
$T^{[f]}P=t^f\!\sum_{m=0}^\infty\! \gamma_{f+mh}\tau^m$, 
}
considered as a series in $\tau\!=\!t^h$, is simple accumulating. 
Then Fact 2 implies that the highest order poles of 
$T^{[f]}P$  are only at solutions $x$ of the  
equation $t^h\!-\!r^h\!=\!0$. In view of the fact that the highest 
order of poles in $t$ on the circle $|t|\!=\!r$
of  $T^{[f]}P$ cannot exceed that of $P$ (recall the explicit expression 
in {\it Fact} 1.) and the fact $P\!=\!\sum_{[e]\in \Z/h\Z}T^{[e]}P$,
the highest order poles of $P$ are also only at solutions $x$ of the  
equation $t^h\!-\!r^h\!=\!0$. 
That is, $\Delta_P^{top}(t)$ is a factor of 
$t^h-r^h$.
For $0\!\le\! e,f\!<\!h$ and a root $x$ of the equation $t^h-r^h$, 
we evaluate  
(\eqref{eq:10.6.4} for  
$\{n_m\!=\!e\!+\!mh\}_{m=0}^\infty$ and $\{n_m\!=\!f\!+\!mh\}_{m=0}^\infty$)
\[
\begin{array}{l}
\frac{T^{[f]}P}{T^{[e]}P}(t)\bigr|_{t=x}\ =\ x^{f-e}
\frac{\sum_{m=0}^\infty\! \gamma_{f+mh}\tau^m}
{\sum_{m=0}^\infty\! \gamma_{e+mh}\tau^m}\bigr|_{\tau=x^h=r^h}
\ =\ x^{f-e}\ \lim\limits_{m\to\infty}
\frac{\gamma_{f+mh}}{\gamma_{e+mh}}.
\end{array}
\]
Then, a similar argument to that for \eqref{eq:11.3.4} shows
the formula 
{\small
\begin{equation}
\label{eq:11.4.5}
\frac{T^{[f]}P}{T^{[e]}P}(t)\Bigr|_{t=x}
\ =\ 
\begin{cases} 
x^{f-e}/a_1^{[f]}a_1^{[f-1]}\cdots a_1^{[e+1]}  &\text{if $e< f$}\\
\qquad 1&\text{if $e= f$}\\
x^{f-e} a_1^{[e]} a_1^{[e-1]} \cdots a_1^{[f+1]}  &\text{if $e> f$}. 
\end{cases}
\end{equation}
}
This implies that the order of poles of $T^{[e]}P(t)$ at a solution $x$
of the equation $t^h-r^h$ is independent of $[e]\in \Z/h\Z$. 
On the other hand, \eqref{eq:11.4.5} implies
{\small
\begin{equation}
\label{eq:11.4.6} 
{\small
\frac{T^{[e]}P}{P}(t)\Bigr|_{t=x}
=\ \frac{1}{A^{[e]}(x^{-1})}.
}
\end{equation}
}
(recall the $A^{[e]}(s)$ \eqref{eq:11.3.2}).
Let $x$ be a solution of $t^h\!-\!r^h\!=\!0$ but 
$\Delta_P^{op}(x^{-1})\!\not=\!0$. Then 
$\delta_a(x^{-1})\!=\!0$ (see \eqref{eq:11.3.6}) and $A^{[e]}(x^{-1})\!=\!0$
for $[e]\!\in\! \Z/h\Z$ (see Assertion i)). 
That is, $\frac{T^{[e]}P}{P}(t)$ has a pole at $t\!=\!x$. 
This implies that the pole of $P(t)$ at $t\!=\!x$ is of order $\!<\!d_m$
(otherwise, the pole at $t\!=\!x$ of $T^{[e]}P$ is of order $d_m$ at most
and can be canceled by dividing by $P$). 
That is, $\Delta_P^{top}(t)\mid t^d\Delta_P^{op}(t^{-1})$.

\medskip
\noindent
{\it Fact} 3. {\it Let $P(t)$ \eqref{eq:11.2.1} belong to $\C\{t\}_r$ and 
be finitely accumulating. Then

i) There exists a positive constant $c$ such that $\gamma_n\ge cr^{-n}n^{d_m-1}$ 
for $n>\!>0$.

ii) 
\centerline{
$t^d \Delta_P^{op}(t^{-1}) \ | \ \Delta_P^{top}(t)$\ .\qquad\qquad
}
}
{\it Proof.} 
i)  Consider the Taylor expansion of the function $*$). 
Using notation $v_n$ in {\it Fact} 2., 
we have $\gamma_n=-v_n\frac{ r^{-n-1}(n;d_m)}{(d_m-1)!} + \text{terms coming from 
poles of order $<d_m$} + \text{terms coming from $Q(t)$}$, 
where $v_n=\sum_{i}c_{i,d_m}(x_i/r)^{-n-1}$ depends only on $n \bmod h$ since 
$x_i$ is the root of the equation $t^h-r^h=0$. Not all of these are zero 
(otherwise $c_{i,d_m}=0$ for all i). 
Let us show that none of the  $v_n$ are zero. Suppose 
the contrary and $v_e\!=\!0\!\not=\!v_f$. Then, one observes easily
$\lim_{m\to\infty}\frac{\gamma_{e+mh}}{\gamma_{f+mh}}=0$. This contradicts 
the assumption $\Omega_1(P)\subset [u,v]$ (positivity of initials).

ii) By definition, the fractional expansion of $\Delta_P^{top}(t)P(t)$ has 
poles of order at most $d_m\!-\!1$. This means that its $(n\!-\!k)$th Taylor coefficient:

\vspace{-0.7cm}
\[
\gamma_{n-k}\cdot\al_l+\gamma_{n-k-1}\cdot\al_{l-1}+\cdots+\gamma_{n-k-l}\cdot1\ \ \sim\ \ o((n-k)^{d_m-1}r^{-(n-k)})
\leqno{**)}
\]
as $n-k \to \infty$ ($k,n\in\Z_{\ge0}$) (here, $\Delta_P^{top}(t)\!=\!t^l\!+\!\al_1t^{l-1}\!+\!\cdots\!+\!\al_l$). 
Let $\sum_ka_ks^k\!\in\!\Omega(P)$ be 
the limit of a subsequence $\{X_{n_m}(P)\}_{m\in\Z_\ge0}$ \eqref{eq:11.2.2}.
Divide $**)$ by $\gamma_{n}$.
Then, using the part i), one has  
\vspace{-0.1cm}
\[
a_{k}\al_l+a_{k+1}\al_{l-1}+\cdots+a_{k+l} \ =\ 0
\]
for any $k\!\ge\!0$. 
Thus $s^l\Delta_P^{top}(1/s) a(s)$ is a polynomial in $s$ of degree 
$<l$.  
Thus the denominator $\Delta_P^{op}(s)$ of $a(s)$ divides $s^l\Delta_P^{top}(s^{-1})$.
So, ii) is shown. \qquad $\Box$

We showed \eqref{eq:11.4.3}. \eqref{eq:11.4.4} follows from \eqref{eq:11.3.11} and \eqref{eq:11.4.3}.
\end{proof}

\noindent
{\it Example.} Recall Mach\`i's example \ref{subsec:11.2} for the modular group $\Gamma$. We have
\vspace{-0.3cm}
\[
{\footnotesize
\begin{array}{l}
T_eP(t)=\sum_{k=0}^\infty \#\Gamma_{2k}t^{2k}=\! \frac{1+5t^2}{(1-2t^2)(1-t^2)}, 
\quad T_oP(t)= \sum_{k=0}^\infty \#\Gamma_{2k+1}t^{2k+1} =\! \frac{2t(2+t^2)}{(1-2t^2)(1-t^2)},
\vspace{-0.1cm}
\end{array}
}
\] 
Then the transformation matrix is given by
\[
\!\!\!\left[\!
\begin{array}{cc}
\frac{T_eP(t)}{P_{\Gamma,G}(t)}\!\!=\!\!
\frac{1+5t^2}{(1+t)^2(1+2t)}\!\mid_{t=\frac{1}{\sqrt{2}}}&
\frac{T_oP(t)}{P_{\Gamma,G}(t)}\!\!=\!\!\frac{2t(2+t^2)}{(1+t)^2(1+2t)}\!\mid_{t=\frac{1}{\sqrt{2}}}\\
\frac{T_eP(t)}{P_{\Gamma,G}(t)}\!\!=\!\!
\frac{1+5t^2}{(1+t)^2(1+2t)}\!\mid_{t=\frac{-1}{\sqrt{2}}} & 
\frac{T_oP(t)}{P_{\Gamma,G}(t)}\!\!=\!\!
\frac{2t(2+t^2)}{(1+t)^2(1+2t)}\!\mid_{t=\frac{-1}{\sqrt{2}}}
\end{array}\!
\right]
\!\!=\!\!
\left[
\begin{array}{cc}
\footnotesize{7(5\sqrt{2}\!-\!7)} & \footnotesize{5(10\!-\!7\sqrt{2})} \\
\!\!\!\!\footnotesize{-7(5\sqrt{2}\!+\!7)} & \footnotesize{5(10\!+\!7\sqrt{2})}
\end{array}
\right]
\]
whose determinant is equal to $\frac{5\cdot7}{\sqrt{2}}\not=0$.

\subsection{The residual representation of trace elements}
\label{subsec:11.5}

As the goal of the present paper, under further asummptions i) $\#(\Omega(\Gamma,G))<\infty$ and ii) $P_{\Gamma,G}\in\C\{t\}_{r_{\Gamma,G}}$, 
we show a trace formula, which states that 
{\it the sum of the limit elements in a orbit of the inertia group 
is
expressed by a linear combination of the proportions of residues of the Poincare series $P_{\Gamma,G}(t)$ and $P_{\Gamma,G}\mathcal{M}(t)$ {\rm (11.2.6,7)} at the poles on the circle of their convergent radius, where the coefficients are given by special values of the opposit polynomials $A^{[e]}(s)$.
} 

\medskip
\noindent
\ \ We first show the following basic consequence of the finiteness $\#(\Omega(\Gamma,G))\!<\!\infty$.

\begin{lemma} Let $(\Gamma,G)$ be the pair of a monoid and its finite 
generating system, which satsfies Assumption 1 but not necessarily 2. If the limit space $\Omega(\Gamma,G)$ is finite, then it is 
finite rationally accumulating with respect to the standard  partition 
$\mathcal{U}_{\tilde h}$ of $\Z_{\ge0}$ for some $\tilde h>0$,
and $\tilde \tau_\Omega$ acts transitively 
on $\Omega(\Gamma,G)$ of period $\tilde h$. 
In particular, $\tilde \tau_\Omega$ is invertible.
\vspace{-0.1cm}
\end{lemma}
\begin{proof}
Recall the action $\tilde \tau_\Omega$ on $\Omega(\Gamma,G)$ 
(Lemma in 11.2).
Then, finiteness of $\Omega(\Gamma,G)$ implies that 
there exists an element $\omega\in \Omega(\Gamma,G)$ and an 
integer $\tilde h\in \Z_{>0}$ such that 
$(\tilde \tau_\Omega)^{\tilde h}\omega=\omega$ and  
$(\tilde \tau_\Omega)^{\tilde h'}\omega\not=\omega$ for $0<\tilde h'<\tilde h$. Consider the set 
$U_\omega\!:=\!\{n\in\Z_{\ge0}\mid \frac{\mathcal{M}(\Gamma_{n})}{\#\Gamma_{n}}\in \mathcal{V}_\omega\}$ (here, $\mathcal{V}_\omega$ is an open 
neighborhood of $\omega$ in $\mathcal{L}_{\R,\infty}$ 
such that  
$\overline{\mathcal{V}_\omega}\cap\Omega({\Gamma,G})=\{\omega\}$). 
Then, the periodicity of the action of $\tilde \tau_\Omega$ on $\omega$ implies 
(using an argument similar to that found  in the proof of 11.2 Lemma, replacing 
$a\in \Omega(P)$ by $\omega\in \Omega(\Gamma,G)$ and $h$ by $\tilde h$, respectively) 
that   $U_\omega$ 
is, up to a finite number of elements, equal to a rational 
set $U^{[\tilde e]}$ for some $[\tilde e]\in \Z/\tilde h\Z$, 
and the following equality holds:
\vspace{-0.1cm}
\[
\Omega(\Gamma,G)\ =\ \{\ \omega,\ \tilde\tau_\Omega\omega,\ \cdots\ ,\ (\tilde\tau_\Omega)^{\tilde h-1}\omega\ \}.
\vspace{-0.1cm}
\]
This implies the finite rationality of $\Omega(\Gamma,G)$ and 
the periodicity of $\tilde\tau_\Omega$.
\end{proof}

\newpage
Let $\Omega(\Gamma,G)$ be finite rationally accumulating
of period $\tilde h$, which consists of\!\!  
\begin{equation}
\begin{array}{l}
\label{11.5.1}
\omega^{[\tilde e]}_{\Gamma,G}\ :=\ {\lim\limits_{m\to\infty}}^{\!\!\!\!cl}\ 
\frac{\mathcal{M}(\Gamma_{\tilde e+m \tilde h})}{\#\Gamma_{\tilde e+m\tilde h}}
\end{array}
\vspace{-0.2cm}
\end{equation}
for $[\tilde e]\!\in\! \Z/\tilde h\Z$. 
Then, 
$\Omega(P_{\Gamma,G})$ is also finite rationally 
accumulating of period $h$ such that $h|\tilde h$ (c.f.\ 11.2 Lemma), 
since the sequence 
$\{\pi(\frac{\mathcal{M}(\Gamma_n)}{\sharp \Gamma_n})\!=\!X_n(P_{\Gamma,G})\}_{n\!\in\! U^{[\tilde e]}}$ 
for the rational set 
$U^{[\tilde e]}\!:=\!\{n\!\in\!\Z_{\ge0}\mid n\bmod \!\tilde h\equiv\! [e]\}$ 
for any $[\tilde e]\!\in\! \Z/\tilde h\Z$ 
is convergent to 
$\pi(\omega^{[\tilde e]}_{\Gamma,G})$.
Let $\tilde h_{\Gamma,G}$ and $h_{\Gamma,G}$ 
be the minimal period of $\Omega(\Gamma,G)$ and $\Omega(P_{\Gamma,G})$,
 respectively.
Then $\pi$ is equivariant under the $\tilde \tau_{\Omega}$-action on 
$\Omega(\Gamma,G)$ and the
$\tau_{\Omega}$-action on $\Omega(P_{\Gamma,G})$ so that 
the subgroup 
$h_{\Gamma,G}\Z/\tilde h_{\Gamma,G}\Z$ of 
$\Z/\tilde h_{\Gamma,G}\Z\simeq\langle\tilde \tau_\Omega\rangle$, 
called the {\it inertia subgroup}, acts simply and transitively on 
the fibers of $\pi$. That is, 
$\Omega(\Gamma,G)/(h_{\Gamma,G}\Z/\tilde h_{\Gamma,G}\Z) \simeq \Omega(P_{\Gamma,G})$.
We call $m_{\Gamma,G}:=\tilde h_{\Gamma,G}/h_{\Gamma,G}$ the {\it inertia} of $(\Gamma,G)$ 
so that the inertia subgroup is isomorphic to 
$\Z/m_{\Gamma,G}\Z$.

\medskip
{\bf Definition.}
The {\it trace element }  for $[e]\in\Z/h_{\Gamma,G}\Z$ is  the sum of 
the elements in the fiber $\pi_\Omega^{-1}(a^{[e]})$ 
(=an orbit of the inerta group $h_{\Gamma,G}\Z/\tilde h_{\Gamma,G}\Z)$:
\begin{equation}
\begin{array}{ll}
\label{eq:11.5.2}
\mathrm{Trace}^{[e]}\ \Omega(\Gamma,G)\ :=
\sum_{\substack{
[\tilde e]\in \Z/\tilde h_{\Gamma,G}\Z,\
[\tilde e]\subset [e]}} \omega_{\Gamma,G}^{[\tilde e]}
\ =\sum_{i=1}^{m_{\Gamma,G}}\omega_{\Gamma,G}^{[\tilde e+ih_{\Gamma,G}]}
\end{array}
\end{equation}
which belongs to the space $\overline{\R\Omega}(\Gamma,G)$.

The  periodicity of $\tilde \tau_{\Omega}$ implies the invertibility of 
$\tilde \tau$ \eqref{eq:11.2.14}. As its consequence, 
let us introduce a $\tilde \sigma$-action on the module 
$\overline{\R\Omega}(\Gamma,G)$. 

\smallskip
{\bf Definition.} \  For any $[\tilde e]\in \Z/\tilde h_{\Gamma,G}\Z$, put $[e]\equiv [\tilde e] 
\bmod h_{\Gamma,G}$ and define 
\begin{equation}
\begin{array}{ll}
\label{eq:11.5.3}
\tilde \sigma (\omega_{\Gamma,G}^{[\tilde e]})\ :=\ \tilde\tau^{-1}(\omega_{\Gamma,G}^{[\tilde e]})\ =\ \frac{1}{a_1^{[e+1]}}\omega_{\Gamma,G}^{[\tilde e+1]}.
\end{array}
\vspace{-0.3cm}
\end{equation}
%
The endomorphism $\tilde \sigma$ is semi-simple since one has
$\tilde \sigma^{\tilde h_{\Gamma,G}}=r_{\Gamma,G}^{\tilde h_{\Gamma,G}} 
\ \mathrm{id}_{\overline{\R\Omega}(\Gamma,G)}$ (c.f.\ \eqref{eq:11.3.3}).
The $\R$-linear map $\pi$ \eqref{eq:11.2.15} is equivariant with respect to 
the endomorphisms $\tilde \sigma$ and $\sigma$ \eqref{eq:11.3.12}. 
By the definition, 
we see that the $\tilde \sigma$-action brings, up to a constant factor, 
a trace element to the other trace element 
\begin{equation}
\begin{array}{ll}
\label{eq:11.5.4}
\tilde \sigma \left(\mathrm{Trace}^{[e]}\ \Omega(\Gamma,G)\right)\ :=\ \frac{1}{a_1^{[e+1]}}\mathrm{Trace}^{[e+1]}\ \Omega(\Gamma,G) 
\end{array}
\vspace{-0.1cm}
\end{equation}
for all $[e]\in \Z/h_{\Gamma,G}\Z$. 
In view of \eqref{eq:11.3.3}, this, in particular, implies
\begin{equation}
\label{eq:11.5.5}
(1-(r_{\Gamma,G}\ \tilde\sigma)^h) \left(\mathrm{Trace}^{[e]}\ \Omega(\Gamma,G)\right)\ =\ 0.
\vspace{-0.1cm}
\end{equation}
After the results of 11.3 and 11.4, the next theorem  
is now straightforward.

\begin{thm}\!\!\!\!\!\!\!{\bf 6.}
Let $(\Gamma,G)$ be a pair of a monoid and its finite generating 
system with $1\not\in G$, satisfying {\bf Assumptions 1} and {\bf 2}. 
Suppose i) $\Omega(\Gamma,G)$ is  finite, and ii) $P_{\Gamma,G}\!\in\!\C\{t\}_{r_{\Gamma,G}}$. 
Let $\tilde h_{\Gamma,G}$ and $h_{\Gamma,G}$ 
be the minimal period of $\Omega(\Gamma,G)$ and $\Omega(P_{\Gamma,G})$,
 respectively, and put $\tilde m_{\Gamma,G}\!:=\!\tilde h_{\Gamma,G}/h_{\Gamma,G}$.
Then, for any  $[e]\!\in\! \Z/h_{\Gamma,G}\Z$, 
the following equality holds.
\begin{equation}
\label{eq:11.5.6}
\begin{array}{lll}
& &h_{\Gamma,G}\mathrm{Trace}^{[e]} \Omega_{\Gamma,G} 
-\left(\sum_{x^{-1}\in V(\delta_{P_{\Gamma,G}})}\!\!
\frac{\delta_{P_{\Gamma,G}}(\tilde\sigma)}{1-x\tilde \sigma}\right) \Delta_{P_{\Gamma,G}}^{op}\!(\tilde \sigma)
 \mathrm{Trace}^{[e]} \Omega_{\Gamma,G}\!\! \\
&\!\!\! = &\  
m_{\Gamma,G}\sum_{x\in V(\Delta_{P_{\Gamma,G}}^{top})}A^{[e]}(x^{-1})
\frac{P_{\Gamma,G}\mathcal{M}(t)}{P_{\Gamma,G}(t)}\biggr|_{t=x} .
\end{array}\!\!\!\!\!
\end{equation}
where we put $\delta_{P_{\Gamma,G}}(\tilde \sigma):=(1-r^{h_{\Gamma,G}}\tilde \sigma^{h_{\Gamma,G}})/\Delta_{P_{\Gamma,G}}^{op}(\tilde \sigma)$ (c.f.\ \eqref{eq:11.3.6}) and we denote by $V(P)$ the set of zero loci of the polynomial $P$.
\end{thm}
\begin{proof} 
Let us call 
$\frac{P_{\Gamma,G}\mathcal{M}(t)}{P_{\Gamma,G}(t)}\biggr|_{t=x}$
in RHS of the formula \eqref{eq:11.5.6} a {\it residue element}, since it is a 
proportion of the residues of $P_{\Gamma,G}\mathcal{M}(t)$ and 
$P_{\Gamma,G}(t)$ at the point $t=x$.
Let us, first, express the residue element by a sum of 
trace elements. 
For the purpose, consider the decomposition of unity:
%
\[
\begin{array}{l}
\frac{P_{\Gamma,G}\mathcal{M}(t)}{P_{\Gamma,G}(t)}\ =\ 
\sum_{[\tilde f]\in \Z/\tilde h_{\Gamma,G}\Z}\frac{T^{[\tilde f]}P_{\Gamma,G}(t)}{P_{\Gamma,G}(t)}\cdot
\frac{T^{[\tilde f]}P_{\Gamma,G}\mathcal{M}(t)}{T^{[\tilde f]}P_{\Gamma,G}(t)}.
\end{array}
\vspace{-0.1cm}
\leqno{*)}\]
where $T^{[\tilde f]}\!=\!T_{U^{[\tilde f]}}$ \eqref{eq:11.4.2} 
is the action of the 
rational set $U^{[\tilde f]}$ of the standard subdivision for 
$\Omega(\Gamma,G)$ so that $\sum_{\tilde f\in \Z/\tilde h_{\Gamma,G}\Z}T^{[\tilde f]}\!=\!1$. 
Let $x$ be a root of $\Delta_{P_{\Gamma,G}}^{top}(t)\!=\!0$, and consider 
the evaluation of both sides of $*$) at $t\!=\!x$. The LHS gives, 
by definition, the 
residue element at $x$. By a slight generalization of the formula 
\eqref{eq:11.3.6}, the first factor in the RHS is given by 
$1/A^{[\tilde f]}(x^{-1}\!)\!=\!1/(m_{\Gamma,G\!}\!\cdot\!\! A^{[f]}(x^{-1}\!))$ 
(note that $A^{[f]}(x^{-1}\!)\!\not=\!0$ since $\delta_{P_{\Gamma,G}}\!(x^{-1}\!)\!\not=\!0$),
where $[f]\!:=\![\tilde f] \bmod h_{\Gamma,G}$. The second factor 
in RHS is
\vspace{-0.3cm}
\[
\frac{\sum_{m=0}^\infty \mathcal{M}(\Gamma_{\tilde f+m \tilde h_{\Gamma,G}})
t^{\tilde f+m \tilde h_{\Gamma,G}}}
{\sum_{m=0}^\infty \#\Gamma_{\tilde f+m \tilde h_{\Gamma,G}}t^{\tilde f+m \tilde h_{\Gamma,G}}} \biggr|_{t=x}
=\ \frac{\sum_{m=0}^\infty \mathcal{M}(\Gamma_{\tilde f+m \tilde h_{\Gamma,G}})
\tilde t^m}
{\sum_{m=0}^\infty \#\Gamma_{\tilde f+m \tilde h_{\Gamma,G}}\tilde t^m} 
\biggr|_{\tilde t=r^{\tilde h_{\Gamma,G}}}
\vspace{-0.1cm}
\]
where, in the RHS, $\tilde t\!:=\!t^{\tilde h_{\Gamma,G}}$ is the new variable 
and $r^{\tilde h_{\Gamma,G}}\!=\!x^{\tilde h_{\Gamma,G}}$ 
is the common singular point of
the two power series (the numerator and the denominator) in $\tilde t$ at the 
crossing of the positive real axis and the circle of the convergent
radius
(c.f.\ 10.6 Lemma i)). Then, since the coefficients of the series 
are non-negative, this proportion of the residue value is equal to
the limit of the proportion of the coefficients of the series 
(c.f.\ \eqref{eq:10.6.4}) 
${\lim\limits_{m\to\infty}}^{\!\!\!\!cl}\frac{\mathcal{M}(\Gamma_{\tilde f+m \tilde h_{\Gamma,G}})}
{\#\Gamma_{\tilde f+m \tilde h_{\Gamma,G}}}$ which is nothing but the
limit element $\omega_{\Gamma,G}^{[\tilde f]}$ (11.5.1). 
Put $\tilde f\!=\!f\!+\!ih_{\Gamma,G}$ for 
$0\!\le\! f\!<\!h_{\Gamma,G}$ 
and $0\!\le\! i\!<\!m_{\Gamma,G}$. Then the RHS 
turns into 
\vspace{-0.2cm}
\[
{\large
\begin{array}{ll}
\frac{1}{m_{\Gamma,G}}\sum_{[f]\in \Z/h_{\Gamma,G}\Z}\frac{1}{A^{[f]}(x^{-1})}
\sum_{i=0}^{m_{\Gamma,G}-1}\omega_{\Gamma,G}^{[f+ih_{\Gamma,G}]}
\end{array}
}
\vspace{-0.2cm}
\]
where the second sum in the RHS gives the trace 
$\mathrm{Trace}^{[f]} \Omega(\Gamma,G)$. That is,
\begin{equation}
\begin{array}{l}
\label{eq:11.5.7}
\frac{P_{\Gamma,G}\mathcal{M}(t)}{P_{\Gamma,G}(t)}\Bigr|_{t=x}
= \
\frac{1}{m_{\Gamma,G}}\sum_{[f]\in \Z/h_{\Gamma,G}\Z}\frac{1}{A^{[f]}(x^{-1})}
\mathrm{Trace}^{[f]} \Omega(\Gamma,G).
\end{array}
\end{equation}
For a fixed $[e]\!\in\! \Z/h_{\Gamma,G}\Z$, we multiply $A^{[e]}(x^{-1})$ 
to both sides of \eqref{eq:11.5.7}, and 
sum over the index $x$ running over the set $V(\Delta_{P_{\Gamma,G}}^{top})$ 
of all roots of 
$\Delta_{P_{\Gamma,G}}^{top}(t)\!=\!0$, whose LHS is 
equal to the RHS of \eqref{eq:11.5.6}.
Using \eqref{eq:11.4.6}, one observes that
$\frac{A^{[e]}(x^{-1})}{A^{[f]}(x^{-1})}
$
is equal to the LHS of \eqref{eq:11.4.5}.
Replace the summation index ``$[f]\!\in\!\Z/h_{\Gamma,G}\Z$'' 
in \eqref{eq:11.5.7} by 
``$[e\!+\!i]$ for $i\!=\!0,\cdots,h_{\Gamma,G}\!-\!1$'' for fixed $[e]$. 
Using the first line of  RHS of \eqref{eq:11.4.5} and 
$i$th repeated applications of \eqref{eq:11.5.4},
the sum in RHS turns out to
\[
\begin{array}{ll}
& \frac{1}{m_{\Gamma,G}}\sum_{x\in V(\Delta_{P_{\Gamma,G}}^{top})}
\sum_{i=0}^{h_{\Gamma,G}-1}\frac{A^{[e]}(x^{-1})}{A^{[e+i]}(x^{-1})}
\mathrm{Trace}^{[e+i]}\ \Omega_{\Gamma,G}\\
=  & \frac{1}{m_{\Gamma,G}}\sum_{x\in V(\Delta_{P_{\Gamma,G}}^{top})}
\sum_{i=0}^{h_{\Gamma,G}-1}
\frac{x^{i}}{a_1^{[e+i]}a_1^{[e+i-1]}\cdots a_1^{[e+1]}} \prod_{j=1}^{i}
(a_1^{[e+j]}\tilde \sigma)\ 
\mathrm{Trace}^{[e]}\ \Omega_{\Gamma,G} \\
=  & \frac{1}{m_{\Gamma,G}}\sum_{x\in V(\Delta_{P_{\Gamma,G}}^{top})}
\sum_{i=0}^{h_{\Gamma,G}-1}
x^{i}\tilde \sigma^i
\ \mathrm{Trace}^{[e]}\ \Omega_{\Gamma,G}.
\end{array}
\]
Here, we note that the sum 
$\sum_{i=0}^{h_{\Gamma,G}-1}x^{i}\tilde \sigma^i$ is expressed as
$\frac{1-(r_{\Gamma,G}\tilde\sigma)^{h_{\Gamma,G}}}{1-x\tilde \sigma}$
and that $x\in  V(\Delta_{P_{\Gamma,G}}^{top})$ is equivalent to
$x^{-1}\in  V(\Delta_{P_{\Gamma,G}}^{op})$ due to the duality 
\eqref{eq:11.4.3}. 
We note further that an identity:

\centerline{
$\sum_{x^{-1}\in V(1-(r_{\Gamma,G}s)^{h_{\Gamma,G}})}
\frac{1-(r_{\Gamma,G}s)^{h_{\Gamma,G}}}{1-x s}
\ =\ h_{\Gamma,G}$
}

\smallskip
\noindent
holds (in the polynomial ring of $s$). Therefore, recalling \eqref{eq:11.3.6} 

\centerline{
$\delta_{P_{\Gamma,G}}(s)\cdot \Delta_{P_{\Gamma,G}}^{op}(s)=1-(r_{\Gamma,G}s)^{h_{\Gamma,G}}$
}

\smallskip
\noindent
we calculate further the sum as follows.
\[
\begin{array}{ll}
  & \frac{1}{m_{\Gamma,G}}\left(\sum_{x^{-1}\in V(\Delta_{P_{\Gamma,G}}^{op})}
\frac{1-(r_{\Gamma,G}\tilde\sigma)^{h_{\Gamma,G}}}{1-x\tilde \sigma}\right)
\ \mathrm{Trace}^{[e]}\ \Omega_{\Gamma,G} \\
=  & \frac{1}{m_{\Gamma,G}}\left( h_{\Gamma,G}\cdot id_{\overline{\R\Omega}(\Gamma,G)}-
\sum_{x^{-1}\in V(\delta_{P_{\Gamma,G}})}
\frac{\delta_{P_{\Gamma,G}}(\tilde\sigma)}{1-x\tilde \sigma} \Delta_{P_{\Gamma,G}}^{op}(\tilde \sigma)\right)
\ \mathrm{Trace}^{[e]}\ \Omega_{\Gamma,G}.
\end{array}
\]
This gives LHS of \eqref{eq:11.5.6}, and hence Theorem is proven.
\end{proof}



\noindent
{\bf  Example.} Consider the free group $F_f$ with $f$ number of generating system $G$ (\S11.1 Example 2.). Using the formula \eqref{eq:11.1.10}, it is immediate to calculate\vspace{-0.2cm}
{\footnotesize
\begin{equation}
\label{eq:11.5.8}
P_{F_f,G}\mathcal{M}(t)=\frac{1}{(1-t)(1-(2f\!-\!1)t)}\sum_{k=0}^\infty t^k\!
\left((1\!+\!t)\!\!\!\!\!\!\sum_{\substack{S\in \langle F_f,G\rangle_0\\d(S)=2k}}\!\!\varphi(S)+2t\!\!\!\!\!\!\sum_{\substack{S\in \langle F_f,G\rangle_0\\d(S)=2k+1}}\!\!\varphi(S)\right)\!\!\!\!\!
\end{equation}
}\vspace{-0.4cm}
{\footnotesize
\begin{equation}
\label{eq:11.5.9}
\frac{P_{F_f,G}\mathcal{M}(t)}{P_{F_f,G}(t)}=\sum_{k=0}^\infty t^k
\left(\sum_{\substack{S\in \langle F_f,G\rangle_0\\d(S)=2k}}\varphi(S)+\frac{2t}{1+t}\sum_{\substack{S\in \langle F_f,G\rangle_0\\d(S)=2k+1}}\varphi(S)\right)
\end{equation}
}
The denominator polynomial $\Delta_{F_f,G}(t)=(1-t)(1-(2f-1)t)$ has two roots 1 and $1/(2f-1)$. The specialization of the variable $t$ in \eqref{eq:11.5.9} to the smaller root $1/(2f-1)$ gives the limit element \eqref{eq:11.1.9}. The specialization to $t=1$ gives
{\small
 $\sum_{\substack{S\in \langle F_f,G\rangle_0}}\varphi(S)$
}
(see \eqref{eq:12.3} and \S12 Problem 3. iii)).

\medskip
\noindent
{\it Remark.} 
1. The second term of the LHS of \eqref{eq:11.5.6} belongs to the kernel of
$\pi$, since one has $\pi(\Delta_{P_{\Gamma,G}}^{op}(\tilde\sigma)\mathrm{Trace}^{[e]} \Omega_{\Gamma,G})=m_{\Gamma,G}\Delta_{P_{\Gamma,G}}^{op}(\sigma) a^{[e]}=0$.
Therefore, we ask whether 

\centerline{
$\Delta_{P_{\Gamma,G}}^{op}(\tilde \sigma)\
 \mathrm{Trace}^{[e]} \Omega_{\Gamma,G}=0$ \ ? 
}

\vspace{0.1cm}
\noindent
This is equivalent to 
the statement that {\it the $\R[\tilde \sigma]$-module spanned by the trace
elements  
$ \mathrm{Trace}^{[e]} \Omega_{\Gamma,G}$ 
is isomorphic to the $\R[\sigma]$-module $\overline{\R\Omega}(P_{\Gamma,G})$.}

\medskip
2. One can directly calculate the following formula:
\begin{equation}
\label{eq:11.5.10}
\begin{array}{l}
\pi\left(\frac{P_{\Gamma,G}\mathcal{M}(t)}{P_{\Gamma,G}(t)}\right)
\ =\ \frac{1}{1-st}.
\end{array}
\end{equation}
\noindent
Specializing $t$ to a root $x$ of $\Delta_{P}(t)\!=\!0$ in  
the formula gives the Cauchy kernel $\frac{1}{1-xs}$. Therefore, 
the $\pi$ image of 
\eqref{eq:11.5.6}  turns out to be the formula \eqref{eq:11.3.9}.

\medskip
3. If $(\Gamma,G)$ is a group of polynomial growth, then 
$\Delta_{P_{\Gamma,G}}(t)\!=\!(1\!-\!t)^{l+1}$ (where $l\!\!=\!\!\rank(\Gamma)\!\!>\!\!0$) 
is never reduced. 
However, due to \eqref{eq:10.6.4}, one sees directly the conclusion of 
Theorem: $\frac{P_{\Gamma,G}\mathcal{M}(t)}{P_{\Gamma,G}(t)}\Bigr|_{t=1}\!\!
\!\!\!=\!\sum_{S\in\langle \Gamma,G\rangle_0}\frac{\varphi(S)}{\#Aut(S)}$ (c.f.\ \eqref{eq:11.1.8}).

\medskip
4. Due to D.\ Epstein \cite{E3}, we know that there 
is a wide class of groups 
satisfying Assumption ii). See the remarks and problems in the next paragraph.

\section{ Concluding Remarks and Problems.}
\label{sec:12}
We are only at the start of the study of the space $\Omega(\Gamma, G)$ 
for discrete groups and monoids.
Here are some problems and conjectures for further study.

\medskip
\noindent
{\bf 1.} A formula similar to \eqref{eq:11.5.6} should be true without  
assuming the finiteness of $\Omega(\Gamma,G)$, where the formula 
should be rewritten as an integral formula. 

\medskip
{\bf Problem 1.} Find measures $\nu_a$ on $\pi^{-1}(a)$ 
and $\mu_a$ on the set $Sing(P_{\Gamma,G})$ of singularities of the  
series on the circle of radius $r$ so that 
\vspace{-0.2cm}
the following holds:
\begin{equation}
\label{eq:12.1}
\frac{\int_{\pi^{-1}(a)} \omega_{\Gamma,G} d\nu_a}{\int_{\pi^{-1}(a)} d\nu_a}
=
\int_{Sing(P_{\Gamma,G})} \frac{P_{\Gamma,G}\mathcal{M}(t)}{P_{\Gamma,G}(t)}\Bigr|_{t=x} d\mu_{a,x},
\end{equation}
for $a\in \Omega(P_{\Gamma,G})$, where $ \omega_{\Gamma,G}$ is a tautologiocal map from $\Omega(\Gamma,G)$ to $\mathcal{L}_{\R,\infty}$.

\bigskip 
\noindent
{\bf 2.} It is known (\cite{E3}) that, for a wide class of groups, the assumption 
ii) in the Theorem 6 is satisfied in a stronger (global) form: the Poincare 
series $P_{\Gamma,G}(t)$ and the growth series $P_{\Gamma,G}\mathcal{M}(t)$ 
are rational functions, where the denominator polynomial 
$\Delta_{\Gamma,G}(t)$ for the rational function 
$P_{\Gamma,G}(t)$ is also the universal denominator for the rational functions 
$P_{\Gamma,G}\mathcal{M}(t)$. More generally, $P_{\Gamma,G}(t)$ analytically continues  to a meromorphic function on a (branched) covering domain of $\C$ (in this case, $\Delta_{\Gamma,G}(t)$ is defined only up to a unit factor).

We remark that denominator polynomial $\Delta_{P_{\Gamma,G}}(t)$ 
for the Poincare series 
$P_{\Gamma,G}(t)$ as an element of $\C\{t\}_{r_{\Gamma,G}}$ (see 11.4 Definition) is the factor of $\Delta_{\Gamma,G}(t)$ consisting of the roots $x$
with minimal $|x|=r_{\Gamma,G}$ (in case $P_{\Gamma,G}(t)$ is defined in a covering of $\C$, whether $|x|$ of a pole $x$ makes sense or not is unclear).

Inspired by these observations, in order to get a global understanding of the monoid $(\Gamma,G)$, we propose  studying the {\it residues of 
$P_{\Gamma,G}\mathcal{M}(t)$ at any root of $\Delta_{\Gamma,G}(t)$}, 
which are defined and shown to belong to $\mathcal{L}_{\C,\infty}$ as follows.

\medskip
{\bf Definition.}
  Let $x$ be a root of $\Delta_{\Gamma,G}(t)\!=\!0$ of the 
multiplicity $d_x>0$. Then, for $0\!\le\! i\!<\! d_x$, 
we define the {\it residue of depth $i$ of the limit 
function $P_{\Gamma,G}\mathcal{M}(t)$ at $x$}  
by the formula
\vspace{-0.2cm}
{\small
\begin{equation}
\label{eq:12.2}
\left(\frac{d^i}{dx^i}\frac{P_{\Gamma,G}\mathcal{M}(t)}{P_{\Gamma,G}(t)}\right)\Biggr|_{t=x}
\end{equation} 
}

{\bf Example.} The formula \eqref{eq:11.1.7} is paraphraised as the formula for the residue of depth 0 at $t=1$.
{\small
\begin{equation}
\label{eq:12.3}
\frac{P_{\Gamma,G}\mathcal{M}(t)}{P_{\Gamma,G}(t)}\Biggr|_{t=1}=\sum_{S \in \langle\Gamma,G\rangle_0}\frac{1}{ \#(\Aut(S))}\varphi(S).
\end{equation}
}

{\bf Assertion.} {\it The residues belong to the space 
$\mathcal{L}_{\C,\infty}$ at infinity. }

\medskip
\noindent
{\it Proof.}
By the definition \eqref{eq:8.4.1},
$\overline{K}(\frac{P_{\Gamma,G}\mathcal{M}(t)}{P_{\Gamma,G}(t)})=\sum_{n=0}^\infty 
\mathcal{M}(\Gamma_n) \frac{t^n}{P_{\Gamma,G}(t)}$,
whose coefficients $\frac{t^n}{P_{\Gamma,G}(t)}$ 
are rational functions divisible by $\Delta_{\Gamma,G}$ and have zeros of order $d_x$ 
at the zero loci $x$ of $\Delta_{\Gamma,G}$.
Since the kabi-map $\overline{K}$ (8.4.1) is continuous with respect to the classical topology,
this implies the vanishing 
$\overline{K} \left(\left(\frac{d^i}{dt^i}\frac{P_{\Gamma,G}\mathcal{M}(t)}{P_{\Gamma,G}(t)}
\!\right)\Bigr|_{t=x}\!\right)\!=\!0$ \ for $0\!\le\! i\!<\! d_x$.\ $\Box$

\medskip
Using the all  residues for all roots of $\Delta_{\Gamma,G}(t)\!=\!0$, 
we introduce the {\it global module of 
limit elements for $(\Gamma,G)$}:
\begin{equation}
\label{eq:12.4}
\begin{array}{ll}
\mathcal{L}(\Gamma,G):=\underset{0<r<\infty}{\oplus} \underset{\substack{x: \text{ a root of }\\
 \Delta_{\Gamma,G}(t)=0\ s.t.\ |x|=r}}{\oplus}
\underset{0\le i<d_x}{\oplus}  \C\cdot \left(
\frac{d^i}{dx^i}\frac{P_{\Gamma,G}\mathcal{M}(t)}{P_{\Gamma,G}(t)}
\right)\Bigr|_{t=x}\!,\!\!\!\!\!\!
\end{array}
\end{equation} 
which is doubly filtered: one filtration is given by the absolute values $|x|$ of
the roots of $\Delta_{\Gamma,G}(t)\!=\!0$, and the other by the order $i$ of the depth of 
residues at $x$.

 Theorem 6 in \S11 states  relationships between the 
$\tilde\tau^{h_{\Gamma,G}}$-invariant part of the module $\overline{\R\Omega}(\Gamma,G)$
with the filter  
at $|x|=r_{\Gamma,G}:=\inf\{r\}$ and the first residues part of the module $\mathcal{L}(\Gamma,G)$. 
We ask its generalization.

\medskip
{\bf Problem 2.} What is the relationship between the modules 
$\overline{\R\Omega}(\Gamma,G)$, $\mathcal{L}(\Gamma,G)$ and 
$\mathcal{L}_{\C,\infty}\langle\Gamma,G\rangle$? 
Find  generalization of  Theorems in \S11 and, in particular, of 
\eqref{eq:11.4.3}, \eqref{eq:11.4.4} and \eqref{eq:11.5.6} in this context.

\bigskip
\noindent
{\bf 3.} Another important aspect of the  residues is that the Poincare series $P_{\Gamma,G}\mathcal{M}(t)$ and $P_{\Gamma,G}(t)$ are series, up to variables in $\Conf_0$, defined over integers $\Z$. Therefore, in case $\Delta_{\Gamma,G}(t)$ is a polynomial in $\Z[t]$, they are rational functions defined over $\mathbb{Q}$, and hence the residue (12.2) is defined over the algebraic number field $\mathbb{Q}(x)$ for a root $x$ of $\Delta_{\Gamma,G}(t)\!=\!0$. The action of an element $\sigma$ of the Galois group of the splitting field of $\Delta_{\Gamma,G}(t)$ commutes with the Kabi-map $\overline{K}$ (8.4.1), and seems to bring the space spanned by the residues at $x$ to that at $\sigma(x)$, and hence induces an action on $\mathcal{L}(\Gamma,G)_\Q$.  

Actually, in several interesting examples (surface groups by Cannon [Ca], Artin monoids [Sa5,6]), we observe that the denominator polynomial $\Delta_{\Gamma,G}(t)$ is, up to the factor of a power of $1-t$, irreducible.  In view of the above observation,  the limit space $\C\cdot \Omega(\Gamma,G)$ studied in the present paper is not ``isolated'' but related by the action of the Galois group with the residue modules at other places $x$ with $|x|>r_{\Gamma,G}$. However, no cocrete example is yet known.

On the other hand, the residue module at $t=1$ is  ``isolated'' 
(with respect to the Galois group action). There are a few examples of higher poles at $t=1$ (see [Sa6]), but we do not yet understand  their nature and role.

\medskip
{\bf Example.}(\cite{Sa2}) \ \ 
 Consider the infinite cyclic group $(\Z,{\pm1})$. 
Then, the growth function is given by $P_{\Z,{\pm1}}(t)\!=\!\frac{1+t}{(1-t)^2}$ and the principal part of the singularities of $P_{\Z,{\pm1}}\mathcal{M}(t)$ is given by 
$P_{\Z,{\pm1}}\mathcal{M}(t)=\sum_{m=0}^{\infty}\!\varphi(I_m)\left(\!\frac{2}{(1-t)^2}\!-\!
\frac{m}{1-t}\!+\!R_m\!\right)$ 
where $I_m$ is a linear graph of $m$-vertices 
and $R_m$ is a polynomial 
 of degree {\footnotesize $\!<\![(m\!-\!1)/2]$}
in $t$. 
\vspace{-0.2cm}
Therefore, the two residues of depth 0 and 1 at $t\!=\!1$ are given by
\[
\vspace{-0.1cm}
\begin{array}{rll}
 \frac{P_{\Z,{\pm1}}\mathcal{M}}{P_{\Z,{\pm1}}}\ \ \Bigr|_{t=1}
& =& \quad \sum_{m=0}^\infty\varphi(I_m), \\
\left(\frac{d}{dt}\frac{P_{\Z,{\pm1}}\mathcal{M}}{P_{\Z,{\pm1}}}\right)\Bigr|_{t=1}
& =& \quad \sum_{m=0}^\infty \frac{m-1}{2}\varphi(I_m), 
\end{array}
\]
 span the space 
$\mathcal{L}_{\R,\infty}\langle\Z,\pm1\rangle$, where the first one is the limit element in $\Omega(\Z,\pm1)$.

In view of these observations, we ask the following problems.

\medskip
{\bf Problem 3.} \ {\bf i)}  Describe the action of the Galois group of the splitting field of  $\Delta_{\Gamma,G}(t)=0$ on $\mathcal{L}(\Gamma,G)_{\mathbb{Q}}$.  Clarify the role of the classical part  $\C\cdot \Omega(\Gamma,G)$.

{\bf ii)}  When is the denominator  polynomial $\Delta_{\Gamma,G}(t)$, up to a factor of a power of $1-t$, irreducible over the integers $\Z$?

{\bf iii)} What is the meaning of the residue module at $t=1$:
\begin{equation}
\label{eq:12.5}
\mathcal{L}(\Gamma,G)_1\!:\!=\!
\underset{0\!\le\! i\!<\!d_1}{\oplus}\R\!\cdot\!\left(\!\frac{d^i}{dt^i}\frac{P_{\Gamma,G}\mathcal{M}(t)}{P_{\Gamma,G}(t)}\!\right)\!\Bigr|_{t=1}
\end{equation}
\noindent
{\bf 4.} Including Mach\`i's example, there are a number of examples 
where $\Omega(P_{\Gamma,G})$ is finite. 
However, we do not know an example when 
$\Omega(\Gamma,G)$ is finite
except for the simple accumulating cases (e.g.\ \eqref{eq:11.1.9}).  
We conjecture the following.

\vspace{0.1cm}
{\bf Conjecture 4.}
For any hyperbolic group $\Gamma$ with  any finite generating system 
$G$, the limit space $\Omega(\Gamma,G)$ 
 is finite accumulating.
\vspace{0.1cm}

Evidence is provided by Coornaert [Co]: {\it if $\Gamma$ is hyperbolic, then 
there exists positive real constants $c_1,c_2$ such that 
$c_1r_{\Gamma,G}^{-n}\!\le\!\#\Gamma_n\!\le\! c_2r_{\Gamma,G}^{-n}$.}
This implies that the property in {\bf Fact }in the proof of {\bf Assertion} in (\ref{subsec:11.3}) holds for hyperbolic groups without assuming the finite rational accumulation of $\Omega(P_{\Gamma,G})$. We further expect that  Coornaert's arguments can be lifted to the level of $A(S,\Gamma_n)$.

\bigskip
\noindent
{\bf 5.}
The following groups are not hyperbolic. However, because of their 
geometric significance, it is interesting to ask the following problems.

\medskip
{\bf Problem 5.1} Are the limit spaces $\Omega(\Gamma,G)$ for the  following 
pair of a group and a system of generators 
simple or finite? 

 \ 1.\ Artin groups of finite type with the generating systems 
given in \cite{BS}\cite{Sa4}, 

\ 2.\ The fundamental groups  of the complement of free divisors 
\vspace{-0.08cm}
with respect to the generating system defining positive homogeneous monoid structure 
\cite{S-I}.

\smallskip
In these examples, $G$ generates a positive homogeneous  monoid 
$\Gamma_+$ in $\Gamma$ such that $\Gamma\!=\! \cup_{n=0}^\infty \Delta^{-n} \Gamma_+$, 
where $\Delta$ is a fundamental element.


\medskip
{\bf Problem 5.2} Clarify the relationship among $\Omega(\Gamma,G)$,  
$\Omega(P_{\Gamma,G})$,  
$\Omega(\Gamma_+,G)$ and $\Omega(P_{\Gamma_+,G})$ (see \cite[Chap.13]{Ba} for $\Gamma_+\!=\!(\Z_+)^2$, and \cite{Sa5} for Artin monoids).


\bigskip 

\begin{thebibliography}{Sa1}
\bibitem[Ba]{Ba}
Baxter, R. J.:
Exactly Solved Models in Statistical Mechanics,
Academic Press, 1982.

\bibitem[Bo]{Bo}
Bogopolski, O.V.\: Infinite commensurable hyperbolic groups are bi-Lipschitz equivalent, Algebra and Logic, Vol.\ {\bf 36}, No.\ 3 (1997), 155-163.

\bibitem[BS]{BS}
Brieskorn, E. and \ Saito, K.:  Artin-Gruppen un Coxeter-Gruppen,
Inventiones math. {\bf 17} (1972) 245-271. 

\bibitem[Ca]{Ca}
Cannon, J. : Geometr Dedikata.

\bibitem[C]{C}
Coornaert, M.: Mesures de Patterson-Sullivan sur le bord d'un espace 
hyperbolique au sens de Gromov, Pacific journal of Math. vol.{\bf 159},
No.\ 2 (1993), 241-270.

\bibitem[E1]{E1}
Epstein, D.B.A.:
A personal communication to the author 1990.

\bibitem[E2]{E2}
Epstein, D.B.A. et al:
Word processing in groups, Jones and Bartlett Publishers, Copyright c 1992 by Jones and Bartlett Publishes, Inc. ISBN 0-86720-244-0.

\bibitem[E3]{E3}
Epstein, D.B.A., Iano-Fletcher, A.R. and Zwick, U.: 
Growth function and Automatic groups, J. Experiment. Math. {\bf 5} (1996), 297-315.

\bibitem[Gi]{Gi}
Gibbs, J. W.:
``Elementary Principles in Statistical Mechanics'' 1902.
Reprinted by Dover, New York, 1960.

\bibitem[Gr1]{Gr1}
Gromov, M.:
Groups of polynomial growth and expanding maps, Publ.Math.IHES {\bf 53} (1981), 53-73.

\bibitem[Gr2]{Gr2}
Gromov, M.:
Hyperbolic Groups, Essays in Group theory, edited by
S. M. Gersten, MSRI Publications 8, Springer-Verlag 1987, 75-263.

\bibitem[H]{H}
Hadamard, J.: Th\'eoreme sur les s\'eries entieres, Acta math. {\bf 22} 
(1899), 55-63.

\bibitem[I]{I}
Ising, E.: (1925) Z. Physik {\bf 31}, 253-8.

\bibitem[M]{M}
Milnor, J.: A note on curvature and fundamental group, J.Diff. Geom.\ {\bf 2}(1968)1-7

\bibitem[O]{O}
Onsager, L.: A two-dimensional Model with an Order-Disorder transition (1944) Phys. rev. {\bf 65},117-49.

\bibitem[P]{P}
Pansu, P.: Croissance des boules et des g\'eod\'esiques ferm\'ees dans les nilvari\'et\'es, Ergodic Theory and Dynamic Systems,  {\bf 3} 1983, 415-445.

\bibitem[R]{R}
Rota, G. -C.:
Coalgebras and Bialgebras in Combinatrics, Lectures at the Umbral
Calculus Conference, the university of Oklahoma, May 15-19, 1978.
Notes by S.A.Joni.

\bibitem[Sa1]{Sa1}
Saito, K.:
Moduli space for Fuchsian groups, Algebraic Analysis,
Vol. II, edited by Kashiwara \& Kawai, Academic Press,
Inc. 1988, 735-786.

\bibitem[Sa2]{Sa2}
Saito, K.:
The Limit Element in the Configuration Algebra for a Discrete group I
 (A precis), Proceedings of the International Congress of Mathematicians, Kyoto 1990, Vol.II (Math.Soc.Japan 1991) p931-942,  
\quad 
preprint RIMS-726, (Nov. 1990).

\bibitem[Sa3]{Sa3}
Saito, K.:
Representation Varieties of a Finitely generated group in ${\rm SL}_2$ 
and ${\rm GL}_2$, preprint RIMS-958, (Dec. 1993).

\bibitem[Sa4]{Sa4}
Saito, K.:
Polyhedra Dual to the Weyl Chamber Decomposition: A Pr\'ecis, Publ. RIMS, Kyoto Univ. {\bf 40} (2004), 1337-1384.

\bibitem[Sa5]{Sa5}
Saito, K.:
Growth function associated with Artin monoids of finite type, 
Proc.~Japan Acd., {\bf 84}, Ser.~A (2008), pp179-183.

\bibitem[Sa6]{Sa6}
Saito, K.:
Growth functions for Artin monoids, 
Proc.~Japan Acd., {\bf 85}, Ser.~A (2009), pp84-88.

\bibitem[S-I]{S-I}
Saito, K.~and Ishibe, T.:
Monoid in the fundamental groups of the complement of logarithmic free divisors in $\C^3$, to appear.


\bibitem[Sc]{Sc}
Schwarzc, A.S.: A volume invariant of coverings,  Dokl.Ak.Nauk USSR {\bf 105}(1955)32-34.

\end{thebibliography}
\end{document}